\documentclass{amsart}
\usepackage{amsmath,amssymb,stmaryrd,tikz,xypic}
\usepackage{multirow}
\usepackage{tikz-cd}
\usepackage{pbox}
\usepackage{color}

\newcommand{\ie}{\textit{i.e.}}%Use i.e.~next word.
\newcommand{\ignore}[1]{}
\newcommand{\cat}[1]{\ensuremath{\mathbf{#1}}}
\newcommand{\op}{\ensuremath{\text{op}}}
\newcommand{\id}[1][]{\ensuremath{\mathrm{id}_{#1}}}
\newcommand{\sem}[1]{\ensuremath{\llbracket #1 \rrbracket}}
\newcommand{\defined}{\ensuremath{\mathop{\downarrow}}}
\renewcommand{\dag}{\dagger} %Dagger of dagger category
\newcommand{\oprel}{\ensuremath{\circ}} %Opposite Relation: dagger in Rel

\newcommand{\catC}{\cat{C}}
\newcommand{\catD}{\cat{D}}
\newcommand{\Hilb}{\cat{Hilb}}
\newcommand{\FHilb}{\cat{FHilb}}
\newcommand{\boundedops}{\mathcal{B}}

\newcommand{\coproj}{\kappa}

\DeclareMathOperator{\Rel}{Rel}
\DeclareMathOperator{\RG}{RG}
\DeclareMathOperator{\Cat}{Cat}
\DeclareMathOperator{\Gpd}{Gpd}
\DeclareMathOperator{\Conn}{Conn}
\DeclareMathOperator{\CP}{CP}

\DeclareMathOperator{\dom}{dom}

\DeclareMathOperator{\Dis}{Dis}

\DeclareMathOperator{\BinaryCat}{2Frob}
\DeclareMathOperator{\TernCat}{3Frob}
\DeclareMathOperator{\TernCatNorm}{3Frob_n}%{\text{l}}}

\DeclareMathOperator{\BinSub}{\BinaryCat_{{sub}}}

\newcommand{\T}{\TernCat(\catD)}

\newcommand{\B}{\BinaryCat(\catD)}

\theoremstyle{plain}
\newtheorem{theorem}{Theorem}[section]
\newtheorem{proposition}[theorem]{Proposition}
\newtheorem{lemma}[theorem]{Lemma}
\newtheorem{corollary}[theorem]{Corollary}
\theoremstyle{definition}
\newtheorem{definition}[theorem]{Definition}
\newtheorem{example}[theorem]{Example}

\newtheorem{remark}[theorem]{Remark}

\usetikzlibrary{decorations.pathreplacing,decorations.markings,arrows.meta,backgrounds}
\pgfdeclarelayer{edgelayer}
\pgfdeclarelayer{nodelayer}
\pgfsetlayers{background,edgelayer,nodelayer,main}

\newenvironment{pic}[1][] {\begin{aligned}\begin{tikzpicture}[font=\tiny,#1]}{\end{tikzpicture}\end{aligned}}

\tikzstyle{dot}=[circle, draw=black, fill=black!25, inner sep=.4ex]
\tikzstyle{bigdot}=[dot, inner sep=0pt]
\tikzstyle{whitedot}=[circle, draw=black, fill=white, inner sep=.4ex]
\tikzstyle{blackdot}=[circle, draw=black, fill=black, inner sep=.4ex]
\tikzset{arrow/.style={decoration={
    markings,
    mark=at position #1 with \arrow{>[length=2pt, width=3pt]}},
    postaction=decorate},
    reverse arrow/.style={decoration={
    markings,
    mark=at position #1 with {{\arrow{<[length=2pt, width=3pt]}}}},
    postaction=decorate}
}

\newif\ifvflip\pgfkeys{/tikz/vflip/.is if=vflip}
\newif\ifhflip\pgfkeys{/tikz/hflip/.is if=hflip}
\newif\ifhvflip\pgfkeys{/tikz/hvflip/.is if=hvflip}
\newlength\morphismheight
\newlength\wedgewidth
\tikzset{width/.initial=1mm}
\makeatletter
\tikzstyle{morphism}=[font=\small,morphismshape]
\pgfdeclareshape{morphismshape}
{
    \savedanchor\centerpoint
    {
        \pgf@x=0pt
        \pgf@y=0pt
    }
    \anchor{center}{\centerpoint}
    \anchorborder{\centerpoint}
    \saveddimen\overallwidth
    {
        \pgfkeysgetvalue{/tikz/width}{\minwidth}
        \pgf@x=\wd\pgfnodeparttextbox
        \ifdim\pgf@x<\minwidth
            \pgf@x=\minwidth
        \fi
    }
    \savedanchor{\upperrightcorner}
    {
        \pgf@y=.5\ht\pgfnodeparttextbox
        \advance\pgf@y by -.5\dp\pgfnodeparttextbox
        \pgf@x=.5\wd\pgfnodeparttextbox
    }
    \anchor{north}
    {
        \pgf@x=0pt
        \pgf@y=0.5\morphismheight
    }
    \anchor{north east}
    {
        \pgf@x=\overallwidth
        \multiply \pgf@x by 2
        \divide \pgf@x by 5
        \pgf@y=0.5\morphismheight
    }
    \anchor{east}
    {
        \pgf@x=\overallwidth
        \divide \pgf@x by 2
        \advance \pgf@x by 5pt
        \pgf@y=0pt
    }
    \anchor{west}
    {
        \pgf@x=-\overallwidth
        \divide \pgf@x by 2
        \advance \pgf@x by -5pt
        \pgf@y=0pt
    }
    \anchor{north west}
    {
        \pgf@x=-\overallwidth
        \multiply \pgf@x by 2
        \divide \pgf@x by 5
        \pgf@y=0.5\morphismheight
    }
    \anchor{south east}
    {
        \pgf@x=\overallwidth
        \multiply \pgf@x by 2
        \divide \pgf@x by 5
        \pgf@y=-0.5\morphismheight
    }
    \anchor{south west}
    {
        \pgf@x=-\overallwidth
        \multiply \pgf@x by 2
        \divide \pgf@x by 5
        \pgf@y=-0.5\morphismheight
    }
    \anchor{south}
    {
        \pgf@x=0pt
        \pgf@y=-0.5\morphismheight
    }
    \anchor{text}
    {
        \upperrightcorner
        \pgf@x=-\pgf@x
        \pgf@y=-\pgf@y
    }
    \backgroundpath
    {
    \begin{scope}
        \pgfkeysgetvalue{/tikz/fill}{\morphismfill}
        \pgfsetstrokecolor{black}
        \pgfsetlinewidth{.7pt}
        \begin{scope}
        \pgfsetstrokecolor{black}
        \pgfsetfillcolor{white}
        \pgfsetlinewidth{.7pt}
                \ifhflip
                    \pgftransformyscale{-1}
                \fi
                \ifvflip
                    \pgftransformxscale{-1}
                \fi
                \ifhvflip
                    \pgftransformxscale{-1}
                    \pgftransformyscale{-1}
                \fi
                \pgfpathmoveto{\pgfpoint
                    {-0.5*\overallwidth-5pt}
                    {0.5*\morphismheight}}
                \pgfpathlineto{\pgfpoint
                    {0.5*\overallwidth+5pt}
                    {0.5*\morphismheight}}
                \pgfpathlineto{\pgfpoint
                    {0.5*\overallwidth + \wedgewidth}
                    {-0.5*\morphismheight}}
                \pgfpathlineto{\pgfpoint
                    {-0.5*\overallwidth-5pt}
                    {-0.5*\morphismheight}}
                \pgfpathclose
                \pgfusepath{fill,stroke}
        \end{scope}
    \end{scope}
    }
}
\makeatother

\newcommand\relto[1][]{\mathbin{\smash{
\begin{tikzpicture}[baseline={([yshift=-1pt]
current bounding box.south)}]
    \node (A) at (0,0) [inner xsep=0pt, inner ysep=1pt, minimum width=0.15cm] {\ensuremath{\scriptstyle #1}};
    \draw [->, line width=0.4pt, line cap=round]
        ([xshift=-2.5pt] A.south west)
        to ([xshift=3pt] A.south east);
    \draw [line width=.4pt] ([xshift=-1pt,yshift=-2pt]A.south) to ([xshift=1pt,yshift=2pt]A.south);
\end{tikzpicture}}}}

\newcommand\sxto[1]{\mathbin{\smash{
\begin{tikzpicture}[baseline={([yshift=-1pt]
current bounding box.south)}]
    \node (A) at (0,0) [inner xsep=0pt, inner ysep=1pt, minimum width=0.15cm] {\ensuremath{\scriptstyle #1}};
    \draw [->, line width=0.4pt, line cap=round]
        ([xshift=-2.5pt] A.south west)
        to ([xshift=3pt] A.south east);
\end{tikzpicture}}}}

\newcommand{\tinymult}[1][dot]{
\smash{\raisebox{-2pt}{\hspace{-5pt}\ensuremath{\begin{pic}[scale=0.4,yscale=-1]
    \node (0) at (0,0) {};
    \node[#1, inner sep=1.5pt] (1) at (0,0.55) {};
    \node (2) at (-0.5,1) {};
    \node (3) at (0.5,1) {};
    \draw (0.center) to (1.center);
    \draw (1.center) to [out=left, in=down, out looseness=1.5] (2.center);
    \draw (1.center) to [out=right, in=down, out looseness=1.5] (3.center);
    \node[#1, inner sep=1.5pt] (1) at (0,0.55) {};
\end{pic}
}\hspace{-3pt}}}}

\newcommand{\tinymultflip}[1][dot]{
\smash{\raisebox{-2pt}{\hspace{-5pt}\ensuremath{\begin{pic}[scale=0.4,yscale=1]
    \node (0) at (0,0) {};
    \node[#1, inner sep=1.5pt] (1) at (0,0.55) {};
    \node (2) at (-0.5,1) {};
    \node (3) at (0.5,1) {};
    \draw (0.center) to (1.center);
    \draw (1.center) to [out=left, in=down, out looseness=1.5] (2.center);
    \draw (1.center) to [out=right, in=down, out looseness=1.5] (3.center);
    \node[#1, inner sep=1.5pt] (1) at (0,0.55) {};
\end{pic}
}\hspace{-3pt}}}}

\newcommand{\tinymulttriple}[1][dot]{
\smash{\raisebox{-2pt}{\hspace{-5pt}\ensuremath{\begin{pic}[scale=0.4,yscale=-1]
    \node (0) at (0,0) {};
    \node[#1, inner sep=1.5pt] (1) at (0,0.55) {};
    \node (2) at (-0.5,1) {};
    \node (3) at (0.5,1) {};
    \draw (0.center) to (1.center);
    \draw (1.center) to [out=left, in=down, out looseness=1.5] (2.center);
    \draw (1.center) to [out=right, in=down, out looseness=1.5] (3.center);
        \draw (1.center) to (0,1);
    \node[#1, inner sep=1.5pt] (1) at (0,0.55) {};
\end{pic}
}\hspace{-3pt}}}}

\newcommand{\tinymulttripleflip}[1][dot]{
\smash{\raisebox{-2pt}{\hspace{-5pt}\ensuremath{\begin{pic}[scale=0.4,yscale=1]
    \node (0) at (0,0) {};
    \node[#1, inner sep=1.5pt] (1) at (0,0.55) {};
    \node (2) at (-0.5,1) {};
    \node (3) at (0.5,1) {};
    \draw (0.center) to (1.center);
    \draw (1.center) to [out=left, in=down, out looseness=1.5] (2.center);
    \draw (1.center) to [out=right, in=down, out looseness=1.5] (3.center);
        \draw (1.center) to (0,1);
    \node[#1, inner sep=1.5pt] (1) at (0,0.55) {};
\end{pic}
}\hspace{-3pt}}}}

\newcommand{\tinycomult}[1][dot]{
\smash{\raisebox{-2pt}{\hspace{-5pt}\ensuremath{\begin{pic}[scale=0.4,yscale=1]
    \node (0) at (0,0) {};
    \node[#1, inner sep=1.5pt] (1) at (0,0.55) {};
    \node (2) at (-0.5,1) {};
    \node (3) at (0.5,1) {};
    \draw (0.center) to (1.center);
    \draw (1.center) to [out=left, in=down, out looseness=1.5] (2.center);
    \draw (1.center) to [out=right, in=down, out looseness=1.5] (3.center);
    \node[#1, inner sep=1.5pt] (1) at (0,0.55) {};
\end{pic}
}\hspace{-3pt}}}}

\newcommand{\tinyunit}[1][dot]{
\smash{\raisebox{1pt}{\hspace{-3pt}\ensuremath{\begin{pic}[scale=0.4,yscale=-1]
    \node (0) at (0,0) {};
    \node[#1, inner sep=1.5pt] (1) at (0,0.55) {};
    \draw (0.center) to (1.north);
\end{pic}
}\hspace{-1pt}}}}

\newcommand{\tinyunitdual}[1][dot]{
\smash{\raisebox{1pt}{\hspace{-3pt}\ensuremath{\begin{pic}[scale=0.4,yscale=1]
    \node (0) at (0,0) {};
    \node[#1, inner sep=1.5pt] (1) at (0,-0.55) {};
    \draw[arrow=.9] (0.center) to (1.north);
\end{pic}
}\hspace{-1pt}}}}

\newcommand{\tinycounit}[1][dot]{
\smash{\raisebox{-5pt}{\hspace{-3pt}\ensuremath{\begin{pic}[scale=0.4,yscale=1]
    \node (0) at (0,0) {};
    \node[#1, inner sep=1.5pt] (1) at (0,0.55) {};
    \draw (0.center) to (1.north);
        \node[#1, inner sep=1.5pt] (1) at (0,0.55) {};
\end{pic}
}\hspace{-1pt}}}}

\newcommand{\tinycup}{\smash{\raisebox{-3pt}{\hspace{-2pt}\ensuremath{\begin{pic}[scale=0.2]%,string]
   \pgftransformscale{1.5} \draw[arrow=.6, scale = 1] (0,0) to[out=-90,in=-90,looseness=1.5] (1.5,0);
\end{pic}}}}}

\newcommand{\tinycap}{\smash{\raisebox{-3pt}{\hspace{-2pt}\ensuremath{\begin{pic}[scale=0.2, yscale=-1]%,string]
   \pgftransformscale{1.5} \draw[arrow=.6, scale = 1] (0,0) to[out=-90,in=-90,looseness=1.5] (1.5,0);
\end{pic}}}}}

\newcommand{\tinycupswap}{\smash{\raisebox{-3pt}{\hspace{-2pt}\ensuremath{\begin{pic}[scale=0.2]%,string]
   \pgftransformscale{1.5} \draw[reverse arrow=.6, scale = 1] (0,0) to[out=-90,in=-90,looseness=1.5] (1.5,0);
\end{pic}}}}}

\newcommand{\tinycapswap}{\smash{\raisebox{-3pt}{\hspace{-2pt}\ensuremath{\begin{pic}[scale=0.2, yscale=-1]%,string]
   \pgftransformscale{1.5} \draw[reverse arrow=.6, scale = 1] (0,0) to[out=-90,in=-90,looseness=1.5] (1.5,0);
\end{pic}}}}}

\newcommand{\tinyswap}{
\smash{\raisebox{-2pt}{\hspace{-2pt}\ensuremath{
   \begin{pic}[scale=.25]
  \draw (0,-.5) to[out=80,in=-100] (1,.5);
  \draw (1,-.5) to[out=100,in=-80] (0,.5);
\end{pic}}}}
}

\newcommand{\tinyinvolution}[1][dot]{
\begin{pic}[scale=.3]
  \node[#1] (d) at (0,0) {};
  \draw[arrow=.9] (0,0.7) to (d);
  \draw[arrow=.9] (0,-0.7) to (d);
\end{pic}
}

\newcommand{\tinyinvolutionswap}[1][dot]{
\begin{pic}[scale=.3]
  \node[#1] (d) at (0,0) {};
  \draw[arrow=.9] (d) to (0,0.7);
  \draw[arrow=.9] (d) to (0,-0.7);
\end{pic}
}

\newcommand{\tablemult}{
 \begin{pic}[scale=.3] \node[dot] (r) at (1,0) {}; \draw[arrow=.9] (r) to +(0,-1); \draw[arrow=.9] (r) to +(0,1);
  \draw[reverse arrow=.9] (r) to[out=180,in=90,looseness=.7] +(-.5,-1); \draw[reverse arrow=.95] (r) to[out=0,in=90] +(.5,-1);
   \end{pic}
}

\newcommand{\tablemultdual}{
    \begin{pic}[scale=.3]
      \node[dot] (d) at (0,0){};
      \draw[arrow=.9] (d) to +(0,-1);
      \draw[reverse arrow=.9] (d) to[out=180,in=90] +(-.5,-1);
      \draw[reverse arrow=.9] (d) to[out=0,in=90] +(.5,-.5) to[out=-90,in=-90,looseness=1.5] +(.75,0) to +(0,2);
      \draw[arrow=.97] (d) to[out=90,in=90,looseness=1.5] +(-1,.2) to +(0,-1.2);
    \end{pic}
}

\newcommand{\tablemultupdown}
{
  \begin{pic}[scale=.3]
    \node[dot] (l) at (0,1) {};
    \node[dot] (r) at (1,0) {};
    \draw[arrow=.5] (r) to[out=90,in=0] (l);
    \draw[arrow=.9] (r) to +(0,-1);
    \draw[reverse arrow=.9] (r) to[out=180,in=90,looseness=.7] +(-.5,-1);
    \draw[arrow=.95] (l) to +(0,-2);
    \draw[reverse arrow=.95] (l) to[out=180,in=90,looseness=.7] +(-.75,-2);
    \draw[arrow=.9] (l) to +(0,1);
    \draw[reverse arrow=.95] (r) to[out=0,in=90] +(.5,-.5) to[out=-90,in=-90] +(.5,0) to +(0,2.5);
  \end{pic}
}

\newcommand{\tablemultdownup}
{
  \begin{pic}[scale=.3]
    \node[dot] (l) at (0,0) {};
    \node[dot] (r) at (1,1) {};
    \draw[arrow=.5] (l) to[out=90,in=180] (r);
    \draw[arrow=.9] (l) to +(0,-1);
    \draw[reverse arrow=.9] (l) to[out=0,in=90,looseness=.7] +(.5,-1);
    \draw[arrow=.95] (r) to +(0,-2);
    \draw[reverse arrow=.95] (r) to[out=0,in=90,looseness=.7] +(.75,-2);
    \draw[arrow=.9] (r) to +(0,1);
    \draw[reverse arrow=.95] (l) to[out=180,in=90] +(-.5,-.5) to[out=-90,in=-90] +(-.5,0) to +(0,2.5);
  \end{pic}
}

\newcommand{\tableridem}
{
  \begin{pic}[scale=.3]
    \node[dot] (r) at (1,0) {};
    \draw[arrow=.9] (r) to +(0,-1);
    \draw[arrow=.9] (r) to +(0,1);
    \draw[reverse arrow=.9] (r) to[out=180,in=90,looseness=.7] +(-.5,-1);
    \draw[reverse arrow=.95] (r) to[out=0,in=90] +(.5,-.5) to[out=-90,in=-90] +(.5,0) to +(0,1.5);
  \end{pic}
}

\newcommand{\tablelidem}
{
\begin{pic}[scale=.3]
    \node[dot] (r) at (1,0) {};
    \draw[arrow=.9] (r) to +(0,-1);
    \draw[arrow=.9] (r) to +(0,1);
    \draw[reverse arrow=.9] (r) to[out=0,in=90,looseness=.7] +(.5,-1);
    \draw[reverse arrow=.95] (r) to[out=180,in=90] +(-.5,-.5) to[out=-90,in=-90] +(-.5,0) to +(0,1.5);
  \end{pic}
}

\newcommand{\tinylidem}
{
\begin{pic}[scale=.25]
    \node[dot] (r) at (1,0) {};
    \draw[arrow=.9] (r) to +(0,-1);
    \draw[arrow=.9] (r) to +(0,1);
    \draw[reverse arrow=.9] (r) to[out=0,in=90,looseness=.7] +(.5,-1);
    \draw[reverse arrow=.95] (r) to[out=180,in=90] +(-.5,-.5) to[out=-90,in=-90] +(-.5,0) to +(0,1.5);
  \end{pic}
}

\newcommand{\tinyridem}
{
  \begin{pic}[scale=.25]
    \node[dot] (r) at (1,0) {};
    \draw[arrow=.9] (r) to +(0,-1);
    \draw[arrow=.9] (r) to +(0,1);
    \draw[reverse arrow=.9] (r) to[out=180,in=90,looseness=.7] +(-.5,-1);
    \draw[reverse arrow=.95] (r) to[out=0,in=90] +(.5,-.5) to[out=-90,in=-90] +(.5,0) to +(0,1.5);
  \end{pic}
}

\title{Monoidal characterisation of groupoids and connectors}
\author{Marino Gran}
\email{marino.gran@uclouvain.be}
\address{Universit{\'e} Catholique de Louvain, Institut de Recherche en Math{\'e}matique et Physique, Chemin du Cyclotron 2, 1348 Louvain-la-Neuve, Belgium}
\author{Chris Heunen}
\email{chris.heunen@ed.ac.uk}
\address{University of Edinburgh, School of Informatics, \\10 Crichton Street, Edinburgh EH8 9AB, United Kingdom}
\thanks{Supported by EPSRC Fellowship EP/L002388/1}
\author{Sean Tull}
\email{sean.tull@ed.ac.uk}
\address{University of Oxford, Department of Computer Science, \\Wolfson Building, Parks Road, Oxford OX1 3QD, United Kingdom}
\thanks{Supported by EPSRC Studentship OUCL/2014/SET}
\thanks{This work came about during visits made possible by funding from the Institut de Recherche en Math{\'e}matique et Physique of the Universit{\'e} Catholique de Louvain}

\begin{document}
\begin{abstract}
  We study internal structures in regular categories using monoidal methods. Groupoids in a regular Goursat category can equivalently be described as special dagger Frobenius monoids in its monoidal category of relations. Similarly, connectors can equivalently be described as Frobenius structures with a ternary multiplication. We study such ternary Frobenius structures and the relationship to binary ones, generalising that between connectors and groupoids.
\end{abstract}
\maketitle

\section{Introduction}

Algebraic structures internal to categories are useful in many situations. 
To name but a few fundamental ones: internal groups and groupoids in algebra and algebraic topology, crossed modules in homotopy theory and non-abelian (co)homology, and categories themselves in higher category theory.
In categorical algebra such internal structures are traditionally studied by assuming suitable exactness properties on the category they live in. For example, the category is often assumed to be regular, Mal'tsev, or semi-abelian. 

Instead of exactness properties one can also assume monoidal structure on an ambient category, and speak of internal monoids or groups. By adopting this approach the algebraic calculations can be rigorously replaced by graphical manipulations~\cite{selinger:graphicallanguages}. Groupoids in a regular category $\cat{C}$ can equivalently be described as special dagger Frobenius monoids in the monoidal category $\Rel(\cat{C})$ of relations over $\cat{C}$~\cite{heunencontrerascattaneo:groupoids,heunentull:regular}. 
These Frobenius structures (recalled in Section~\ref{sec:groupoids}) were inspired by quantum theory, and this correspondence was proved functorial for regular Mal'tsev categories $\cat{C}$.

In this article we extend this correspondence from regular Mal'tsev categories to the more general regular Goursat categories~\cite{carbonikellypedicchio:goursat}. Mal'tsev categories are those satisfying $2$-permutability, meaning that $R \circ S = S \circ R$ for any pair of equivalence relations $R$ and $S$ on the same object. Goursat categories (recalled in Section~\ref{sec:groupoids}) are only $3$-permutable, meaning that $R \circ S \circ R = S \circ R \circ S$. These form a large class of categories $\cat{C}$ whose category of internal groupoids $\Gpd(\cat{C})$ is regular (as far as we know being the largest class with this property). This ensures that $\Rel(\Gpd(\cat{C}))$ is well-defined, an important fact that fails when $\cat{C}$ is the category $\cat{Set}$ of sets and functions (simply because the ordinary category $\Gpd(\cat{Set})$ of groupoids is not regular). We prove that $\Rel(\Gpd(\cat{C}))$ is equivalent to a category of special dagger Frobenius structures in $\Rel(\cat{C})$ (in Section~\ref{sec:groupoids}).

We then extend these results from groupoids to \emph{connectors}~\cite{bourngran:centrality} (also known as \emph{pregroupoids}~\cite{kock1988generalized}). Internal groupoids were originally studied in differential geometry~\cite{ehresmann:connexions}, homotopy theory~\cite{loday:spaces}, and later also in categorical algebra because of their deep relation to commutators~\cite{janelidzepedicchio:internal}. It has since been realised that commutator theory in Mal'tsev categories can be entirely based on the properties of connectors. Indeed, in any Mal'tsev category there is a unique connector between two equivalence relations $R$ and $S$ precisely when their commutator $[R,S]$ is trivial~\cite{janelidzepedicchio:pseudogroupoids,bourngran:centrality}. Here we show that connectors in $\cat{C}$ can also be described using monoidal methods, namely as normal dagger Frobenius 3-structures in $\Rel(\cat{C})$ (see Section~\ref{sec:connectors}). Whereas Frobenius (2-)structures have a binary multiplication, Frobenius 3-structures are defined by a ternary multiplication. 

Frobenius structures have natural models in settings other than $\Rel(\cat{C})$.  Frobenius 2-structures in the category of Hilbert spaces are (finite-dimensional) C*-algebras~\cite{vicary:quantumalgebras}, and Frobenius 3-structures include Hilbert triple systems and ternary rings of operators in (finite-dimensional) Hilbert spaces, which are studied in geometry and operator algebra~\cite{kaur:ternary,zalar:triplesystems}. We develop some of the theory of abstract Frobenius 3-structures, including a normal form theorem (in Section~\ref{sec:frobenius3}). 
Finally, we study the relationship between Frobenius 2-structures and Frobenius 3-structures in arbitrary monoidal categories (in Section~\ref{sec:2vs3}), generalising that between groupoids and connectors.

Given our generalisation of Frobenius 2-structures to 3-structures, we leave open the natural question of whether there is a useful notion of Frobenius $n$-structure for general $n$. For example, while 2-structures and 3-structures correspond to groupoids and connectors, one might expect Frobenius 4-structures to relate to \emph{pseudogroupoids}~\cite{janelidzepedicchio:pseudogroupoids}. %Similarly, whereas Frobenius 2-structures in a category correspond to monads on the category~\cite{heunenkarvonen:daggermonads}, a natural question is to externalise Frobenius 3-structures similarly.

\noindent {\em Acknowledgement}. The authors are grateful to the referee for her/his useful comments on a preliminary version of this article.

\section{Monoidal categories of relations}\label{sec:preliminaries}

In this section the notion of regular category and its internal regular logic are briefly recalled. Given a regular category $\cat{C}$ we shall be interested in the construction of the category $\Rel(\cat{C})$ of relations in $\cat{C}$, extending the usual passage from the category $\cat{Set}$ of sets to the category $\Rel$ of relations.

\subsection*{Regular categories}

Recall that an arrow $f \colon A \rightarrow B$ in a category $\cat{C}$ is a \emph{regular epimorphism}  if it is the coequaliser of two arrows in $\cat{C}$. In the category of sets regular epimorphisms are simply surjective maps; more generally, in any algebraic variety (in the sense of universal algebra) regular epimorphisms are the same thing as surjective homomorphisms.

A finitely complete category $\cat{C}$ is \emph{regular} when any map factorises as a regular epimorphism followed by a monomorphism, and regular epimorphisms are pullback-stable, \ie~in a pullback square
\[\begin{tikzpicture}[xscale=3,yscale=1.5]
  \node (tl) at (0,1) {$E \times_B A$};
  \node (tr) at (1,1) {$A$};
  \node (bl) at (0,0) {$E$};
  \node (br) at (1,0) {$B$};
  \draw[->] (tl) to node[above]{$p_2$} (tr);
  \draw[->>] (tr) to node[right]{$g$} (br);
  \draw[->>] (tl) to node[left]{$p_1$} (bl);
  \draw[->] (bl) to node[below]{$p$} (br);
\end{tikzpicture}
\]
the arrow $p_1$ is a regular epimorphism whenever $g$ is a regular epimorphism. 
It is easy to see that the (regular epimorphism, mononomorphism) factorisation of an arrow $f$
\[\begin{tikzpicture}[xscale=1.5]
  \node (l) at (0,0) {$A$};
  \node (r) at (2,0) {$B$};
  \node (m) at (1,-1) {$I$};
  \draw[->] (l) to node[above]{$f$} (r);
  \draw[->>] (l) to node[below]{$p$} (m);
  \draw[>->] (m) to node[below]{$i$} (r);
\end{tikzpicture}\]
is unique, up to isomorphism. The subobject $i \colon I \rightarrow B$ is called the (regular) \emph{image} of the arrow $f$. In the category of sets the map $i \colon I \rightarrow B$ is indeed the inclusion in $B$ of the image $I = \{ f(a) \, \mid a \in A \}$ of $f$.

Examples of regular categories abound in mathematics: any elementary topos (and its dual category), such as the category $\cat{Set}$ of sets (and its dual $\cat{Set}^{op}$); any abelian category, such as the category $\cat{Mod}_R$ of modules over a ring $R$; any algebraic variety (in the sense of universal algebra) such as the categories $\cat{Gp}$ of groups,  $\cat{Mon}$ of monoids, $\cat{Rng}$ of rings, or $\cat{Vect}_k$ of vector spaces on a field $k$; any category monadic over $\cat{Set}$, such as the category of compact Hausdorff spaces, or the category of C$^*$-algebras. If $\cat{D}$ is a regular category, any functor category $[\cat{C},\cat{D}]$ is regular.

\subsection*{Categories of relations}

Given a regular category $\cat{C}$ the objects of the category $\Rel(\cat{C})$ of relations in $\cat{C}$ are the same as the objects of $\cat{C}$, while a morphism from an object $A$ to an object $B$ is simply a \emph{relation} $(r_1, r_2) \colon R \rightarrowtail A \times B$ from $A$ to $B$, \ie~a \emph{subobject} of $A \times B$, represented by this monomorphism. We often denote morphisms in $\Rel(\catC)$ simply as $R \colon A \relto B$.
The composite of a relation $(r_1, r_2) \colon R \rightarrowtail A \times B$ and a relation $(s_1, s_2) \colon S \rightarrowtail B \times C$ in $\Rel(\cat{C})$ is the relation $S \circ R \rightarrowtail A \times C$ obtained as the image in the (regular epimorphism, monomorphism) factorisation of the canonical morphism $R \times_B S \to A \times C$, giving rise to the diagram
\begin{equation}\label{eq:pullbackcomposition}
\begin{pic}[font=\small,xscale=2]
  \node (x) at (-2,0) {$A$};
  \node (y) at (0,0) {$B$};
  \node (z) at (2,0) {$C$};
  \node (r) at (-1,1) {$R$};
  \node (s) at (1,1) {$S$};
  \node (p) at (0,2) {$R \times_B S$};
  \draw[->] (r) to node[left=1.5mm]{$r_1$} (x);
  \draw[->] (r) to node[right=1mm]{$r_2$} (y);
  \draw[->] (s) to node[left=1.5mm]{$s_1$} (y);
  \draw[->] (s) to node[right=1mm]{$s_2$} (z);
  \draw[->] (p) to (r);
  \draw[->] (p) to (s);
 \node (sr) at (0,1) {$S \circ R$};
 \draw[->>, dotted] (p) to (sr);
 \draw[->,dotted] (sr) to (x);%node[above=-3mm]{$(S \circ R)_1$} (x);
 \draw[->,dotted] (sr) to (z);%node[above=-3mm]{$(S \circ R)_2$} (z);
\end{pic}\end{equation}
where $R \times_B S$ is the pullback of $r_2$ and $s_1$. As observed above this image is uniquely defined (up to isomorphism). Since in $\cat{C}$ pullbacks of regular epimorphisms are regular epimorphisms this composition is associative, giving a well-defined category $\cat{\Rel (\cat{C})}$, where the identity on an object $A$ is given by the discrete relation $\Delta = (\id[A], \id[A])\colon A \rightarrow A \times A$.

In $\cat{Set}$, we can describe~\eqref{eq:pullbackcomposition} using the formula
\begin{equation}\label{eq:regularcomposition}
  S \circ R = \{ (a,c) \in A \times C \mid (\exists b \in B)\; R(a,b) \wedge S(b,c) \}
\end{equation}
and the category ${ \Rel(\cat{Set})}$ becomes the usual category $\Rel$ of relations. 

The description here above of the composite $ S \circ R$ makes sense in any regular category $\cat{C}$, via its \emph{regular logic}: this is the fragment of first order logic whose formulae use only the connectives $\exists$ and $\wedge$, and equality. Any regular formula $\phi$ whose function symbols are morphisms in $\cat{C}$ and whose relation symbols are subobjects in $\cat{C}$ inductively defines a subobject $$\sem{ (a_1,\ldots,a_n) \in A_1 \times \ldots \times A_n \mid \phi \ } \rightarrowtail A_1 \times \cdots \times A_n$$ as follows. Equality $\sem{a \in A \mid f(a)=g(a)}$ is interpreted as the equaliser of parallel arrows $f,g \colon A \to B$. Conjunction $\sem{a \in A \mid R(a) \wedge S(a)}$ is interpreted as the pullback of $R,S \rightarrowtail A$. Existential quantification $\sem{a \in A \mid (\exists b \in B)\, R(a,b)}$ is interpreted as the regular image %$\exists_A(R)$ 
of the composite $\smash{\xymatrix{ R \quad \ar@{>->}[r] & A \times B \ar[r]^-{\pi_1}& A}}$.
Whenever one can derive an implication $\phi \Rightarrow \psi$ in regular logic, it follows that $\sem{\phi} \leq \sem{\psi}$ as subobjects. This allows us to state and prove (regular) theorems as if reasoning in $\Rel(\cat{Set})$. 

For example, as in $\cat{Set}$, a relation $R \colon A \relto A$ is called \emph{reflexive} when $R(a,a)$ holds $\forall a \in A$, \emph{symmetric} when $R(a,b) \iff R(b,a)$, and \emph{transitive} when $R(a,b) \wedge R(b,c) \Rightarrow R(a,c)$, equivalently $R \circ R \leq R$. A symmetric, reflexive and transitive relation is called an \emph{equivalence relation}.

\subsection*{Compact dagger categories}

Categories of relations $\Rel(\cat{C})$ in a regular category $\cat{C}$ are automatically monoidal, and satisfy more properties.

\begin{definition}
  A \emph{(symmetric) monoidal dagger category} is a (symmetric) mon\-oidal category $(\cat{D}, \otimes, I, \alpha, \rho, \lambda)$ equipped with a functor $(-)^\dag \colon \cat{D}^\op \to \cat{D}$ satisfying $A^\dag=A$ on objects and $f^{\dag\dag}=f$ on morphisms, such that $(f \otimes g)^\dag = f^\dag \otimes g^\dag$ and the coherence isomorphisms $\alpha$ (expressing associativity), $\lambda$ and $\rho$ (expressing the unit axioms), satisfy $\alpha^{-1}=\alpha^\dag$, $\lambda^{-1}=\lambda^\dag$, $\rho^{-1}=\rho^\dag$, and (in the symmetric case) $\sigma^{-1}=\sigma^\dag$.

  A \emph{left dual} for an object $A$ in a monoidal category is an object $A^*$ together with morphisms $\eta \colon I \to A^* \otimes A$ and $\varepsilon \colon A \otimes A^* \to I$ satisfying $\id[A] = (\varepsilon \otimes \id[A]) \circ (\id[A] \otimes \eta)$ and $\id[A^*] = (\id[A^*] \otimes \varepsilon) \circ (\eta \otimes \id[A^*])$. 
  A \emph{(two-sided) dual} additionally comes with morphisms $\eta' \colon I \to A \otimes A^*$ and $\varepsilon' \colon A^* \otimes A \to I$ making $A$ a left dual for $A^*$. 
 In a monoidal dagger category, it is a \emph{dagger dual} when additionally $\eta'=\varepsilon^{\dag}$ and $\varepsilon'=\eta^{\dag}$. In a symmetric monoidal category, a left dual is \emph{symmetric} when $\varepsilon = \eta^\dag \circ \sigma$.
  A \emph{compact dagger category} $\cat{D}$ is a symmetric monoidal dagger category in which every object has a symmetric dagger dual.
\end{definition}

Whenever objects $A$ and $B$ have (left) duals, morphisms $f \colon A \to B$ are in bijection with morphisms $f^* \colon B^*  \to A^*$, where 
\[
f^* = (\id[A^*] \otimes \varepsilon_B) \circ  (\id[A^*] \otimes f \otimes \id[B^*]) \circ (\eta_A \otimes \id[B^*])\text.
\] 
Similarly, such morphisms are in bijection with morphisms $I \to A^* \otimes B$. 
In the case of dagger duals, we define $f_* = (f^*)^{\dag}$.

%Our main example is the following. 

\begin{example} 
  For any regular category $\cat{C}$, the category $\Rel(\catC)$ is a compact dagger category.
 % Objects are the same as in $\catC$. 
 % Morphisms $R \colon A \relto B$ in $\Rel(\cat{C})$ are subobjects $R \rightarrowtail A \times B$ in $\cat{C}$. 
  %Identities are diagonals $A \to A \times A$. 
  %Composition is 
  %\[
  %  S \circ R = \sem{(a,c) \in A \times C \mid (\exists b \in B)\, R(a,b) \wedge S(b,c)}\text{.}
  %\] 
  The monoidal product in $\Rel(\cat{C})$ is provided by the product of $\cat{C}$, with $I=1$.
  The dagger is denoted $(-)^\oprel$ and given by $$\big(R \rightarrowtail A \times B \big)^\oprel = \big(R \rightarrowtail A \times B \simeq B \times A\big).$$
  Finally, symmetric dagger dual objects are $A^* = A$ with $\eta = \sem{(a,a) \mid a \in A} \rightarrowtail A \times A$.
\end{example}

\begin{example} \label{ex:hilb_fhilb}
The category $\Hilb$ of (complex) Hilbert spaces and continuous linear maps is a symmetric monoidal dagger category. The dagger is given by the adjoint of a linear map, and the monoidal product by the Hilbert space tensor product. The objects with duals in $\Hilb$ are precisely the finite-dimensional Hilbert spaces, and these form a full compact dagger subcategory $\FHilb$.
\end{example}

\subsection*{Graphical Calculus}

Monoidal dagger categories come with a \emph{graphical calculus}, given as follows. Morphisms $f \colon A \to B$ are drawn as
$\setlength\morphismheight{3mm}\begin{pic}
  \node[morphism,font=\tiny] (f) at (0,0) {$f$};
  \draw (f.south) to +(0,-.1)node[below] {$A$};
  \draw (f.north) to +(0,.1) node[above] {$B$};
\end{pic}$, with:%; composition, tensor product, and dagger, as:
\[
  \begin{pic}
    \node[morphism] (f) {$g \circ f$};
    \draw (f.south) to +(0,-.55) node[below] {$A$};
    \draw (f.north) to +(0,.55) node[above] {$C$};
  \end{pic}
  = 
  \begin{pic}
    \node[morphism] (g) at (0,.75) {$g\vphantom{f}$};
    \node[morphism] (f) at (0,0) {$f$};
    \draw (f.south) to +(0,-.2) node[below] {$A$};
    \draw (g.south) to node[right] {$B$} (f.north);
    \draw (g.north) to +(0,.2) node[above] {$C$};
  \end{pic}
  \qquad\qquad
  \begin{pic}
    \node[morphism] (f) {$f \otimes g$};
    \draw (f.south) to +(0,-.55) node[below] {$A \otimes C$};
    \draw (f.north) to +(0,.55) node[above] {$B \otimes D$};
  \end{pic}
  = 
  \begin{pic}
    \node[morphism] (f) at (-.4,0) {$f$};
    \node[morphism] (g) at (.4,0) {$g\vphantom{f}$};
    \draw (f.south) to +(0,-.55) node[below] {$A$};
    \draw (f.north) to +(0,.55) node[above] {$B$};
    \draw (g.south) to +(0,-.55) node[below] {$C$};
    \draw (g.north) to +(0,.55) node[above] {$D$};
  \end{pic}
  \qquad\qquad
  \begin{pic}
    \node[morphism] (f) {$f^\dag$};
    \draw (f.south) to +(0,-.55) node[below] {$B$};
    \draw (f.north) to +(0,.55) node[above] {$A$};
  \end{pic}
  =
  \begin{pic}
    \node[morphism,hflip] (f) {$f$};
    \draw (f.south) to +(0,-.55) node[below] {$B$};
    \draw (f.north) to +(0,.55) node[above] {$A$};
  \end{pic}
\]
The identity $A \to A$ is just the line, $\begin{aligned}\begin{pic}
  \draw (0,-.15) to (0,.15);
\end{pic}\end{aligned}$\;; the (identity on) the monoidal unit object $I$ is the empty picture, 
the swap map $\sigma$ becomes $\begin{aligned}\begin{pic}[scale=.25]
  \draw (0,-.5) to[out=80,in=-100] (1,.5);
  \draw (1,-.5) to[out=100,in=-80] (0,.5);
\end{pic}\end{aligned}$. 
To indicate whether a wire represents an object $A$ or its dual $A^*$, we decorate it with a small arrow, pointing upwards or downwards, respectively.
The canonical morphisms $\eta \colon I \to A^* \otimes A$ and $\varepsilon \colon A \otimes A^* \to I$ are depicted as `cups' and `caps' diagrammatically, so that the defining equations then become:
\[
  \begin{pic}[scale=.5]
    \draw[arrow=.36, arrow=.66] (0,0) to (0,1) to[out=90,in=90,looseness=2] (1,1) to[out=-90,in=-90,looseness=2] (2,1) to (2,2);
  \end{pic}
  =
  \begin{pic}[scale=.5]
    \draw[arrow=.5] (0,0) to (0,2);
  \end{pic}
  \qquad \qquad
  \begin{pic}[scale=.5]
    \draw[reverse arrow=.37, reverse arrow=.67] (0,0) to (0,1) to[out=90,in=90,looseness=2] (-1,1) to[out=-90,in=-90,looseness=2] (-2,1) to (-2,2);
  \end{pic}
  =
  \begin{pic}[scale=.5]
    \draw[reverse arrow=.5] (0,0) to (0,2);
  \end{pic}
\]
% \begin{lemma}\label{lem:names}
%   If $A$ and $A^*$ are dual objects in a monoidal category, there is a one-to-one correspondence between morphisms $A \to B$ and those $I \to A^* \otimes B$.
% \end{lemma}
% \begin{proof}
%   Send $f \colon A \to B$ to $(f \otimes \id[B]) \circ \eta \colon I \to A^* \otimes B$, and send $g \colon I \to A^* \otimes B$ to $(\varepsilon \otimes \id[B]) \circ (f \otimes \eta) \colon A \to B$.
% \end{proof}
%
For many more examples of dagger categories and their theory we refer to~\cite{heunenvicary:cqm}.

\section{Frobenius structures} \label{sec:frobenius2}

The main connection between monoidal methods and categorical algebra this paper describes rephrases internal groupoids in a regular category $\cat{C}$ as certain monoids in $\Rel(\cat{C})$. The precise structure we need is the following.

\begin{definition}\label{def:frobeniusstructure}
  A \emph{Frobenius structure} in a monoidal category consists of a monoid $\mu \colon A \otimes A \to A$ with unit $\eta \colon I \to A$ and a comonoid $\delta \colon A \to A \otimes A$ with counit $\varepsilon \colon A \to I$ on the same object $A$ that satisfy the \emph{Frobenius law} $(\mu \otimes \id[A]) \circ (\id[A] \otimes \delta) = (\id[A] \otimes \mu) \circ (\delta \otimes \id[A])$.
  We will draw the multiplication, unit, comultiplication, and counit, as $\tinymult$, $\tinyunit$, $\tinycomult$, and $\tinycounit$, so that the Frobenius law becomes:
  \begin{align} \label{eq:frobenius:binary}
    \begin{pic}[scale=.6]
      \draw (0,0) to (0,1) to[out=90,in=180] (.5,1.5) to (.5,2);
      \draw (.5,1.5) to[out=0,in=90] (1,1) to[out=-90,in=180] (1.5,.5) to (1.5,0);
      \draw (1.5,.5) to[out=0,in=-90] (2,1) to (2,2);
      \node[dot] at (.5,1.5) {};
      \node[dot] at (1.5,.5) {};
    \end{pic}
    & \quad = \quad
    \begin{pic}[scale = 0.6, xscale=-1]
      \draw (0,0) to (0,1) to[out=90,in=180] (.5,1.5) to (.5,2);
      \draw (.5,1.5) to[out=0,in=90] (1,1) to[out=-90,in=180] (1.5,.5) to (1.5,0);
      \draw (1.5,.5) to[out=0,in=-90] (2,1) to (2,2);
      \node[dot] at (.5,1.5) {};
      \node[dot] at (1.5,.5) {};
    \end{pic} \\
  \intertext{It is \emph{special} when $\mu \circ \delta = \id[A]$:}
    \begin{pic}
      \node[dot] (t) at (0,.6) {};
      \node[dot] (b) at (0,0) {};
      \draw (b.south) to +(0,-0.3);
      \draw (t.north) to +(0,+0.3);
      \draw (b.east) to[out=30,in=-90] +(.15,.3);
      \draw (t.east) to[out=-30,in=90] +(.15,-.3);
      \draw (b.west) to[out=-210,in=-90] +(-.15,.3);
      \draw (t.west) to[out=210,in=90] +(-.15,-.3);
    \end{pic}
    & \quad = \quad
    \begin{pic}
      \draw(0,0) to (0,1.5);
    \end{pic}
    \label{eq:special:binary}
  \end{align}
  A \emph{dagger} Frobenius structure is a Frobenius structure in a monoidal dagger category that additionally satisfies $\delta=\mu^\dag$ and $\varepsilon = \eta^\dag$.
\end{definition}

Because the unit of a monoid is completely determined by the multiplication, we will often write $(A,\tinymult)$ for a dagger Frobenius structure.

\begin{remark}\label{rem:frobeniuslaw}
  Associativity implies that the two maps in the Frobenius law~\eqref{eq:frobenius:binary} also equal $\delta \circ \mu$. For dagger Frobenius structures, conversely, the equation $\delta \circ \mu = (\id \otimes \mu) \circ (\delta \otimes \id)$ implies the Frobenius law~\eqref{eq:frobenius:binary}. See \textit{e.g.}~\cite{heunenvicary:cqm}.
\end{remark}

In the category of finite-dimensional Hilbert spaces and linear maps, special dagger Frobenius structures correspond precisely to finite-dimensional C*-algebras~\cite{vicary:quantumalgebras}. In the category of sets and relations, they correspond precisely to small groupoids~\cite{heunencontrerascattaneo:groupoids}. Section~\ref{sec:groupoids} below will extend this to categories of relations over other regular base categories.

\begin{example} \label{ex:pants:binary}
Dagger dual objects $A$, $A^*$ in a monoidal dagger category
   give $A^* \otimes A$ a dagger Frobenius structure with multiplication 
  $\begin{pic}[yscale=.4,xscale=.3]
    \draw[reverse arrow=.5] (0,0) to[out=90,in=-90] (1,1);
    \draw[reverse arrow=.5] (1,0) to[out=90,in=90] (2,0);
    \draw[arrow=.5] (3,0) to[out=90,in=-90] (2,1);
  \end{pic}$ 
  and unit
  $\begin{pic}[scale=.3]
    \draw[arrow=.55] (0,0) to[out=-90,in=-90,looseness=1.5] (1,0);
  \end{pic}$.
\end{example}

The Frobenius structures in a monoidal dagger category form the objects of a category in various ways. One pertinent choice of morphisms, inspired by physics, is the following.

\begin{definition}\label{def:cp}
  For a monoidal dagger category $\cat{C}$, write $\CP(\cat{C})$ for the category of special dagger Frobenius structures $(A, \tinymult[whitedot])$ in $\cat{C}$ and \emph{completely positive} morphisms $f \colon (A, \tinymult[whitedot]) \to (B, \tinymult[dot])$, that is, morphisms $f \colon A \to B$ in $\cat{C}$ such that
  \begin{equation}\label{eq:cpcondition}
    \begin{pic}[xscale=-1]
     \draw (0,.2) to (0,1) to[out=90,in=180] (.5,1.5) to (.5,1.8);
     \draw (.5,1.5) to[out=0,in=90] (1,1) to[out=-90,in=180] (1.5,.5) to (1.5,.2);
     \draw (1.5,.5) to[out=0,in=-90] (2,1) to (2,1.8);
     \node[morphism] at (1,1) {$f$};
     \node[dot] at (.5,1.5) {};
     \node[whitedot] at (1.5,.5) {};
    \end{pic}
    \quad = \quad
    \begin{pic}
     \node[morphism, hflip] (g) at (0,1) {$g$};
     \node[morphism] (f) at (0,.4) {$g$};
     \draw ([xshift=1mm]f.south east) to +(0,-.3);
     \draw ([xshift=-1mm]f.south west) to +(0,-.3);
     \draw (g.south) to (f.north);
     \draw ([xshift=1mm]g.north east) to +(0,.3);
     \draw ([xshift=-1mm]g.north west) to +(0,.3);
    \end{pic}
  \end{equation}
  for some $g \colon A \otimes B \to X$ in $\cat{C}$. Write $C(f)$ for the left-hand side of~\eqref{eq:cpcondition}.
\end{definition}

This is indeed a well-defined category. In fact, if $\cat{C}$ is a symmetric monoidal dagger category, then so is $\CP(\cat{C})$~\cite{coeckeheunenkissinger:cpstar,heunenvicary:cqm}.
\section{Groupoids}\label{sec:groupoids}

This section makes precise the connection between Frobenius structures in $\Rel(\cat{C})$ and groupoids in $\cat{C}$.

\begin{definition} 
  An \emph{internal category} $C$ in a finitely complete category consists of
  objects $C_0$ (objects) and $C_1$ (morphisms), and morphisms $s$ (source), $t$ (target), $u$ (identity), and $m$ (composition):
  \[\begin{pic}[font=\small,xscale=2]
    \node (0) at (-1,0) {$C_0$};
    \node (1) at (0,0) {$C_1$};
    \node (11) at (1,0) {$C_1 \times_{C_0} C_1$};
    \draw[->] (11) to node[above=-1mm] {$m$} (1);
    \draw[->] (0) to node[above=-1mm] {$u$} (1);
    \draw[->] ([yshift=-3mm]1.west) to node[above=-1mm] {$s$} ([yshift=-3mm]0.east);
    \draw[->] ([yshift=3mm]1.west) to node[above=-1mm] {$t$} ([yshift=3mm]0.east);
    % \draw[->] (1.110) to[out=110,in=180] +(.05,.3) to[out=0,in=45] node[right]{$i$} (1.70);
  \end{pic}\]
  where $C_1 \times_{C_0} C_1$ is the pullback of $s$ and $t$ (composable pairs of morphisms). These morphisms must satisfy familiar equations representing associativity of composition and usual behaviour of identities. 
  An \emph{internal groupoid} additionally has an inversion morphism $i \colon C_1 \to C_1$ satisfying usual axioms. 
\end{definition}
Internal functors between internal groupoids are defined in an obvious manner, and we write $\Gpd(\catC)$ for the category of internal groupoids in $\catC$.

\begin{example} 
 Internal groupoids in $\cat{Set}$ are just (small) groupoids.
 Internal groupoids in the category $\cat{Gp}$ of groups are known to form a category equivalent to the category $\cat{CrMod}$ of \emph{crossed modules}~\cite{BrownSpencer}. More generally, internal groupoids in any congruence modular variety admit a nice description in terms of the universal algebraic commutator of congruences (see \cite{janelidzepedicchio:internal, GranGroupoids}, and references therein). Topological groupoids and Lie groupoids naturally occur in algebraic and differential topology.
 \end{example}

\begin{example}\label{ex:indiscretegroupoid}
  Any object $A$ in a category with finite products induces the \emph{indiscrete groupoid} 
  \[\begin{pic}[font=\small,xscale=2.5]
    \node (0) at (-1,0) {$A$};
    \node (1) at (0,0) {$A \times A$};
    \node (11) at (1,0) {$A \times A \times A$};
    \draw[->] (11) to node[above=-1mm] {$\pi_{13}$} (1);
    \draw[->] (0) to node[above=-1mm] {$\Delta$} (1);
    \draw[->] ([yshift=-3mm]1.west) to node[above=-1mm] {$\pi_1$} ([yshift=-3mm]0.east);
    \draw[->] ([yshift=3mm]1.west) to node[above=-1mm] {$\pi_2$} ([yshift=3mm]0.east);
    \draw[->] (1.110) to[out=110,in=180] +(.05,.3) to[out=0,in=45] node[right=-.5mm]{$\sigma$} (1.70);
  \end{pic}\]
  on $A$, which we may interpret as having a unique morphism from $a$ to $b$ for each pair $(a,b) \in A \times A$. The identities are given by the diagonal $\Delta \colon A \to A \times A$, while the inversion is the swap map $\sigma = (\pi_2, \pi_1) \colon A \times A \to A \times A$.
\end{example}

% We can now prove our first main result, describing groupoids in a regular category -- an internal structure from categorical algebra -- as Frobenius structures in the category of relations on that category -- a structure from monoidal category theory. 

We can now give our first result relating groupoids in a regular category -- an internal structure from categorical algebra -- with Frobenius structures in the category of relations on that category -- a structure from monoidal category theory. The result first appears in \cite{heunentull:regular} with essentially the same proof. This extends the result for groupoids in $\cat{Set}$ from~\cite{heunencontrerascattaneo:groupoids}, which was first shown for abelian groupoids in~\cite{pavlovic2009quantum}.

\begin{theorem}\label{thm:frobeniusstructuresaregroupoids} 
  There is a one-to-one correspondence between groupoids in a regular category $\cat{C}$ and special dagger Frobenius structures in $\Rel(\cat{C})$:
  \begin{itemize}
    \item the composition $A \times_{A_0} A \to A$ of a groupoid in $\catC$ with morphisms $A_1 = A$ corresponds to the multiplication $A \times A \relto A$ of a Frobenius structure in $\Rel(\cat{C})$;
    \item the identity $A_0 \rightarrowtail A$ of the groupoid corresponds to the unit $I \relto A$ of the Frobenius structure in $\Rel(\cat{C})$.
  \end{itemize}
\end{theorem}
\begin{proof}
  We will show that the argument of~\cite{heunencontrerascattaneo:groupoids} can be made to use only regular logic, and therefore holds within any regular category $\cat{C}$. Let $(A, \tinymult)$ be a special dagger Frobenius structure in $\Rel(\cat{C})$. Say $\tinymult = (M \colon A \times A \relto A)$, and $\tinyunit = (U \rightarrowtail A)$. The other equations of Definition~\ref{def:frobeniusstructure} translate to:
  \begin{align}
  (\exists x \in U) \  M(x,a,a') & \iff a = a' \label{units2_monproof} \\ 
  (\exists x \in U) \  M(a,x,a') & \iff a = a' \label{units1_monproof} \\
  (\exists e \in A) \ M(a,b,e) \wedge M(e,c,d) & \iff  (\exists e \in A) \  M(a,e,d) \wedge M(b,c,e) \label{assoc_monproof} \\
  (\exists b,c \in A) \ M(b,c,a) \wedge M(b,c,a') & \iff  a=a' \label{special_monproof} & \\
  (\exists e \in A) \ M(a,e,c) \wedge M(e,d,b) & \iff  (\exists e \in A) \ M(c,e,a) \wedge M(e,b,d) \label{frob_monproof}
  \end{align}   
  It follows from~\eqref{special_monproof} that as a relation $M$ is single-valued and hence takes the form $(A \times A \leftarrowtail B \sxto{m} A)$, for some subobject $B$ of $A \times A$ and morphism $m$, in $\cat{C}$. Write $(a,b)\defined$ for $B(a,b)$, so that  $M(a,b,c)$ means $(a,b)\defined$ and $m(a,b) = c$. Define relations $S \colon A \relto U$, $T \colon A \relto U$, and $I \colon A \relto A$ by 
   \begin{align*}
    S & = \sem{ (a,x) \in A \times U \mid  (a,x)\defined }\text{,} \\
    T & = \sem{ (a,y) \in A \times U \mid  (y,a)\defined }\text{,} \\
    I & = \sem{ (a,b) \in A \times A \mid  (\exists x,y \in U \   M(a,b,x) \wedge M(b,a,y) }\text{.}
  \end{align*}
  It suffices to show these relations are total and single-valued, as they then correspond uniquely to morphisms $s$, $t$, and $i$ in $\cat{C}$ defining the data of a groupoid 
  \[
  \begin{pic}[font=\small,xscale=2.5]
    \node (0) at (-1,0) {$U$};
    \node (1) at (0,0) {$A$};
    \node (11) at (1,0) {$A \times_U A = B$, };
    \draw[->] (11) to node[above=-1mm] {$m$} (1);
    \draw[>->] (0) to (1);
    \draw[->] ([yshift=-1.5mm]1.west) to node[below=-1mm] {$s$} ([yshift=-1.5mm]0.east);
    \draw[->] ([yshift=1.5mm]1.west) to node[above=-1mm] {$t$} ([yshift=1.5mm]0.east);
    \draw[->] (1.110) to[out=110,in=180] +(.05,.3) to[out=0,in=45] node[right=-1mm]{$i$} (1.70);
  \end{pic}
  \]
  where we must also show that $B$ is in fact the pullback of $s$ and $t$. 

  From the unit laws~\eqref{units2_monproof} and \eqref{units1_monproof} it follows that elements of $U$ only compose when they are equal, i.e.~$(\forall x,y \in U) (x,y) \downarrow \implies x = y$.

Now if $(a,x)\defined$ and $(a,y)\defined$, then $(x,y)\defined$ by associativity, and so $x=y$. Therefore $S$ is total and single-valued. An analogous argument holds for $T$.
  Instantiating~\eqref{frob_monproof} with $b=s(a)$, $c=t(a)$ and $d=a$ shows that $I$ is total:
  \[
    (\exists e \in A)  \   M(a,e,s(a)) \wedge M(e,a,t(a))\text{,}
  \]
  that is, `every morphism has an inverse'. Uniqueness of inverses then follows as for any category, once we have shown that the composition $m$ is associative. 
  Writing $a^{-1}$ for any inverse of $a$, associativity~\eqref{assoc_monproof} gives $m(a^{-1},a) = s(a)$, and it follows that $a$ and $b$ are composable whenever $s(a) = t(b)$. 
  Conversely, when $a$ and $b$ are composable, $m(a,b) = m(m(a,s(a)),b) = m(a, m(s(a),b))$ by~\eqref{assoc_monproof} and so $s(a) = t(b)$. Hence $B$ is indeed the pullback of $s$ and $t$. 

  It remains to verify that these morphisms satisfy the equations defining an internal groupoid. Associativity of $m$ only requires further that $(a,b)\defined$ and $(b,c)\defined$ imply $(m(a,b),c)\defined$. From~\eqref{assoc_monproof} we find that $(m(a,b), s(b))\defined$ whenever $a$ and $b$ compose, and hence $s(m(a,b)) = s(b)$ as desired. Finally, that inverses behave as expected follows from the definition of $I$. 

  Thus any dagger special Frobenius structure in $\Rel(\cat{C})$ defines a groupoid in $\cat{C}$. This is the only possible choice of $s$, $t$ and $i$ compatible with $M$ and $U$ since any groupoid operations must satisfy the formulae defining $S$, $T$ and $I$. 

  Conversely, let us check that any groupoid defines such a Frobenius structure. Speciality~\eqref{special_monproof} simply states that the relation $\tinymult$ is single-valued and surjective, which holds since $m(a,s(a)) = a$ for any $a$ in $A$. 
  Equation~\eqref{assoc_monproof} follows from associativity of composition $m$.  
  Unitality~\eqref{units2_monproof} and~\eqref{units1_monproof} follows from the equations satisfied by $u$, $s$ and $t$. 
  Finally, the Frobenius law~\eqref{frob_monproof} simply amounts to the statement that $m(a^{-1},c) = m(b,d^{-1})$ if and only if $m(c^{-1},a) = m(d,b^{-1})$.  
\end{proof}

Our next goal is to make the correspondence of Theorem~\ref{thm:frobeniusstructuresaregroupoids} into an equivalence of categories. For this, we restrict attention to a special class of regular categories.

\begin{definition}\cite{carbonikellypedicchio:goursat} 
  A \emph{Goursat category} is a regular category that satisfies any of the following equivalent conditions:
  \begin{itemize}
    \item $S\circ R \circ S=R\circ S\circ R$ for all equivalence relations $R,S \colon A \relto A$;
    \item every relation $R \colon A \relto B$ satisfies $R^\oprel \circ R \circ R^\oprel \circ R=R^\oprel \circ R$, i.e.: 
    % $R^\oprel \circ R \circ R^\oprel \circ R=R^\oprel \circ R$ for every relation $R \colon A \relto B$;
    % \item every relation $R \colon A \relto B$ satisfies the following property:
      \begin{equation} \label{eq:difunctionalrelation}
        R(a,b), R(c,b), R(c,d), R(e,d) \implies (\exists f) R(a,f), R(e,f)\text{.}
      \end{equation}
    % \item $R^\oprel \circ R \circ R^\oprel \circ R=R^\oprel \circ R$ for every relation $R \colon A \relto B$;
    % \item every relation $R \colon A \relto B$ satisfies the following property:
    %   \begin{equation} \label{eq:difunctionalrelation}
    %     R(a,b), R(c,b), R(c,d), R(e,d) \implies (\exists f) R(a,f), R(e,f)\text{.}
    %   \end{equation}
  \end{itemize}
\end{definition}

The first condition is known as \emph{3-permutability}. As observed in \cite{carbonikellypedicchio:goursat}, most natural examples of Goursat categories are in fact \emph{Mal'tsev} categories, meaning they satisfy the stronger condition of 2-permutability, with $R \circ S = S \circ R$ for equivalence relations $R$, $S$ on the same object.
%A regular category is \emph{Mal'tsev} when any pair of equivalence equivalence relations $R, S$ on the same object satisfy $R \circ S = S \circ R$. Any Mal'tsev category is in particular a Goursat category. 
%There are many interesting examples of Mal'tsev categories in algebra. 
Familiar varieties such as groups, quasi-groups, rings, associative algebras, and Heyting algebras are Mal'tsev categories, as is more generally any variety whose algebraic theory contains a ternary term $p(x,y,z)$ such that $p(x,y,y) = x$ and $p(x,x,y)= y$~\cite{Smith}. In the group case such a term is obtained by defining $p(x,y,z) = x \cdot y^{-1} \cdot z$, where we denote the group operation multiplicatively. Any abelian category is a Mal'tsev category, as is the dual category of any elementary topos, or the category of C$^*$-algebras. On the contrary, neither the category $\cat{Set}$ of sets nor the category $\cat{CHaus}$ of compact Hausdorff spaces are Mal'tsev categories. An example of a Goursat category which is not a Mal'tsev category is provided by the variety of implication algebras. 

Similarly to what happens in the Mal'tsev case, also $3$-permutable algebraic varieties can be characterised in terms of the existence of two ternary operations $p$ and $q$ satisfying the identities $p(x,y,y) = x$, $p(x,x,y) = q(x,y,y)$ and $q(x,x,y)= y$. In recent years some new categorical characterisations of Mal'tsev and Goursat categories have been discovered (see~\cite{granrodelo:universal, granrodelo:beckchevalley, GranRodeloNguefeu}, and references therein).

\begin{lemma}\label{lem:goursat}
  A regular category $\cat{C}$ is a Goursat category if and only if $\Gpd(\cat{C})$ is regular Goursat.
\end{lemma}
\begin{proof}
If $\cat{C}$ is a Goursat category, let us first prove that $\Gpd(\cat{C})$ is a regular category.  As explained in ~\cite{GranRodeloNguefeu}, the factorisation of any functor in $\Gpd(\cat{C})$ as a regular epimorphism followed by a monomorphism is obtained in the category $\Gpd(\cat{C})$ in the same way as in the functor category $\RG(\cat{C})$ of internal reflexive graphs in $\cat{C}$, which is obviously a regular Goursat category. 
Indeed, given an internal functor $(f_0, f_1)$ in $\Gpd (\cat{C})$ from a groupoid $A$ to a groupoid $B$, depicted as
  \begin{equation}\label{functor}
   \begin{aligned}\begin{tikzpicture}[yscale=1.5]
    \node (a2) at (0,2) {$A_1 \times_{A_0} A_1$};
    \node (a1) at (0,1) {$A_1$};
    \node (a0) at (0,0) {$A_0$};
    \node (b2) at (3,2) {$B_1 \times_{B_0} B_1$};
    \node (b1) at (3,1) {$B_1$};
    \node (b0) at (3,0) {${B_0,}$};
    \draw[->, bend right = 45] ([xshift=-2mm]a2.south) to node[left]{$p_1$} ([xshift=-2mm]a1.north);
    \draw[->, bend left = 45] ([xshift=2mm]a2.south) to node[right]{$p_2$} ([xshift=2mm]a1.north);
    \draw[->] (a2.south) to node[left=-1.2mm]{$m$} (a1.north);
    \draw[->] (a0.north) to node[left=-1mm]{$e$} (a1.south);
    \draw[->] ([xshift=-2mm]a1.south) to node[left] {$s$} ([xshift=-2mm]a0.north);
    \draw[->] ([xshift=2mm]a1.south) to node[right=-1mm] {$t$} ([xshift=2mm]a0.north);
    \draw[->] ([yshift=1mm]a1.west) to[out=135,in=90] ([xshift=-3mm]a1.west) node[left]{$i$} to[out=-90,in=-135] ([yshift=-1mm]a1.west);
    \draw[->, bend right = 45] ([xshift=-2mm]b2.south) to node[left]{$p_1$} ([xshift=-2mm]b1.north);
    \draw[->, bend left = 45] ([xshift=2mm]b2.south) to node[right]{$p_2$} ([xshift=2mm]b1.north);
    \draw[->] (b2.south) to node[left=-1.2mm]{$m$} (b1.north);
    \draw[->] (b0.north) to node[left=-1mm]{$e$} (b1.south);
    \draw[->] ([xshift=-2mm]b1.south) to node[left] {$s$} ([xshift=-2mm]b0.north);
    \draw[->] ([xshift=2mm]b1.south) to node[right=-1mm] {$t$} ([xshift=2mm]b0.north);
    \draw[->] ([yshift=1mm]b1.east) to[out=45,in=90] ([xshift=3mm]b1.east) node[right]{$i$} to[out=-90,in=-45] ([yshift=-1mm]b1.east);  
    \draw[->] (a0) to node[below]{$f_0$} (b0);
    \draw[->] ([yshift=-.5mm]a1.east) to node[below]{$f_1$} ([yshift=-.5mm]b1.west);
    \draw[->] (a2) to node[above]{$f_2$} (b2);
  \end{tikzpicture}\end{aligned}
  \end{equation}
 the restriction $f_2 \colon A_1 \times_{A_0} A_1  \rightarrow B_1 \times_{B_0} B_1$ to the ``objects of composable pairs of morphisms'' is a regular epimorphism whenever $f_0$ and $f_1$ are regular epimorphisms (by Theorem $1.3$ (ii) in \cite{granrodelo:beckchevalley}, for instance). This implies that $\Gpd(\cat{C})$ is closed under regular quotients in $\RG(\cat{C})$ (Theorem $3.11$ (ii)  in \cite{GranRodeloNguefeu}), so that the regular image of the factorisation in $\RG(\cat{C})$ of  the internal functor $(f_0, f_1) \colon A \rightarrow B$ is again an internal groupoid in $\cat C$.
 That, in turn, implies that the regular epimorphism-monomorphism factorisations are pullback stable in $\Gpd (\cat{C})$, since pullbacks in $\Gpd (\cat{C})$ are computed ``componentwise'' at the levels of ``objects'', ``morphisms'' and ``composable pairs'', respectively. In other words, the regularity of $\Gpd (\cat{C})$ is inherited by the regularity of the functor category $\RG(\cat{C})$ of reflexive graphs in $\cat{C}$.
 Finally, the category $\Gpd (\cat{C})$ is a Goursat category simply because it is a full subcategory of $\RG(\cat{C})$ that is
 stable under pullbacks and regular quotients in the Goursat category $\RG(\cat{C})$. Indeed, this follows immediately from the fact that a regular category is a Goursat category if and only if the regular image of an equivalence relation is again an equivalence relation \cite{carbonikellypedicchio:goursat}.

  Conversely, assume that $\Gpd (\cat{C})$ is a regular Goursat category. Now, $\cat{C}$ can be identified, via the ``discrete'' functor, with the full replete subcategory $\Dis(\cat{C})$ of $\Gpd (\cat{C})$ whose objects are discrete equivalence relations. Then the fact that $\Dis(\cat{C})$ is closed in $\Gpd (\cat{C})$ under finite limits and subobjects easily implies that the (regular epimorphism, monomorphism) factorisation in $\Gpd (\cat{C})$ of an arrow in $\Dis(\cat{C})$ is also its (regular epimorphism, monomorphism) factorisation in $\Dis(\cat{C})$. It follows that $\cat{C} \simeq \Dis(\cat{C})$ is a Goursat category whenever $\Gpd (\cat{C})$ is a Goursat category, since the regular image of an equivalence relation in $\Dis(\cat{C})$ is then an equivalence relation. 
\end{proof}

Thanks to Lemma~\ref{lem:goursat}, it now makes sense to speak of $\Rel(\Gpd(\cat{C}))$ when $\cat{C}$ is a Goursat category. 

\begin{theorem}\label{thm:relgpdtocprel}
  For any Goursat category $\cat{C}$, there is a functor $$\Rel(\Gpd(\cat{C})) \to \CP(\Rel(\cat{C}))$$ induced by the assignment $A \mapsto A_1$. This functor is an equivalence of monoidal dagger categories.
\end{theorem} 
\begin{proof}
  To see that the functor is well-defined and full, we prove the following: if $(A,\tinymult[whitedot])$ and $(B,\tinymult[dot])$ are groupoids in $\cat{C}$, then a relation $R \colon B \relto A$ defines a subgroupoid of $B \times A$ if and only if it is completely positive. Because $\Rel(\cat{C})(B,A) \simeq \Rel(\cat{C})(I,B^* \otimes A)$ it suffices to consider the case $B=I$. Now let $C(R)$ denote the left-hand side of~\eqref{eq:cpcondition}, then we have 
\begin{equation} \label{eq:CR}
  C(R)= \sem{(a,b) \in A \mid b^{-1} \circ a \in R}
\end{equation}
  First suppose that $R \colon I \to A$ is completely positive, so that $C(R)$ is of the form $S^\oprel \circ S = \sem{(a, c) \in A\times A \mid (\exists b \in C)\, S(a,b) \wedge S(c,b) }$ for some relation $S \colon A \relto C$ in $\cat{C}$. This ensures that $C(R)$ satisfies $$C(R)(a,b) \Rightarrow C(R)(a,a) \wedge C(R)(b,a)$$ from which it follows that $R$ is closed under identities and inverses:
 \[
    R(a) \Rightarrow R(a^{-1}) \wedge R(\id[\dom(a)])\text{.}
 \]
  Since $\cat{C}$ is a Goursat category, also $C(R) \circ C(R) = S^\oprel\circ S\circ S^\oprel\circ S = S^\oprel \circ S = C(R)$, which implies that $R$ is closed under composition, and hence a subgroupoid of $A$. 
  Conversely, if $R$ is a subgroupoid, then from~\eqref{eq:CR} it is easy to check that $C(R)=C(R)^\oprel \circ C(R)$, making $R$ completely positive.   

  The functor is surjective on objects by Theorem~\ref{thm:frobeniusstructuresaregroupoids}.
  To show that the functor is faithful, we need that two subobjects in $\Gpd(\catC)$ are isomorphic there if and only if they are isomorphic in $\catC$. For this it suffices to show that the forgetful functor $\Gpd(\cat{C}) \to \cat{C}$, given by $A \mapsto A_1$, reflects isomorphisms. If $f=(f_0, f_1) \colon A \to B$ is an internal functor in $\Gpd(\cat{C})$ as in \eqref{functor} and  $f_1$ is an isomorphism in $\cat{C}$,
     then both $f_0$ and $f_2$ are isomorphisms. It follows that $f$ an isomorphism in $\Gpd(\cat{C})$.
\end{proof}

\begin{remark}\label{rem:categorification}
  In particular, since we've seen that $\cat{CrMod} \simeq \Gpd(\cat{Gp})$, it follows from the previous theorem that there is an equivalence between the categories  $\CP(\Rel(\cat{Gp}))$ and $\Rel(\cat{CrMod})$.
  Similarly, the category $\CP(\Rel(\cat{Vect}_k))$ is equivalent to the category of relations in 2-vector spaces considered in \cite{baezcrans:lie}. 
  In a Goursat category $\cat{C}$, the forgetful functor $\Gpd(\cat{C}) \to \Cat(\cat{C})$ is an isomorphism, that is, every category in $\cat{C}$ uniquely defines an internal groupoid~\cite{martinsrodelovanderlinden:permutability}.
  Thus the $\CP$ construction is related to the inductive process defining $n$-fold categories.
  More precisely, for a Goursat category $\cat{C}$, define $\Cat^n(\cat{C})$ by $\Cat^0(\cat{C})=\cat{C}$ and $\Cat^{n+1}(\cat{C}) = \Cat(\Cat^n(\cat{C}))$. 
  Then $$\CP(\Rel(\Cat^n(\cat{C}))) \cong \Cat^{n+1}(\cat{C})$$ for all $n \geq 0$.
\end{remark}

\section{Frobenius 3-structures}\label{sec:frobenius3}

In this section we develop a ternary analogue of Frobenius structures. From now on we will call (binary) Frobenius structures \emph{Frobenius 2-structures} or simply \emph{2-structures}.

\begin{definition} \label{def:3-struc}
  A \emph{Frobenius 3-structure} in a monoidal category consists of (two-sided) dual objects $A$ and $A^*$, together with a morphism $\tinymulttriple \colon A \otimes A^* \otimes A \to A$ satisfying \emph{associativity} and \emph{symmetry}:
  \begin{align} 
    \begin{pic}[scale=.4]
        \node[dot] (t) at (0,1) {};
        \node[dot] (b) at (1.5,0) {};
        \draw[arrow=.7] (t) to +(0,1);
        \draw[arrow=.85] (t) to +(0,-2);
        \draw[reverse arrow=.7] (t) to[out=0,in=110] (b);
        \draw[reverse arrow=.9] (t) to[out=180,in=90] (-1.5,-1);
        \draw[reverse arrow=.8] (b) to[out=180,in=90] (0.5,-1);
        \draw[reverse arrow=.8] (b) to[out=0,in=90] (2.5,-1);
        \draw[arrow=.7] (b) to +(0,-1);
    \end{pic}
    & =
    \begin{pic}[yscale=.4,xscale=-.4]
        \node[dot] (t) at (0,1) {};
        \node[dot] (b) at (1.5,0) {};
        \draw[arrow=.7] (t) to +(0,1);
        \draw[arrow=.85] (t) to +(0,-2);
        \draw[reverse arrow=.7] (t) to[out=0,in=110] (b);
        \draw[reverse arrow=.9] (t) to[out=180,in=90] (-1.5,-1);
        \draw[reverse arrow=.8] (b) to[out=180,in=90] (0.5,-1);
        \draw[reverse arrow=.8] (b) to[out=0,in=90] (2.5,-1);
        \draw[arrow=.7] (b) to +(0,-1);
    \end{pic}    
    \label{eq:associativity:ternary} \\
    \begin{pic}[scale=.4]
      \node[dot] (d) at (0,0) {};
      \draw[arrow=.7] (d.north) to +(0,1);
      \draw[reverse arrow=.9] (d) to[out=180,in=90] (-.75,-1.5);
      \draw[arrow=.92] (d) to[out=-90,in=-90,looseness=1.5] ([xshift=2cm]d.south) to ([xshift=2cm,yshift=1cm]d.north);
      \draw[reverse arrow=.9] (d) to[out=0,in=90] ([xshift=5mm]d.south) to[out=-90,in=-90,looseness=1.5] ([xshift=15mm]d.south) to ([xshift=15mm,yshift=1cm]d.north);
    \end{pic}
    & = 
    \begin{pic}[yscale=.4,xscale=-.4]
      \node[dot] (d) at (0,0) {};
      \draw[arrow=.7] (d.north) to +(0,1);
      \draw[reverse arrow=.9] (d.east) to[out=180,in=90] (-.75,-1.5);
      \draw[arrow=.92] (d) to[out=-90,in=-90,looseness=1.5] ([xshift=2cm]d.south) to ([xshift=2cm,yshift=1cm]d.north);
      \draw[reverse arrow=.9] (d) to[out=0,in=90] ([xshift=5mm]d.south) to[out=-90,in=-90,looseness=1.5] ([xshift=15mm]d.south) to ([xshift=15mm,yshift=1cm]d.north);
    \end{pic}
    \label{eq:symmetry:ternary}
  \end{align}
 We call $\tinymulttriple$ the \emph{multiplication} of the 3-structure. Its \emph{comultiplication} $\tinymulttripleflip \colon A \to A \otimes A^* \otimes A$ is the map~\eqref{eq:symmetry:ternary}. In a monoidal dagger category, a \emph{dagger Frobenius 3-structure} is a Frobenius 3-structure for which $A$ and $A^*$ are dagger dual and which satisfies \emph{dagger symmetry}, meaning that~\eqref{eq:symmetry:ternary} holds with $\tinymulttripleflip=\tinymulttriple^\dag$.

  A Frobenius 3-structure is \emph{normal} when its \emph{left loop} and \emph{right loop} 
  \begin{align}\label{eq:loops}
    \begin{pic}[scale=.5]
      \node[dot] (d) at (0,0) {};
      \draw[arrow=.9] (d) to +(0,1);
      \draw[reverse arrow=.9] (d) to +(0,-1);
      \draw[arrow=.5] (d) to[out=-135,in=-90,looseness=2] ([xshift=-5mm]d) to[out=90,in=135,looseness=2] (d);
    \end{pic}
    =
    \begin{pic}[scale=.5]
      \node[dot] (d) at (0,0) {};
      \draw[arrow=.9] (d) to +(0,1);
      \draw[reverse arrow=.9] (d) to[out=0,in=90] +(.75,-1);
      \draw[arrow=.75] (d) to[out=-90,in=0] +(-.5,-.75) to[out=180,in=180,looseness=1.5] (d);
    \end{pic}
    \qquad\qquad
    \begin{pic}[scale=.5]
      \node[dot] (d) at (0,0) {};
      \draw[arrow=.9] (d) to +(0,1);
      \draw[reverse arrow=.9] (d) to +(0,-1);
      \draw[arrow=.5] (d) to[out=-45,in=-90,looseness=2] ([xshift=5mm]d) to[out=90,in=45,looseness=2] (d);
    \end{pic}
    =
    \begin{pic}[yscale=.5,xscale=-.5]
      \node[dot] (d) at (0,0) {};
      \draw[arrow=.9] (d) to +(0,1);
      \draw[reverse arrow=.9] (d) to[out=0,in=90] +(.75,-1);
      \draw[arrow=.75] (d) to[out=-90,in=0] +(-.5,-.75) to[out=180,in=180,looseness=1.5] (d.east);
    \end{pic}
  \end{align}
  are both identities. Note that both loops commute.
  Finally, we will call a morphism $\tinymulttriple$ \emph{left idempotent}, or \emph{right idempotent}, when its canonical endomorphism $l_A = \tinylidem$ on $A^* \otimes A$, or $r_A = \tinyridem$ on $A \otimes A^*$, is idempotent, respectively. 
\end{definition}

At times we will call Frobenius 3-structures simply \emph{3-structures}.

\begin{lemma}\label{lem:coassociative:ternary}
  Frobenius 3-structures satisfy:
  \begin{align} 
    \begin{pic}[xscale=.4,yscale=-.4]
        \node[dot] (t) at (0,1) {};
        \node[dot] (b) at (1.5,0) {};
        \draw[reverse arrow=.7] (t) to +(0,1);
        \draw[reverse arrow=.85] (t) to +(0,-2);
        \draw[arrow=.7] (t) to[out=0,in=110] (b.south);
        \draw[arrow=.9] (t) to[out=180,in=90] (-1.5,-1);
        \draw[arrow=.8] (b) to[out=180,in=90] (0.5,-1);
        \draw[arrow=.8] (b) to[out=0,in=90] (2.5,-1);
        \draw[reverse arrow=.7] (b) to +(0,-1);
    \end{pic}
    & \;\,=
    \begin{pic}[scale=-.4]
        \node[dot] (t) at (0,1) {};
        \node[dot] (b) at (1.5,0) {};
        \draw[reverse arrow=.7] (t) to +(0,1);
        \draw[reverse arrow=.85] (t) to +(0,-2);
        \draw[arrow=.7] (t) to[out=0,in=110] (b.south);
        \draw[arrow=.9] (t) to[out=180,in=90] (-1.5,-1);
        \draw[arrow=.8] (b) to[out=180,in=90] (0.5,-1);
        \draw[arrow=.8] (b) to[out=0,in=90] (2.5,-1);
        \draw[reverse arrow=.7] (b) to +(0,-1);
    \end{pic}
    \label{eq:coassociativity:ternary}  \\
    \begin{pic}[scale=.6]
      \draw[arrow=.2] (0,0) to (0,1) to[out=90,in=180] (.5,1.5) to (.5,2);
      \draw[reverse arrow=.4] (.5,1.5) to[out=0,in=90] (1,1) to[out=-90,in=180] (1.5,.5) to (1.5,0);
      \draw[arrow=.9] (1.5,.5) to[out=0,in=-90] (2,1) to (2,2);
      \draw[arrow=.8] (.5,1.5) to (.5,0);
      \draw[reverse arrow=.85] (1.5,.5) to (1.5,2);
      \node[dot] at (.5,1.5) {};
      \node[dot] at (1.5,.5) {};
    \end{pic}
    \;=\; &
    \begin{pic}[scale=.6]
       \node[dot] (b) at (0,0.5) {};
       \node[dot] (t) at (0, 1.5) {};
       \draw[reverse arrow=.9] (b) to[out=0,in=90] (0.5,0);
       \draw[arrow=.25] (-0.5,0) to[out=90,in=180] (b);
       \draw[arrow=.9] (t) to[out=0,in=-90] (0.5,2);
       \draw[arrow=.9] (t) to[out=180,in=-90] (-0.5,2);
       \draw[arrow=.6] (b.north) to (t.south);
       \draw[arrow=.8](b.south) to (0,0);
       \draw[reverse arrow=.7](t.north) to (0,2);
    \end{pic}
    \;=\;
    \begin{pic}[yscale=.6,xscale=-.6]
      \draw[arrow=.2] (0,0) to (0,1) to[out=90,in=180] (.5,1.5) to (.5,2);
      \draw[reverse arrow=.4] (.5,1.5) to[out=0,in=90] (1,1) to[out=-90,in=180] (1.5,.5) to (1.5,0);
      \draw[arrow=.9] (1.5,.5) to[out=0,in=-90] (2,1) to (2,2);
      \draw[arrow=.8] (.5,1.5) to (.5,0);
      \draw[reverse arrow=.85] (1.5,.5) to (1.5,2);
      \node[dot] at (.5,1.5) {};
      \node[dot] at (1.5,.5) {};
    \end{pic}
    \label{eq:frobenius:ternary}
  \end{align}   
\end{lemma}
\begin{proof}
  For coassociativity~\eqref{eq:coassociativity:ternary}:
  \[
    \begin{pic}[xscale=.3,yscale=-.35]
        \node[dot] (t) at (0,1) {};
        \node[dot] (b) at (1.5,0) {};
        \draw[reverse arrow=.7] (t) to +(0,1);
        \draw[reverse arrow=.85] (t) to +(0,-2);
        \draw[arrow=.7] (t) to[out=0,in=110] (b.south);
        \draw[arrow=.9] (t) to[out=180,in=90] (-1.5,-1);
        \draw[arrow=.8] (b) to[out=180,in=90] (0.5,-1);
        \draw[arrow=.8] (b) to[out=0,in=90] (2.5,-1);
        \draw[reverse arrow=.7] (b) to +(0,-1);
    \end{pic}
    =
    \begin{pic}[scale=.4]
      \node[dot] (l) at (0,0) {};
      \node[dot] (r) at (1,.5) {};
      \draw[arrow=.9] (r) to +(0,.7);
      \draw[arrow=.7] (l) to[out=90,in=180,looseness=.7] (r);
      \draw[reverse arrow=.95] (r) to[out=0,in=90] +(.5,-.3) to[out=-90,in=-90,looseness=2] +(.5,0) to +(0,1);
      \draw[arrow=.97] (r) to[out=-90,in=-90,looseness=2] +(1.5,-.2) to +(0,.9);
      \draw[reverse arrow=.9] (l) to[out=0,in=90,looseness=.8] +(.7,-1.5);
      \draw[reverse arrow=.95] (l) to[out=180,in=90] +(-.5,-.3) to[out=-90,in=-90,looseness=2] +(-.5,0) to +(0,1.5);
      \draw[arrow=.97] (l) to[out=-90,in=-90,looseness=2] +(-1.5,-.2) to +(0,1.4);
    \end{pic}
    =
    \begin{pic}[xscale=-.4,yscale=.4]
      \node[dot] (l) at (0,0) {};
      \node[dot] (r) at (1,.5) {};
      \draw[arrow=.9] (r) to +(0,.7);
      \draw[arrow=.7] (l) to[out=90,in=180,looseness=.7] (r.east);
      \draw[reverse arrow=.95] (r) to[out=0,in=90] +(.5,-.3) to[out=-90,in=-90,looseness=2] +(.5,0) to +(0,1);
      \draw[arrow=.97] (r) to[out=-90,in=-90,looseness=2] +(1.5,-.2) to +(0,.9);
      \draw[reverse arrow=.9] (l.west) to[out=0,in=90] +(.6,-1.5);
      \draw[reverse arrow=.95] (l) to[out=180,in=90] +(-.5,-.3) to[out=-90,in=-90,looseness=2] +(-.5,0) to +(0,1.5);
      \draw[arrow=.97] (l) to[out=-90,in=-90,looseness=2] +(-1.5,-.2) to +(0,1.4);
    \end{pic}
    =
    \begin{pic}[scale=-.4]
        \node[dot] (t) at (0,1) {};
        \node[dot] (b) at (1.5,0) {};
        \draw[reverse arrow=.7] (t) to +(0,1);
        \draw[reverse arrow=.85] (t) to +(0,-2);
        \draw[arrow=.7] (t) to[out=0,in=110] (b.south);
        \draw[arrow=.9] (t) to[out=180,in=90] (-1.5,-1);
        \draw[arrow=.8] (b) to[out=180,in=90] (0.5,-1);
        \draw[arrow=.8] (b) to[out=0,in=90] (2.5,-1);
        \draw[reverse arrow=.7] (b) to +(0,-1);
    \end{pic}
  \]
  We verify the second equation in the Frobenius law~\eqref{eq:frobenius:ternary}: \\
  \[
    \begin{pic}[scale=.6]
       \node[dot] (b) at (0,0.5) {};
       \node[dot] (t) at (0, 1.5) {};
       \draw[reverse arrow=.9] (b) to[out=0,in=90] (0.5,0);
       \draw[arrow=.25] (-0.5,0) to[out=90,in=180] (b);
       \draw[arrow=.9] (t) to[out=0,in=-90] (0.5,2);
       \draw[arrow=.9] (t) to[out=180,in=-90] (-0.5,2);
       \draw[arrow=.6] (b.north) to (t.south);
       \draw[arrow=.8](b.south) to (0,0);
       \draw[reverse arrow=.7](t.north) to (0,2);
    \end{pic}
    % =
    % \begin{pic}[scale=.6]
    %    \node[dot] (b) at (0,0.5) {};
    %    \node[dot] (t) at (0, 1.5) {};
    %    \draw[reverse arrow=.9] (b) to[out=0,in=90] (0.5,0);
    %    \draw[arrow=.25] (-0.5,0) to[out=90,in=180] (b);
    %    \draw[arrow=.9] (t) to[out=0,in=-90] (0.5,2);
    %    \draw[arrow=.6] (b.north) to (t.south);
    %    \draw[arrow=.8](b.south) to (0,0);
    %    \draw[reverse arrow=.7](t.north) to (0,2);
    %    \draw[arrow=.97] (t) to[out=180,in=-90] (-0.3,1.7) to[out=90,in=90,looseness=2] (-.6,1.7) to (-.6,1.2) to[out=-90,in=-90,looseness=2] (-.9,1.2) to (-.9,2); 
    % \end{pic}
    =
    \begin{pic}[scale=.6]
       \node[dot] (b) at (0,0.5) {};
       \node[dot] (t) at (-.5, 1.5) {};
       \draw[reverse arrow=.9] (b) to[out=0,in=90] (0.5,0);
       \draw[arrow=.25] (-0.5,0) to[out=90,in=180] (b);
       \draw[arrow=.6] (b.north) to[out=90,in=0] (t);
       \draw[arrow=.8](b.south) to (0,0);
       \draw[arrow=.9] (t) to +(0,.5);
       \draw[reverse arrow=.95] (t) to[out=180,in=90] +(-.4,-.3) to[out=-90,in=-90,looseness=2] +(-.3,0) to +(0,.8);
       \draw[arrow=.97] (t) to[out=-90,in=-90,looseness=2] +(-1,-.2) to +(0,.7);
    \end{pic}
    =
    \begin{pic}[scale=.6]
        \node[dot] (l) at (-.5,.-.5) {};
        \node[dot] (r) at (.3,.2) {};
        \draw[arrow=.8] (r) to +(0,.5);
        \draw[arrow=.9] (r) to +(0,-1.5);
        \draw[arrow=.9] (r) to[out=0,in=90] +(.7,-1.5);
        \draw[arrow=.5] (l) to[out=90,in=180] (r);
        \draw[reverse arrow=.85] (l) to[out=0,in=90] +(.5,-.8);
        \draw[reverse arrow=.9] (l) to[out=180,in=90] +(-.3,-.2) to[out=-90,in=-90,looseness=2] +(-.4,0) to +(0,1.4);
        \draw[arrow=.94] (l) to[out=-90,in=-90,looseness=2] +(-1,-.2) to +(0,1.4);
    \end{pic}    
    =
    % \begin{pic}[yscale=.6,xscale=-.6]
    %   \draw[arrow=.2] (0,0) to (0,1) to[out=90,in=180] (.5,1.5) to (.5,2);
    %   \draw[reverse arrow=.4] (.5,1.5) to[out=0,in=90] (1,1) to[out=-90,in=180] (1.5,.5) to (1.5,0);
    %   \draw[arrow=.8] (.5,1.5) to (.5,0);
    %   \draw[reverse arrow=.85] (1.5,.5) to (1.5,2);
    %   \draw[arrow=.95] (1.5,.5) to[out=0,in=-90] +(.3,.3) to[out=90,in=90,looseness=2] +(.4,0) to +(0,-.3) to[out=-90,in=-90,looseness=2] +(.4,-.3) to +(0,1.5);
    %   \node[dot] at (.5,1.5) {};
    %   \node[dot] at (1.5,.5) {};
    % \end{pic}    
    % =
    \begin{pic}[yscale=.6,xscale=-.6]
      \draw[arrow=.2] (0,0) to (0,1) to[out=90,in=180] (.5,1.5) to (.5,2);
      \draw[reverse arrow=.4] (.5,1.5) to[out=0,in=90] (1,1) to[out=-90,in=180] (1.5,.5) to (1.5,0);
      \draw[arrow=.9] (1.5,.5) to[out=0,in=-90] (2,1) to (2,2);
      \draw[arrow=.8] (.5,1.5) to (.5,0);
      \draw[reverse arrow=.85] (1.5,.5) to (1.5,2);
      \node[dot] at (.5,1.5) {};
      \node[dot] at (1.5,.5) {};
    \end{pic}
  \]
  The first equality in the Frobenius law~\eqref{eq:frobenius:ternary} is similar.
\end{proof}

\begin{example}[Dual Frobenius 3-structure]\label{ex:dual3}
  If $(A,\tinymulttriple)$ is a Frobenius 3-structure in a monoidal category $\cat{C}$, then any dual object $A^*$ also has a Frobenius 3-structure given by:
  \begin{equation} \label{eq:dual3} 
    \begin{pic}[scale=.5]
      \node[dot] (d) at (0,0){};
      \draw[arrow=.9] (d) to +(0,-1);
      \draw[reverse arrow=.9] (d) to[out=180,in=90] +(-.5,-1);
      \draw[reverse arrow=.9] (d) to[out=0,in=90] +(.5,-.5) to[out=-90,in=-90,looseness=1.5] +(.75,0) to +(0,2);
      \draw[arrow=.97] (d) to[out=90,in=90,looseness=1.5] +(-1,.2) to +(0,-1.2);
    \end{pic}
    =
    \begin{pic}[xscale=-.5,yscale=.5]
      \node[dot] (d) at (0,0){};
      \draw[arrow=.9] (d) to +(0,-1);
      \draw[reverse arrow=.9] (d.east) to[out=180,in=90] +(-.4,-1);
      \draw[reverse arrow=.9] (d.west) to[out=0,in=90] +(.4,-.5) to[out=-90,in=-90,looseness=1.5] +(.75,0) to +(0,2);
      \draw[arrow=.97] (d) to[out=90,in=90,looseness=1.5] +(-1,.2) to +(0,-1.2);
    \end{pic}    
  \end{equation}
\end{example}

\begin{example}[Opposite Frobenius 3-structure]\label{ex:opposite3}
  If an object in a symmetric monoidal category has a Frobenius 3-structure $\tinymulttriple$, then it also has another one given by:
  \begin{equation}\label{eq:opposite3}
    \begin{pic}[xscale=.5,yscale=.3]
      \node[dot] (d) at (0,0) {};
      \draw[arrow=.9] (d) to +(0,1);
      \draw[arrow=.9] (d) to +(0,-2);
      \draw[reverse arrow=.92] (d) to[out=0,in=90] +(.75,-.5) to[out=-90,in=90] +(-1.5,-1.5);
      \draw[reverse arrow=.92] (d) to[out=180,in=90] +(-.75,-.5) to[out=-90,in=90] +(1.5,-1.5);
    \end{pic}
  \end{equation}
\end{example}

\begin{definition}
  A Frobenius 3-structure on $A$ is \emph{commutative} when it equals its opposite, and $A$ and $A^*$ are symmetric duals.
\end{definition}

\begin{example}[Product of Frobenius 3-structures]\label{ex:product3}
  If $(A,\tinymulttriple)$ and $(B,\tinymulttriple[whitedot])$ are Frobenius 3-structures in a symmetric monoidal category, then so is $A \otimes B$, using $(A \otimes B)^* = B^* \otimes A^*$ and multiplication:
  \begin{equation} \label{eq:product3} 
    \begin{pic}[scale=.5]
      \node[whitedot] (l) at (0,0) {};
      \node[dot] (r) at (2,0) {};
      \draw[arrow=.9] (l) to +(0,1);
      \draw[arrow=.9] (r) to +(0,1);
      \draw[reverse arrow=.9] (l) to[out=180,in=90] +(-1.5,-2.5);
      \draw[reverse arrow=.92] (r) to[out=180,in=90] +(-2.5,-2.5);
      \draw[arrow=.91] (l) to[out=-90,in=90] +(1.5,-2.5);
      \draw[arrow=.91] (r) to[out=-90,in=90] +(-1.5,-2.5);
      \draw[reverse arrow=.92] (l) to[out=0,in=90] +(2.5,-2.5);
      \draw[reverse arrow=.9] (r) to[out=0,in=90] +(1.5,-2.5);
    \end{pic}
  \end{equation}
\end{example}

\noindent
If $\tinymulttriple$ and $\tinymulttriple[whitedot]$ are commutative or normal, then so are~\eqref{eq:dual3}, \eqref{eq:opposite3}, and~\eqref{eq:product3}.

% Removed "Pants" here as doesn't seem to need a name.
\begin{example}\label{ex:triplepants}\label{ex:landr}
Any object $A$ with a dual comes with Frobenius 3-structures:
  \[
    \begin{pic}[xscale=.4,yscale=.7]
      \draw[arrow=.5] (0,0) to[out=90,in=-90] (1,1);
      \draw[reverse arrow=.53] (1,0) to[out=90,in=90] (2,0);
    \end{pic}
    \qquad\text{ and }\qquad
    \begin{pic}[xscale=-.4,yscale=.7]
      \draw[arrow=.5] (0,0) to[out=90,in=-90] (1,1);
      \draw[reverse arrow=.5] (1,0) to[out=90,in=90] (2,0);
    \end{pic}
  \]
  If objects $A$ and $B$ in a monoidal category have duals $A^*$ and $B^*$,
  then $A^* \otimes B$ has a Frobenius 3-structure: 
  \begin{equation}\label{eq:tpants}
    \begin{pic}[xscale=.3]
      \draw[reverse arrow=.1] (0,0) node[below]{$A$} to[out=90,in=-90,looseness=.8] (2,1);
      \draw[arrow=.2, arrow=.92] (1,0) node[below]{$B$} to[out=90,in=90] (2,0) node[below]{$B$};
      \draw[arrow=.2, arrow=.92] (3,0) node[below]{$A$} to[out=90,in=90] (4,0) node[below]{$A$};
      \draw[arrow=.1] (5,0) node[below]{$B$} to[out=90,in=-90,looseness=.8] (3,1);
    \end{pic}
  \end{equation}
  In a symmetric monoidal category, this decomposes as 
$
(A^*,
    \begin{pic}[xscale=.2,yscale=.35]
      \draw[reverse arrow=.53] (0,0) to[out=90,in=-90] (1,1);
      \draw[arrow=.6] (1,0) to[out=90,in=90] (2,0);
    \end{pic}
)
\otimes
(B,
    \begin{pic}[xscale=-.2,yscale=.35]
      \draw[arrow=.5] (0,0) to[out=90,in=-90] (1,1);
      \draw[reverse arrow=.5] (1,0) to[out=90,in=90] (2,0);
    \end{pic}
)
$.
\end{example}

\begin{example} \label{ex:HilbertTriple}
By definition, a dagger Frobenius 3-structure in $\Hilb$ is a finite-dimensional Hilbert space $H$ together with a ternary map $[-,-,-]$ on $H$, linear in the first and third arguments and anti-linear in the second, satisfying:
\begin{align*}
    \langle [a, b, c], d \rangle = \langle a, [d, c, b] \rangle = \langle c, [b, a, d] \rangle\text, \\
    [[a, b, c], d, e]= [a, b, [c, d, e]]\text.
  \end{align*}
% These follow from symmetry and the dagger law, and associativity, respectively. 
Such structures are known as finite-dimensional associative \emph{Hilbert triple systems}~\cite{zalar:triplesystems}. The most well-known are the \emph{ternary rings of operators} (TROs); subspaces of $\boundedops(H)$ (or $\boundedops(H, K)$) closed under the norm and $(a,b,c) \mapsto {a \circ b^* \circ c}$. In $\Hilb$ each 3-structure~\eqref{eq:tpants} may be identified with $\boundedops(H,K)$ under this operation. In fact, every associative Hilbert triple system may be written as an orthogonal sum of ternary structures of this form up to an overall sign factor, and those with $[-,-,-]=0$~\cite{zalar:triplesystems}. 

TROs were first studied in finite dimension by Hestenes~\cite{hestenes1962ternary}, and have since been shown to essentially coincide with \emph{Hilbert C$^*$-modules}~\cite{zettl1983characterization}. For further related ternary structures in algebra and geometry, see~\cite{chu2011jordan}. 
\end{example}

We turn to Frobenius 3-structures in $\Rel(\catC)$ shortly, in Section~\ref{sec:connectors}.

\subsection*{Normal forms}

We now establish coherence for 3-structures. This requires one additional property. 

\begin{definition}
A Frobenius 3-structure is said to satisfy \emph{sliding} when:
 
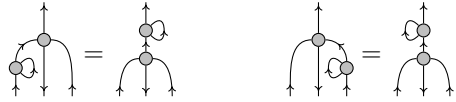
\[
    \begin{pic}[scale=.5]
        \node[dot] (d) at (0,0) {};
        \node[dot] (e) at (-.75,-.75) {};
        \draw[arrow=.9] (d) to +(0,1);
        \draw[arrow=.9] (d) to +(0,-1.5);
        \draw[reverse arrow=.9] (d) to[out=0,in=90] +(.75,-1.5);
        \draw[arrow=.6] (e) to[out=90,in=180] (d);
        \draw[reverse arrow=.7] (e) to +(0,-.75);
        \draw[arrow=.5] (e.-45) to[out=-45,in=-90,looseness=2] ([xshift=3mm]e.east) to[out=90,in=45,looseness=2] (e.45);
    \end{pic}
    = 
    \begin{pic}[scale=.5]
        \node[dot] (d) at (0,0) {};
        \node[dot] (e) at (0,.75) {};
        \draw[arrow=.5] (e.-45) to[out=-45,in=-90,looseness=2] ([xshift=3mm]e.east) to[out=90,in=45,looseness=2] (e.45);  
        \draw[arrow=.7] (d) to (e);
        \draw[arrow=.9] (e) to +(0,.75);
        \draw[arrow=.9] (d) to +(0,-1);
        \draw[reverse arrow=.9] (d) to[out=0,in=90] +(.7,-1);
        \draw[reverse arrow=.9] (d) to[out=180,in=90] +(-.7,-1);
    \end{pic}
    \qquad \qquad
        \begin{pic}[scale=.5]
        \node[dot] (d) at (0,0) {};
        \node[dot] (e) at (.75,-.75) {};
        \draw[arrow=.9] (d) to +(0,1);
        \draw[arrow=.9] (d) to +(0,-1.5);
        \draw[reverse arrow=.9] (d) to[out=180,in=90] +(-.75,-1.5);
        \draw[arrow=.6] (e) to[out=90,in=0] (d);
        \draw[reverse arrow=.7] (e) to +(0,-.75);
        \draw[arrow=.5] (e.-135) to[out=-135,in=-90,looseness=2] ([xshift=-3mm]e.west) to[out=90,in=135,looseness=2] (e.135);  
    \end{pic}
    = 
    \begin{pic}[scale=.5]
        \node[dot] (d) at (0,0) {};
        \node[dot] (e) at (0,.75) {};
        \draw[arrow=.5] (e.-135) to[out=-135,in=-90,looseness=2] ([xshift=-3mm]e.west) to[out=90,in=135,looseness=2] (e.135);  
        \draw[arrow=.7] (d) to (e);
        \draw[arrow=.9] (e) to +(0,.75);
        \draw[arrow=.9] (d) to +(0,-1);
        \draw[reverse arrow=.9] (d) to[out=0,in=90] +(.7,-1);
        \draw[reverse arrow=.9] (d) to[out=180,in=90] +(-.7,-1);
    \end{pic}
  \]
  \end{definition}

By a \emph{Frobenius 3-structure diagram} we mean a finite connected diagram built from the pieces $\tinymulttriple, \tinymulttripleflip, \tinycup, \tinycupswap, \tinycap, \tinycapswap$ using identities, composition, and tensor products. 
% We say two such diagrams are \emph{equivalent} when they may be obtained from one another by bending at most one input and/or one output. 
After bending at most one input and/or output, any such diagram has input and output\footnote{When equating diagrams featuring ellipsis (``$\cdots$''), each series of inputs or outputs marked (with ``$\cdots$'') may be instantiated with any number of inputs or outputs respectively, so long as both diagrams have the same type.} of the form $A \otimes A^* \otimes \cdots \otimes A$, and we say such a diagram is in \emph{normal form} when for some natural numbers $m$ and $n$ it is equal to 
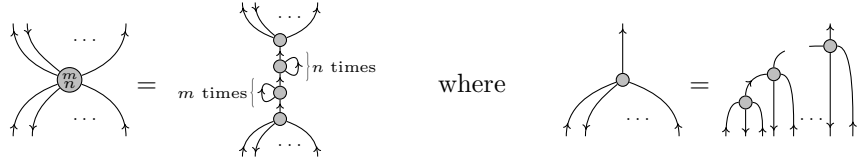
\[
  \begin{pic}[scale=.75]
    \node[bigdot] (d) at (0,0) {$\stackrel{m}{n}$};
    \draw[reverse arrow=.9] (d) to[out=-160,in=90] +(-1,-1);
    \draw[arrow=.9] (d) to[out=-140,in=90] +(-.66,-1);
    \draw[reverse arrow=.9] (d) to[out=-20,in=90] +(1,-1);
    \node at (.3,-.7) {$\cdots$};
    \draw[arrow=.9] (d) to[out=160,in=-90] +(-1,1);
    \draw[reverse arrow=.9] (d) to[out=140,in=-90] +(-.75,1);
    \draw[arrow=.9] (d) to[out=20,in=-90] +(1,1);
    \node at (.3,.7) {$\cdots$};
  \end{pic}
  =
  \begin{pic}[scale=.5]
    \node[dot] (a) at (0,0) {};
    \node[dot] (b) at (0,.7) {};
    \node[dot] (c) at (0,1.4) {};
    \node[dot] (d) at (0,2.1) {};
    \draw[arrow=.7] (a) to (b);
    \draw[arrow=.7] (b) to (c);
    \draw[arrow=.7] (c) to (d);
    \draw[reverse arrow=.9] (a) to[out=-160,in=90] +(-1,-1);
    \draw[arrow=.9] (a) to[out=-140,in=90] +(-.66,-1);
    \draw[reverse arrow=.9] (a) to[out=-20,in=90] +(1,-1);
    \node at (.3,-.7) {$\cdots$};
    \draw[arrow=.9] (d) to[out=160,in=-90] +(-1,1);
    \draw[reverse arrow=.9] (d) to[out=140,in=-90] +(-.66,1);
    \draw[arrow=.9] (d) to[out=20,in=-90] +(1,1);
    \node at (.3,2.7) {$\cdots$};
    \draw[arrow=.6] (b.-140) to[out=-140,in=-90,looseness=2] ([xshift=-3mm]b.west) to[out=90,in=140,looseness=2] (b.140);
    \draw[arrow=.6] (c.-40) to[out=-40,in=-90,looseness=2] ([xshift=3mm]c.east) to[out=90,in=40,looseness=2] (c.40);  
    \draw [gray,decorate,decoration={brace,amplitude=1pt}] ([xshift=-7mm,yshift=-2mm]b.south) -- ([xshift=-7mm,yshift=2mm]b.north) node [black,midway,xshift=-5.5mm] {{$m$ times}};    
    \draw [gray,decorate,decoration={brace,amplitude=1pt,mirror}] ([xshift=7mm,yshift=-2mm]c.south) -- ([xshift=7mm,yshift=2mm]c.north) node [black,midway,xshift=5mm] {{$n$ times}};  
  \end{pic}
  \qquad \text{where} \qquad
  \begin{pic}[scale=.75]
    \node[dot] (d) at (0,0) {};
    \draw[reverse arrow=.9] (d) to[out=-160,in=90] +(-1,-1);
    \draw[arrow=.9] (d) to[out=-140,in=90] +(-.66,-1);
    \draw[reverse arrow=.9] (d) to[out=-20,in=90] +(1,-1);
    \node at (.3,-.7) {$\cdots$};
    \draw[arrow=.9] (d) to +(0,1);
  \end{pic}
  =
  \begin{pic}[scale=.75]
    \node[dot] (a) at (0,0) {};
    \node[dot] (b) at (.5,.5) {};
    \node[dot] (c) at (1.5,1) {};
    \draw[arrow=.5] (a) to[out=90,in=180] (b);
    \draw[reverse arrow=.9] (a) to[out=180,in=90,looseness=.8] +(-.35,-.6);
    \draw[arrow=.9] (a) to +(0,-.6);
    \draw[reverse arrow=.9] (a) to[out=0,in=90,looseness=.8] +(.3,-.6);
    \draw[arrow=.9] (b) to +(0,-1.1);
    \draw[reverse arrow=.9] (b) to[out=0,in=90,looseness=.6] +(.35,-1.1);
    \draw[arrow=.9] (c) to +(0,.4);
    \draw[arrow=.9] (c) to +(0,-1.6);
    \draw[reverse arrow=.9] (c) to[out=0,in=90,looseness=.5] +(.4,-1.6);
    \draw (b) to[out=90,in=180] node[circle,fill=white,pos=.6]{} (c);
    \node at (1.2,-.3) {$\cdots$};
  \end{pic}
\]
% for some natural numbers $m$ and $n$, 
and $\begin{pic}[scale=.4]
    \node(c) at (0,1) {};
    \node[dot] (d) at (0,2.1) {};
    \draw[arrow=.7] (c) to (d);
    \draw[arrow=.9] (d) to[out=160,in=-90] +(-1,1);
    \draw[reverse arrow=.9] (d) to[out=140,in=-90] +(-.66,1);
    \draw[arrow=.9] (d) to[out=20,in=-90] +(1,1);
    \node at (.3,2.7) {$\cdots$};
  \end{pic}
$ is defined in terms of $\tinymulttriple$ similarly.

For a commutative Frobenius 3-structure, the left and right loops are equal, and so we write $n$ in place of $n, 0$ inside a normal form. This also implies that the structure satisfies sliding. A \emph{commutative} Frobenius 3-structure diagram is one that may additionally include any swap morphism $\tinyswap$ (it suffices to allow the swap on each of $A \otimes A$, $A \otimes A^*$, $A^* \otimes A$ and $A^* \otimes A^*$). Two such diagrams are \emph{equivalent} when they can be obtained from one another by bending (any number of) inputs and outputs and applying symmetry maps.

\begin{lemma}\label{lem:spiderrules} 
  For any Frobenius 3-structure $\tinymulttriple$ satisfying sliding we have:
  \begin{equation} \label{tern:spider_bend}
  \begin{pic}[scale=.75]
    \node[bigdot] (d) at (0,0) {$\stackrel{m}{n}$};
    \draw[reverse arrow=.9] (d) to[out=-160,in=90] +(-1,-1);
    \draw[arrow=.9] (d) to[out=-140,in=90] +(-.66,-1);
    \draw[reverse arrow=.9] (d) to[out=-20,in=90] +(1,-1);
    \node at (.3,-.7) {$\cdots$};
    \draw[arrow=.9] (d) to[out=160,in=-90] +(-1,1);
    \draw[reverse arrow=.9] (d) to[out=140,in=-90] +(-.75,1);
    \node at (.3,.7) {$\cdots$};
    \draw[arrow=.3,arrow=.97] (d) to[out=20,in=-90] +(1,.7) to[out=90,in=90,looseness=2] +(.5,0) to +(0,-1.7);
  \end{pic}
  =
  \begin{pic}[scale=.75]
    \node[bigdot] (d) at (0,0) {$\stackrel{m}{n}$};
    \draw[reverse arrow=.9] (d) to[out=-160,in=90] +(-1,-1);
    \draw[arrow=.9] (d) to[out=-140,in=90] +(-.66,-1);
    \node at (.3,-.7) {$\cdots$};
    \draw[arrow=.9] (d) to[out=160,in=-90] +(-1,1);
    \draw[reverse arrow=.9] (d) to[out=140,in=-90] +(-.75,1);
    \node at (.3,.7) {$\cdots$};
    \draw[arrow=.9] (d) to[out=20,in=-90] +(1,1);
    \draw[reverse arrow=.3, reverse arrow=.95] (d) to[out=-20,in=90] +(1,-.6) to[out=-90,in=-90,looseness=2] +(.5,0) to +(0,1.6);
  \end{pic}
  \end{equation}
  \begin{equation} \label{tern:spider_cap}
  \begin{pic}[scale=.75]
    \node[dot] (d) at (0,0) {$\stackrel{m}{n}$};
    \draw[reverse arrow=.9] (d) to[out=-160,in=90] +(-1,-1);
    \draw[arrow=.9] (d) to[out=-140,in=90] +(-.66,-1);
    \draw[reverse arrow=.9] (d) to[out=-20,in=90] +(1,-1);
    \node at (.3,-.7) {$\cdots$};
    \draw[arrow=.9] (d) to[out=160,in=-90] +(-1,1);
    \draw[reverse arrow=.9] (d) to[out=140,in=-90] +(-.75,1);
    \draw[arrow=.9] (d) to[out=20,in=-90] +(1,1);
    \draw[arrow=.35,arrow=.72] (d) to[out=100,in=180] +(0,1) to[out=0,in=80] (d.80);
    \node at (-.4,.7) {$\cdots$};
    \node at (.5,.7) {$\cdots$};
  \end{pic}
  =
  \begin{pic}[scale=.75]
    \node[bigdot] (d) at (0,0) {$\stackrel{m\!+\!1}{n}$};
    \draw[reverse arrow=.9] (d) to[out=-160,in=90] +(-1,-1);
    \draw[arrow=.9] (d) to[out=-140,in=90] +(-.66,-1);
    \draw[reverse arrow=.9] (d) to[out=-20,in=90] +(1,-1);
    \node at (.3,-.7) {$\cdots$};
    \draw[arrow=.9] (d) to[out=160,in=-90] +(-1,1);
    \draw[reverse arrow=.9] (d) to[out=140,in=-90] +(-.75,1);
    \draw[arrow=.9] (d) to[out=20,in=-90] +(1,1);
    \node at (.3,.7) {$\cdots$};
  \end{pic}
  \qquad\qquad
  \begin{pic}[scale=.75]
    \node[dot] (d) at (0,0) {$\stackrel{m}{n}$};
    \draw[reverse arrow=.9] (d) to[out=-160,in=90] +(-1,-1);
    \draw[arrow=.9] (d) to[out=-140,in=90] +(-.66,-1);
    \draw[reverse arrow=.9] (d) to[out=-20,in=90] +(1,-1);
    \node at (.3,-.7) {$\cdots$};
    \draw[arrow=.9] (d) to[out=160,in=-90] +(-1,1);
    \draw[reverse arrow=.9] (d) to[out=140,in=-90] +(-.75,1);
    \draw[arrow=.9] (d) to[out=20,in=-90] +(1,1);
    \draw[reverse arrow=.33,reverse arrow=.7] (d) to[out=100,in=180] +(0,1) to[out=0,in=80] (d.80);
    \node at (-.4,.7) {$\cdots$};
    \node at (.5,.7) {$\cdots$};
  \end{pic}
  =
  \begin{pic}[scale=.75]
    \node[bigdot] (d) at (0,0) {$\stackrel{m}{n\!\!+\!\!1}$};
    \draw[reverse arrow=.9] (d) to[out=-160,in=90] +(-1,-1);
    \draw[arrow=.9] (d) to[out=-140,in=90] +(-.66,-1);
    \draw[reverse arrow=.9] (d) to[out=-20,in=90] +(1,-1);
    \node at (.3,-.7) {$\cdots$};
    \draw[arrow=.9] (d) to[out=160,in=-90] +(-1,1);
    \draw[reverse arrow=.9] (d) to[out=140,in=-90] +(-.75,1);
    \draw[arrow=.9] (d) to[out=20,in=-90] +(1,1);
    \node at (.3,.7) {$\cdots$};
  \end{pic}
  \end{equation}
  \begin{equation}\label{tern:spider_compose}
  \begin{pic}[scale=.75]
    \node[bigdot] (d) at (0,0) {$\stackrel{m}{n}$};
    \node[bigdot] (e) at (.5,-.5) {$\stackrel{k}{l}$};
    \draw[arrow=.7] (e) to[out=110,in=-20] (d);
    \draw[reverse arrow=.9] (d) to[out=140,in=-90] +(-.5,.8);
    \draw[arrow=.9] (d) to[out=160,in=-90] +(-.7,.8);
    \draw[arrow=.9] (d) to[out=20,in=-90] +(.7,.8);
    \node at (.1,.5) {$\cdots$};

    \draw[reverse arrow=.9] (e) to[out=-160,in=90] +(-.7,-.7);
    \draw[arrow=.9] (e) to[out=-140,in=90] +(-.5,-.7);
    \draw[reverse arrow=.9] (e) to[out=-20,in=90] +(.7,-.7);
    \node at (.6,-1) {$\cdots$};

    \draw[reverse arrow=.9] (d) to[out=-160,in=90] +(-1,-1.2);
    \draw[arrow=.9] (d) to[out=-150,in=90] +(-.8,-1.2);
    \node at (-.4,-1) {$\cdots$};

    \draw[arrow=.95] (e) to[out=20,in=-90] +(1,1.3);
    \node at (1.1,.5) {$\cdots$};
  \end{pic}
  =
  \begin{pic}[scale=.75]
    \node[bigdot] (d) at (0,0) {$\stackrel{m\!+\!n}{k\!\!+\!\!l}$};
    \draw[reverse arrow=.9] (d) to[out=-160,in=90] +(-1,-1);
    \draw[arrow=.9] (d) to[out=-140,in=90] +(-.66,-1);
    \draw[reverse arrow=.9] (d) to[out=-20,in=90] +(1,-1);
    \node at (.3,-.7) {$\cdots$};
    \draw[arrow=.9] (d) to[out=160,in=-90] +(-1,1);
    \draw[reverse arrow=.9] (d) to[out=140,in=-90] +(-.75,1);
    \draw[arrow=.9] (d) to[out=20,in=-90] +(1,1);
    \node at (.3,.7) {$\cdots$};
  \end{pic}
  \end{equation}
The horizontal and vertical reflections of these equations also hold. If $\tinymulttriple$ is commutative, then additionally:
  \[
    \begin{pic}[scale=.5]
      \node[bigdot] (d) at (0,0) {$m$};
      \draw[reverse arrow=.9] (d) to[out=-160,in=90] +(-1,-1);
      \draw[reverse arrow=.9] (d) to[out=-20,in=90] +(1,-1);
      \node at (0,-.7) {$\cdots$};
      \draw[arrow=.9] (d) to[out=160,in=-90] +(-1,1);
      \draw[arrow=.9] (d) to[out=20,in=-90] +(1,1);
      \node at (0,.7) {$\cdots$};
    \end{pic}
    =
    \raisebox{1mm}{$\begin{pic}[scale=.5]
      \node[bigdot] (d) at (0,0) {$m$};
      \draw[reverse arrow=.9] (d) to[out=-160,in=90] +(-1,-1);
      \draw[reverse arrow=.9] (d) to[out=-20,in=90] +(1,-1);
      \node at (0,-.7) {$\cdots$};
      \draw[arrow=.9] (d) to[out=160,in=-90] +(-1,1);
      \draw[arrow=.9] (d) to[out=20,in=-90] +(1,1);
      \draw[arrow=.42] (d) to[out=120,in=-90] +(-.4,1) to[out=90,in=-90] +(.8,.5);
      \draw[arrow=.42] (d) to[out=60,in=-90] +(.4,1) to[out=90,in=-90] +(-.8,.5);
      \node at (.05,.7) {$\cdots$};
      \node at (.6,.7) {$\cdot$};\node at (.8,.7) {$\cdot$};
      \node at (-.6,.7) {$\cdot$};\node at (-.8,.7) {$\cdot$};
    \end{pic}$}
    =
    \raisebox{1mm}{$\begin{pic}[scale=.5]
      \node[bigdot] (d) at (0,0) {$m$};
      \draw[reverse arrow=.9] (d) to[out=-160,in=90] +(-1,-1);
      \draw[reverse arrow=.9] (d) to[out=-20,in=90] +(1,-1);
      \node at (0,-.7) {$\cdots$};
      \draw[arrow=.9] (d) to[out=160,in=-90] +(-1,1);
      \draw[arrow=.9] (d) to[out=20,in=-90] +(1,1);
      \draw[reverse arrow=.39] (d) to[out=120,in=-90] +(-.4,1) to[out=90,in=-90] +(.8,.5);
      \draw[reverse arrow=.39] (d) to[out=60,in=-90] +(.4,1) to[out=90,in=-90] +(-.8,.5);
      \node at (.05,.7) {$\cdots$};
      \node at (.6,.7) {$\cdot$};\node at (.8,.7) {$\cdot$};
      \node at (-.6,.7) {$\cdot$};\node at (-.8,.7) {$\cdot$};
    \end{pic}$}
    =
    \raisebox{-1.5mm}{$\begin{pic}[xscale=.5,yscale=-.5]
      \node[bigdot] (d) at (0,0) {$m$};
      \draw[arrow=.9] (d) to[out=-160,in=90] +(-1,-1);
      \draw[arrow=.9] (d) to[out=-20,in=90] +(1,-1);
      \node at (0,-.7) {$\cdots$};
      \draw[reverse arrow=.9] (d) to[out=160,in=-90] +(-1,1);
      \draw[reverse arrow=.9] (d) to[out=20,in=-90] +(1,1);
      \draw[reverse arrow=.42] (d) to[out=120,in=-90] +(-.4,1) to[out=90,in=-90] +(.8,.5);
      \draw[reverse arrow=.42] (d) to[out=60,in=-90] +(.4,1) to[out=90,in=-90] +(-.8,.5);
      \node at (.05,.7) {$\cdots$};
      \node at (.6,.7) {$\cdot$};\node at (.8,.7) {$\cdot$};
      \node at (-.6,.7) {$\cdot$};\node at (-.8,.7) {$\cdot$};
    \end{pic}$}
    =
    \raisebox{-1.5mm}{$\begin{pic}[xscale=.5,yscale=-.5]
      \node[bigdot] (d) at (0,0) {$m$};
      \draw[arrow=.9] (d) to[out=-160,in=90] +(-1,-1);
      \draw[arrow=.9] (d) to[out=-20,in=90] +(1,-1);
      \node at (0,-.7) {$\cdots$};
      \draw[reverse arrow=.9] (d) to[out=160,in=-90] +(-1,1);
      \draw[reverse arrow=.9] (d) to[out=20,in=-90] +(1,1);
      \draw[arrow=.42] (d) to[out=120,in=-90] +(-.4,1) to[out=90,in=-90] +(.8,.5);
      \draw[arrow=.42] (d) to[out=60,in=-90] +(.4,1) to[out=90,in=-90] +(-.8,.5);
      \node at (.05,.7) {$\cdots$};
      \node at (.6,.7) {$\cdot$};\node at (.8,.7) {$\cdot$};
      \node at (-.6,.7) {$\cdot$};\node at (-.8,.7) {$\cdot$};
    \end{pic}$}
  \]
\end{lemma}
\begin{proof}
  To see~\eqref{tern:spider_cap}, note that applying a cup or cap to a normal form simply adds an extra loop, and that sliding allows us to move loops freely around the diagrams. Equation~\eqref{tern:spider_compose} follows from this fact and by repeatedly applying associativity and the Frobenius law~\eqref{eq:frobenius:ternary}. For \eqref{tern:spider_bend}, again by moving loops around diagrams it suffices to consider when $m=n=0$. We then use symmetry: 
    \[
  \begin{pic}[xscale=-.66]
    \node[dot] (d) at (0,0) {};
    \draw[reverse arrow=.9] (d) to[out=-160,in=90] +(-.5,-.75);
    \draw[reverse arrow=.9] (d) to[out=-20,in=90] +(.5,-.75);
    \node at (0,-.5) {$\cdots$};
    \draw[arrow=.9] (d) to[out=20,in=-90] +(.5,.75);
    \node at (0,.5) {$\cdots$};
    \draw[arrow=.25, arrow=.97] (d) to[out=160,in=-90] +(-.5,.5) to[out=90,in=90,looseness=2] +(-.3,0) to +(0,-1.25); 
  \end{pic}
  =
  \scalebox{-1}[1]{$
  \begin{pic}[xscale=.75,yscale=.6]
    \node[dot] (a) at (.5,-.75) {};
    \node[dot] (b) at (0,-.3) {};
    \node[dot] (c) at (0,.3) {};
    \node[dot] (d) at (.5,.75) {};
    \draw[reverse arrow=.9] (a) to[out=-160,in=90] +(-.3,-.5);
    \draw[reverse arrow=.9] (a) to[out=-20,in=90] +(.3,-.5);
    \node at (.5,-1.1) {$\cdots$};
    \draw[arrow=.9] (d) to[out=160,in=-90] +(-.3,.5);
    \draw[arrow=.9] (d) to[out=20,in=-90] +(.3,.5);
    \node at (.5,1) {$\cdots$};
    \draw[arrow=.6] (c) to[out=0,in=-90] (d);
    \draw[arrow=.4] (a) to[out=90,in=0] (b);
    \draw[arrow=.7] (b) to (c);
    \draw[arrow=.9] (b) to +(0,-.95);
    \draw[reverse arrow=.9] (b) to[out=180,in=90,looseness=.7] +(-.4,-.95);
    \draw[reverse arrow=.9] (c) to +(0,.95);
    \draw[arrow=.2, arrow=.97] (c) to[out=180,in=-90] +(-.4,.5) to[out=90,in=90,looseness=2] +(-.3,0) to +(0,-2);
  \end{pic}
  $}
  =
    \scalebox{-1}[1]{$
  \begin{pic}[xscale=.75,yscale=.75]
    \node[dot] (a) at (.5,-.75) {};
    \node[dot] (b) at (0,-.25) {};
    \node[dot] (c) at (-.5,-.75) {};
    \node[dot] (d) at (0,.25) {};
    \draw[reverse arrow=.9] (a) to[out=-160,in=90] +(-.3,-.5);
    \draw[reverse arrow=.9] (a) to[out=-20,in=90] +(.3,-.5);
    \node at (.5,-1.1) {$\cdots$};
    \draw[arrow=.9] (d) to[out=160,in=-90] +(-.3,.5);
    \draw[arrow=.9] (d) to[out=20,in=-90] +(.3,.5);
    \node at (0,.5) {$\cdots$};
    \draw[arrow=.4] (a) to[out=90,in=0] (b);
    \draw[reverse arrow=.7] (b) to[out=180,in=90] (c);
    \draw[arrow=.9] (b) to +(0,-1);
    \draw[arrow=.7] (b) to (d);
    \draw[reverse arrow=.9] (c) to[out=0,in=90] +(.3,-.5);
    \draw[arrow=.9] (c) to +(0,-.5);
    \draw[reverse arrow=.125, reverse arrow=.95] (c) to[out=180,in=90] +(-.3,-.2) to[out=-90,in=-90,looseness=2] +(-.3,0) to +(0,1.7);
  \end{pic}
  $}
  =
  \begin{pic}[xscale=-.66]
    \node[dot] (d) at (0,0) {};
    \draw[reverse arrow=.9] (d) to[out=-20,in=90] +(.5,-.75);
    \node at (0,-.5) {$\cdots$};
    \draw[arrow=.9] (d) to[out=20,in=-90] +(.5,.75);
    \node at (0,.5) {$\cdots$};
    \draw[arrow=.9] (d) to[out=160,in=-90] +(-.5,.75);
    \draw[ reverse arrow=.25, reverse arrow=.97] (d) to[out=-160,in=90] +(-.5,-.5) to[out=-90,in=-90,looseness=2] +(-.3,0) to +(0,1.25); 
  \end{pic}  
  \]
  In the commutative case, associativity and commutativity give invariance under swapping legs of type $A$, while the same for legs of type $A^*$ follows by considering the dual Frobenius 3-structure.
\end{proof}

The statement and proof of the following theorem will frequently talk about `bending legs' of a morphism $f$. By this we mean turning it into $(\id \otimes \varepsilon) \circ (\pi \otimes \id) \circ (f \otimes \id)$ if bending an output, or into $(f \otimes \id) \circ (\pi \otimes \id) \circ (\id \otimes \eta)$ if bending an input, for some permutation $\pi$ built from the swap map. Graphically, bending an output looks as follows:
\[
  \begin{pic}[scale=.5,font=\small]
    \draw (-.2,.5) rectangle (1.5,1.5);
    \node at (.75,1) {$f$};
    \draw (0,1.5) to +(0,.6);
    \node[font=\tiny] at (.4,1.7) {$\dots$};
    \draw (.7,1.5) to +(0,.6);
    \draw (1,1.5) to +(0,.6);
    \draw (1.3,1.5) to +(0,.6);
    \draw (0,0) to +(0,.5);
    \node[font=\tiny] at (.4,0.2) {$\dots$};
    \draw (.7,0) to +(0,.5);
    \draw (1,0) to +(0,.5);
    \draw (1.3,0) to +(0,.5);
  \end{pic}  
  \qquad \leadsto \qquad
  \begin{pic}[scale=.5,font=\small]
    \draw (-.2,.5) rectangle (1.5,1.5);
    \node at (.75,1) {$f$};
    \draw (0,1.5) to +(0,.6);
    \node[font=\tiny] at (.4,1.7) {$\dots$};
    \draw (.7,1.5) to +(0,.1) to[out=90,in=90] +(1,.1) to +(0,-1.6);
    \draw (1,1.5) to +(0,.6);
    \draw (1.3,1.5) to +(0,.6);
    \draw (0,0) to +(0,.5);
    \node[font=\tiny] at (.4,0.2) {$\dots$};
    \draw (.7,0) to +(0,.5);
    \draw (1,0) to +(0,.5);
    \draw (1.3,0) to +(0,.5);
  \end{pic}  
\]

% \begin{theorem}[Normal form for Frobenius 3-structures]
%   For any (commutative) sliding Frobenius 3-structure, any (commutative) Frobenius 3-structure diagram with at least one input is equivalent to a normal form. In the commutative case, this also holds for diagrams with no inputs or outputs.
% \end{theorem}
\begin{theorem}[Normal form for Frobenius 3-structures]
  Let $\tinymulttriple$ be a Frobenius 3-structure that satisfies sliding. 
  \begin{itemize} 
    \item Any Frobenius 3-structure diagram with at least one input is equal to one in normal form after bending at most one input and/or output. 
    \item If $\tinymulttriple$ is commutative, any commutative Frobenius 3-structure diagram is equivalent to one in normal form. 
  \end{itemize}
\end{theorem}
\begin{proof}
  We will describe transforming a diagram into normal form relatively informally. Making all transformations explicit is routine, but only obscures the main algorithm and needlessly inflate the exposition.

  First consider the non-commutative case. We argue by induction on the number of dots $d$. The case $d = 0$ holds as bending a cup or cap yields the identity morphism, and the case $d = 1$ follows directly from symmetry~\eqref{eq:symmetry:ternary}. 
  
  Consider a diagram built from $d + 1$ dots. Use Lemma~\ref{lem:spiderrules} to bend wires and so assume there are no inputs. Naturality lets us rewrite the diagram so that it has a bottommost dot, and it is equal to one of the form: 
  
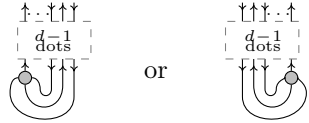
\[
  \begin{pic}[scale=.5]
    \node[dot] (d) at (0,0) {};
    \draw[arrow=.7] (d) to (0,.5);
    \draw[reverse arrow=.9] (d) to[out=0,in=90] +(.4,-.3) to[out=-90,in=-90,looseness=2] +(.3,0) to (.7,.5);
    \draw[arrow=.93] (d) to[out=-90,in=-90,looseness=2] (1,-.2) to (1,.5);
    \draw[reverse arrow=.95] (d) to[out=180,in=90] +(-.4,-.4) to[out=-90,in=-90,looseness=1.3] (1.3,-.5) to (1.3,.5);
    \draw[dashed, draw=gray] (-.2,.5) rectangle (1.75,1.5);
    \node at (.75,1.1) {$d\!-\!1$};
    \node at (.75,.85) {dots};
    \draw[arrow=.9] (0,1.5) to +(0,.5);
    \node at (.4,1.7) {$\dots$};
    \draw[reverse arrow=.85] (.7,1.5) to +(0,.5);
    \draw[arrow=.9] (1,1.5) to +(0,.5);
    \draw[reverse arrow=.85] (1.3,1.5) to +(0,.5);
  \end{pic}
  \qquad \text{or} \qquad
  \begin{pic}[xscale=-.5,yscale=.5]
    \node[dot] (d) at (0,0) {};
    \draw[arrow=.7] (d) to (0,.5);
    \draw[reverse arrow=.9] (d) to[out=0,in=90] +(.4,-.3) to[out=-90,in=-90,looseness=2] +(.3,0) to (.7,.5);
    \draw[arrow=.93] (d) to[out=-90,in=-90,looseness=2] (1,-.2) to (1,.5);
    \draw[reverse arrow=.95] (d) to[out=180,in=90] +(-.4,-.4) to[out=-90,in=-90,looseness=1.3] (1.3,-.5) to (1.3,.5);
    \draw[dashed, draw=gray] (-.2,.5) rectangle (1.75,1.5);
    \node at (.75,1.1) {$d\!-\!1$};
    \node at (.75,.85) {dots};
    \draw[arrow=.9] (0,1.5) to +(0,.5);
    \draw[reverse arrow=.85] (.7,1.5) to +(0,.5);
    \draw[arrow=.9] (1,1.5) to +(0,.5);
    \draw[reverse arrow=.85] (1.3,1.5) to +(0,.5);
        \node at (.3,1.7) {$\dots$};
  \end{pic}
  \]
  (This includes the case where the bottommost dot originally had no output wires at all. Because the diagram of $d$ dots was connected, in that case the upper subdiagram of $d-1$ dots must connect two inputs with a cup.)
  By induction hypothesis the upper subdiagram consists of (at most four) disconnected pieces, each in normal form up to bending legs. Applying the rules of Lemma~\ref{lem:spiderrules} it follows that their composite with the lower subdiagram indeed equals one in normal form, after bending one output.
  % Maybe treat smallish example?

  The commutative case follows in the same way. The base case $d=0$ is trivial since duals are assumed to be symmetric, and the base case $d=1$ is straightforward by commutativity. The rest of the argument is identical using Lemma~\ref{lem:spiderrules}.
\end{proof}

\begin{remark}
  The result does not hold for diagrams with no input or output in the non-commutative case, since for example the closed loop $\tinycapswap \circ \tinycup$ need not coincide with the closed loop $\tinycap \circ \tinycupswap$.
\end{remark}

\begin{corollary} \label{cor:normalform_consequences}
  In a monoidal category:
  \begin{itemize}
    \item any two Frobenius 3-diagrams of a sliding Frobenius 3-structure that are of the same type and have no internal loops are equal;
    \item any two Frobenius 3-diagrams of a normal Frobenius 3-structure that are of the same type and have more than one input or output are equal.
    \qed
  \end{itemize}
\end{corollary}

\section{Connectors}\label{sec:connectors}

In this section, we extend the correspondence between Frobenius 2-structures and groupoids to one between 3-structures and connectors. 

In any category $\catC$ with pullbacks one can define a double equivalence relation as an internal equivalence relation in the category of internal equivalence relations in $\catC$. In order to compare this notion with the one of normal-dagger Frobenius $3$-structure, the following formulation of the notion of double equivalence relation will be useful:
\begin{definition} \label{def:double_eq} 
A \emph{double equivalence relation} on two equivalence relations $R$ and $S$ on an object $A$ in $\catC$ is a subobject $\Lambda \rightarrowtail A^4$ such that:
\begin{itemize}
\item the equivalence relations $R, S$ satisfy $R(x,y) \iff \Lambda(x,y,y,x)$ and $S(y,z) \iff \Lambda(y,y,z,z)$;
\item the relations $(x,y)r(u,z) \iff \Lambda(x,y,z,u) \iff (y,z)l(x,u)$ are equivalence relations on $R, S$, respectively.
\end{itemize}
In particular, $\Lambda(x,y,z,u) \implies R(x,y), R(z,u), S(y,z), S(x,u)$.
%It is a \emph{pre-groupoid} when $R(x,y)$, $S(y,z)$ imply there is a unique $u$ with $\Lambda(x,y,z,u)$.
\end{definition}

%In a monoidal dagger category, we call a morphism $\tinymulttriple \colon A \otimes A^* \otimes A \to A$ \emph{dagger symmetric} when~\eqref{eq:symmetry:ternary} holds and is equal to $\tinymulttripleflip$, as for a dagger Frobenius 3-structure, and \emph{left and right idempotent} just as in Definition~\ref{def:3-struc}.

\begin{proposition} \label{prop:double_eq_rel}
Let $\catC$ be a regular category. There is a one-to-one correspondence between double equivalence relations  on subobjects $B \rightarrowtail A$ and dagger symmetric, left and right idempotent morphisms $(A, \tinymulttriple)$ in $\Rel(\catC)$. %, where $\tinymulttriple((x,y,z),u) \iff \Lambda(x,y,z,u)$.
\end{proposition}
\begin{proof}
Consider $\tinymulttriple$ as a subobject of $A^4$. Symmetry of the relations $l$ and $r$ amounts to the rule:
\begin{equation} \label{eq:sym_lambda}
\tinymulttriple(x,y,z,u) \iff \tinymulttriple(u,z,y,x) \iff \tinymulttriple(y,x,u,z)
\end{equation}
which is easily seen to be equivalent to dagger symmetry of $\tinymulttriple$. Right dagger idempotence is precisely the requirement that $r$, defined as above, satisfies $r = r^{\oprel} = r \circ r$. Equivalently, $r$ restricts to an equivalence relation on $R = \{(x,y) \mid (x,y)r(x,y)\}$, as required, and dually the same holds for $l$ and $S$. 

It remains to check that the rest of Definition~\ref{def:double_eq} holds automatically. Symmetry of the relation $R$ follows from~\eqref{eq:sym_lambda}, while transitivity follows from transitivity of $l$, observing that $R(x,y) \iff (x,x) l (y,y)$. Equivalently, $R$ restricts to an equivalence relation on $B = \{x \mid R(x,x) \} = \{x \mid \tinymulttriple(x,x,x,x)\}$. Dually, the same holds for $S$.
\end{proof}

For any pair of equivalence relations $r_1, r_2 \colon R \rightrightarrows A $ and $s_1, s_2 \colon S \rightrightarrows A$ we write $R \times_A S$ for the pullback of $r_2$ and $s_1$, $\ie$ the following subobject
\[
  R \times_A S = \sem{(x,y,z) \in A \times A \times A \mid R(x,y) \wedge S(y,z)}.
\]

\begin{definition}
  A \emph{connector} between equivalence relations $R \rightrightarrows A$ and $S \rightrightarrows A$ in a 
  % finitely complete category 
regular category 
  is a morphism $p\colon R \times_A S \to A$ such that:
  \begin{align}
   \label{con_xSpxyz}  x  S  p(x,y,z) & & \text{ for }(x,y,z) \in R \times_A S\text{;}  \\
   \label{con_zSpxyz} z  R p(x,y,z) & & \text{ for }(x,y,z) \in R \times_A S\text{;}  \\
   \label{con_pxxy} p(x,y,y) & = x & \text{ whenever defined;}\\
   \label{con_pxyy} p(y,y,z) & = z & \text{whenever defined;}\\
   \label{con_assoc} p(p(x,y,z),u,v) & = p(x,y,p(z,u,v))\text{;}
  \end{align}
  where~\eqref{con_assoc} means that if one side defined so is the other and they are equal. 
\end{definition}
The (object of) morphisms of any (internal) groupoid forms a connector where $p(x,y,z)=x \circ y^{-1} \circ z$, with $R(x,y)$ whenever $s(x) = s(y)$ and $S(y,z)$ whenever $t(y)=t(z)$. Viewing a vector space as an additive group, for vectors $x,y,z$ the vector $p(x,y,z)$ can be visualised as completing the parallelogram~\cite{Smith}.
% For vectors $x,y,z$ in a vector space, $p(x,y,z)$ can be visualised as completing the parallelogram~\cite{Smith}.
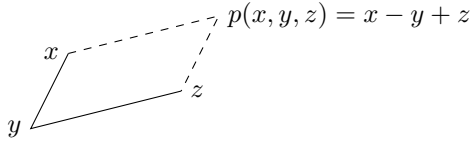
\[\begin{tikzpicture}
  \node[left] at (.5,1) {$x$};
  \node[left] at (0,0) {$y$};
  \node[right] at (2,.5) {$z$};
  \node[right] at (2.5,1.5) {$p(x,y,z)=x-y+z$};
  \draw (2,.5) to (0,0) to (.5,1);
  \draw[dashed] (.5,1) to (2.5,1.5) to (2,.5);
\end{tikzpicture}\]

%A double equivalence relation is called a \emph{pre-groupoid} whenever $R(x,y)$, $S(y,z)$ imply there is a unique $u$ with $\Lambda(x,y,z,u)$ \cite{kock1988generalized}. In this case, $\Lambda$ corresponds to a partial function $p(x,y,z)$ in $\catC$, and the definition simplifies as follows. 

Any connector $(A, R, S, p)$ may be seen to define a double equivalence relation with $\Lambda(x,y,z,u) \iff p(x,y,z) = u$, and then its equivalence relations $R, S$ coincide with those of Definition~\ref{def:double_eq}. The double equivalence relations arising this way are called \emph{pregroupoids}, and indeed this is how these structures were first studied~\cite{kock1988generalized} (see also \cite{Johnstone}). Later on it became clear that the assumption that $R$ and $S$ were \emph{effective} equivalence relations could be dropped to develop commutator theory in the larger context of regular categories, and this led to the notion of connector \cite{bourngran:centrality, CentralityNormality}.

\begin{remark}
  In universal algebra connectors are useful to treat commutators in categorical terms. For instance, in the category $\cat{Gp}$ of groups two normal subgroups $K$ and $L$ of a group $G$ have trivial commutator, i.e. $[K,L]=\{1\}$, if and only if there is a connector $p \colon {R_K} \times_G  R_L \rightarrow G$ between the congruences $R_K$ and $R_L$ canonically associated with the quotients $G/K$ and $G/L$, respectively. Similarly, in the regular Mal'tsev category $\cat{Gp(Haus})$ of Hausdorff groups, the topological closure $\overline{[K,L]}= \{1\}$ of the group-theoretic commutator of two normal closed subgroups $K$ and $L$ of a Hausdorff group $G$ is trivial if and only if there is a connector between the corresponding congruences. In the category $\cat{CRng}$ of commutative (not necessarily unital) rings the role of the group-theoretic commutator is played by the product of ideals: two ideals $I$ and $J$ of a commutative ring $R$ are such that $I\cdot J=0$ if and only if there is a connector between the corresponding congruences $R_I$ and $R_J$. In all these examples a connector on two given congruences is unique, when it exists \cite{bourngran:centrality}.
\end{remark}

\begin{theorem}\label{thm:frobenius3structuresareconnectors}
  Let $\cat{C}$ be a regular category.
  There is a one-to-one correspondence between connectors $(A,p)$ in $\cat{C}$ and normal dagger Frobenius 3-structures $(A,\tinymulttriple)$ in $\Rel(\cat{C})$, where $\tinymulttriple \colon (x,y,z) \mapsto p(x,y,z)$ whenever $p(x,y,z)$ is defined.
\end{theorem}

\begin{proof}
Any normal 3-structure satisfies $\tinymulttriple \circ \tinymulttripleflip = \id[]$, hence corresponding to a partial function $p \colon A \times A \times A \relto A$ in $\catC$. Let us write $(x,y,z) \defined$ to mean that $p(x,y,z)$ is defined. Associativity~\eqref{eq:associativity:ternary} of $\tinymulttriple$ translates precisely to associativity~\eqref{con_assoc} of connectors, while normality translates to:
\begin{align} \label{eq:norm_assoc_pregroupoid}
    (\exists y) \ p(x,y,y) = u \iff x = u \qquad&
    \qquad (\exists y) \   p(y,y,z) = u \iff z = u
  \end{align}
which gives~\eqref{con_pxxy} and \eqref{con_pxyy}. 

By Proposition~\ref{prop:double_eq_rel}, we have that $\tinymulttriple$ forms a double equivalence relation on $A$. Defining $R$ and $S$ as for any double equivalence relation, this gives~\eqref{con_xSpxyz} and~\eqref{con_zSpxyz} automatically, and that $(x,y,z) \defined \implies R(x,y) \wedge S(y,z)$. Finally, we check the converse, so that $p$ is indeed defined on $R \times_A S$. Suppose that $R(x,y)$ and $S(y,z)$, so that $x = p(x,y,y)$ and $y = p(z,z,y)$. Then by associativity $x = p(x,y,p(z,z,y)) = p(p(x,y,z),z,y)$, and so $(x,y,z)\defined$, as required.% Hence $p$ is a connector.
%Alternative: just invoke that a connector is a double equivalence relation.
% Conversely, we have remarked that for any connector $(A,p)$, $\tinymulttriple$ as defined above is a double equivalence relation. Hence using Proposition~\ref{prop:double_eq_rel}, $\tinymulttriple$ in particular satisfies symmetry~\eqref{eq:symmetry:ternary}. 

%Proves symmetry directly.
In the other direction, for any connector $(A,R, S, p)$, normality~\eqref{eq:norm_assoc_pregroupoid} for $\tinymulttriple$ follows from~\eqref{con_pxxy} and~\eqref{con_pxyy}, since $R$ and $S$ are reflexive and so $p(x,x,x) =x$ for all $x$ by~\eqref{con_pxxy}. It remains to check dagger symmetry for $\tinymulttriple$, which we translated earlier as~\eqref{eq:sym_lambda}. Suppose that $u = p(x,y,z)$. Then since $S$ is symmetric we have $S(z,y)$ and hence
%\[
$
p(u,z,y) = p(p(x,y,z),z,y) = p(x,y,p(z,z,y)) = p(x,y,y) = x
$
%\]
as required. Dually, one may show that $p(y,x,u) = z$, completing~\eqref{eq:sym_lambda}. 

 Finally, note that this correspondence is an equivalence since for any connector $(A, R, S, p)$ the relations $R$ and $S$ are defined in terms of the double equivalence relation $\Lambda = \tinymulttriple$ as in Definition~\ref{def:double_eq}. To see this for $R$, observe that $p(x,y,y) = x \iff (x,y,y)\defined \iff R(x,y) \wedge S(y,y)$ since $p$ is defined on the pullback $R \times_A S$. But since $S$ is reflexive this holds precisely when $R(x,y)$. 
\end{proof}

\begin{remark}\label{rem:faithful}
  The functor $\Rel(\Conn(\cat{C})) \to \Rel(\cat{C})$ given by $(A, R, S, p) \mapsto A$ is \emph{not} faithful for regular categories $\cat{C}$. Therefore we cannot upgrade Theorem~\ref{thm:frobenius3structuresareconnectors} to an equivalence of categories as in Theorem~\ref{thm:frobeniusstructuresaregroupoids} using a ternary analogue to the $\CP$ construction. 
  To see this, let $A$ be a non-trivial abelian group in $\cat{C}$, and $p \colon A^3 \rightarrow A$ the canonical connector between $A^2= A \times A \rightrightarrows A$ and $A^2=A \times A \rightrightarrows A$ defined by $p(x,y,z)= x-y+z$. The following commutative diagram
  
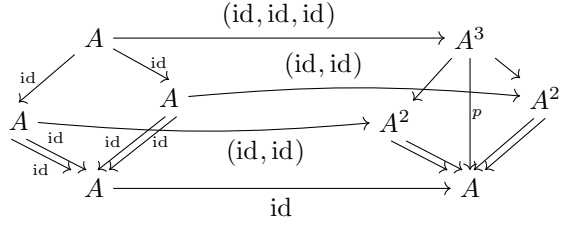
\[\begin{tikzpicture}
    \node (b) at (0,0) {$A$};
    \node (l) at (-1,.9) {$A$};
    \node (r) at (1,1.2) {$A$};
    \node (t) at (0,2) {$A$};
    \node (B) at (5,0) {$A$};
    \node (L) at (4,.9) {$A^2$};
    \node (R) at (6,1.2) {$A^2$};
    \node (T) at (5,2) {$A^3$};
    \draw[->] (t) to node[left,font=\tiny]{$\id$} (l.north);
    \draw[->] (t) to node[right,font=\tiny]{$\id$} (r.north);
    \draw[->] ([xshift=-1mm]l.south) to node[below,font=\tiny]{$\id$} ([xshift=-3mm]b.north);
    \draw[->] ([xshift=1mm]l.south) to node[above,font=\tiny]{$\id$} ([xshift=-1mm]b.north);
    \draw[->] ([xshift=-.5mm]r.south) to node[left,font=\tiny]{$\id$} ([xshift=.5mm]b.north);
    \draw[->] ([xshift=1mm]r.south) to node[right,font=\tiny]{$\id$} ([xshift=2mm]b.north);
    \draw[->] (T) to (L);
    \draw[->] (T) to (R);
    \draw[->] ([xshift=-.5mm]L.south) to ([xshift=-3mm]B.north);
    \draw[->] ([xshift=1.5mm]L.south) to ([xshift=-1mm]B.north);
    \draw[->] ([xshift=-1.5mm]R.south) to ([xshift=.5mm]B.north);
    \draw[->] (R.south) to ([xshift=2mm]B.north);  
    \draw[->] (T) to node[right=-1mm,font=\tiny]{$p$} (B.north);
    \draw[->] (b) to node[below]{$\id$} (B);
    \draw[->] (t) to node[above]{$(\id,\id,\id)$} (T);
    \draw[->] (l) to[out=-5,in=-175] node[below,pos=.7]{$(\id , \id)$} (L);
    \draw[->] (r) to[out=5,in=175] node[above,pos=.4]{$(\id , \id)$} (R);
  \end{tikzpicture}\]
  provides a morphism $\id \colon (A, A, A, \id[A]) \to (A, A^2, A^2, p)$ of connectors that is not an isomorphism of connectors, because the diagonal morphism $(\id , \id) \colon A \to A^2$ is not an isomorphism. 
  Then, since the diagram
  \[\begin{tikzpicture}[xscale=3,yscale=.75]
    \node (t) at (0,2) {$(A, A, A, \id[A])$};
    \node (b) at (0,0) {$(A, A^2, A^2, p)$};
    \node (m) at (1,1) {$(A,A^2, A^2, p)$};
    \draw[>->] (t) to node[above]{$\id$} (m);
    \draw[>->] (b) to node[below]{$\id$} (m);
    \draw[->] (t) to node[left]{$\id$} (b);
  \end{tikzpicture}\]
  commutes, the functor $\Rel(\Conn(\cat{C})) \to \Rel(\cat{C})$ is not faithful, since the left vertical morphism is an isomorphism when we look at it in $\catC$.
  \qed
\end{remark}

\section{Relating Frobenius 3-structures and Frobenius 2-structures}\label{sec:2vs3}

In this final section we exhibit several relationships between Frobenius 2-structures and Frobenius 3-structures, generalising those between groupoids and connectors.

A Frobenius 2-structure (on a dual object) is \emph{symmetric} when
\begin{align}\label{eq:symmetric:binary}
  \begin{pic}[scale=.4]
    \node[dot] (d) at (0,0) {};
    \draw (d.north) to +(0,.6) node[dot]{};
    \draw[reverse arrow=.9] (d.west) to[out=180,in=90] +(-.7,-1.5);
    \draw[reverse arrow=.95] (d.east) to[out=0,in=90] +(.5,-.5) to[out=-90,in=-90,looseness=1.5] +(1,0) to +(0,2);
  \end{pic} 
  & = 
  \begin{pic}[xscale=-.4,yscale=.4]
    \node[dot] (d) at (0,0) {};
    \draw (d.north) to +(0,.6) node[dot]{};
    \draw[reverse arrow=.9] (d.east) to[out=180,in=90] +(-.7,-1.5);
    \draw[reverse arrow=.95] (d.west) to[out=0,in=90] +(.5,-.5) to[out=-90,in=-90,looseness=1.5] +(1,0) to +(0,2);
  \end{pic} 
\end{align}

We call the map~\eqref{eq:symmetric:binary} the \emph{involution} for $\tinymult$ and draw it as $\tinyinvolution$.
%\begin{pic}[scale=.3]
%  \node[dot] (d) at (0,0) {};
%  \draw[arrow=.9] (0,0.7) to (d);
%  \draw[arrow=.9] (0,-0.7) to (d);
%\end{pic}$.
A dagger Frobenius 3-structure (on a dagger dual object) is \emph{unital} when there exists a morphism $\ \ \tinyunit \colon I \to A$ satisfying
\begin{align}\label{eq:unital:ternary}
  \begin{pic}[scale=.4]
    \node[dot] (d) at (0,0) {};
    \draw[arrow=.8] (d) to +(0,1);
    \draw[arrow=.5] (d) to +(0,-1) node[dot]{};
    \draw[reverse arrow=.7] (d) to[out=180,in=90] +(-1,-1) node[dot]{};
    \draw[reverse arrow=.9] (d) to[out=0,in=90,looseness=.8] +(1,-2);
  \end{pic} 
  =
  \begin{pic}[scale=.4]
    \draw[arrow=.5] (0,0) to (0,3);
  \end{pic}
  =
  \begin{pic}[xscale=-.4,yscale=.4]
    \node[dot] (d) at (0,0) {};
    \draw[arrow=.8] (d) to +(0,1);
    \draw[arrow=.5] (d) to +(0,-1) node[dot]{};
    \draw[reverse arrow=.7] (d) to[out=180,in=90] +(-1,-1) node[dot]{};
    \draw[reverse arrow=.9] (d.west) to[out=0,in=90,looseness=.8] +(1,-2);
  \end{pic} 
\end{align}
where $\tinyunitdual = (\tinyunit)_* \colon I \to A^*$.

\begin{theorem} \label{thm:2struct_equal_unit3struc} 
  There is a one-to-one correspondence between dagger symmetric Frobenius 2-structures and unital dagger Frobenius 3-structures:
  
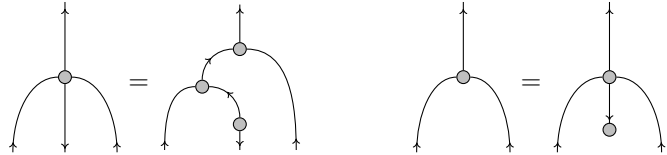
\[
    \begin{pic}[xscale=.75]
      \node[dot] (d) at (0,0) {};
      \draw[arrow=.9] (d) to +(0,1);
      \draw[arrow=.95] (d) to +(0,-1);
      \draw[reverse arrow=.95] (d.west) to[out=180,in=90] +(-.8,-1);
      \draw[reverse arrow=.95] (d.east) to[out=0,in=90] +(.8,-1);
    \end{pic}
    =
    \begin{pic}[scale=.5]
      \node[dot] (l) at (0,0){};
      \node[dot] (r) at (1,1){};
      \node[dot] (i) at (1,-1){};
      \draw[arrow=.5] (l.north) to[out=90,in=180] (r.west);
      \draw[arrow=.9] (r.north) to +(0,1);
      \draw[reverse arrow=.95] (l.west) to[out=180,in=90] ([xshift=-2cm,yshift=-5mm]i.south);
      \draw[reverse arrow=.5] (l.east) to[out=0,in=90] (i.north);
      \draw[arrow=.8] (i.south) to +(0,-.5);
      \draw[reverse arrow=.95] (r.east) to[out=0,in=90] ([xshift=15mm,yshift=-5mm]i.south);
    \end{pic}
    \qquad \qquad
    \begin{pic}[xscale=.66]
      \node[dot] (d) at (0,0) {};
      \draw[arrow=.9] (d) to +(0,1);
      \draw[reverse arrow=.95] (d.west) to[out=180,in=90] +(-.8,-1);
      \draw[reverse arrow=.95] (d.east) to[out=0,in=90] +(.8,-1);
    \end{pic}
    =
    \begin{pic}[xscale=.75]
      \node[dot] (d) at (0,0) {};
      \draw[arrow=.9] (d) to +(0,1);
      \draw[arrow=.8] (d) to +(0,-.7) node[dot]{};
      \draw[reverse arrow=.95] (d.west) to[out=180,in=90] +(-.8,-1);
      \draw[reverse arrow=.95] (d.east) to[out=0,in=90] +(.8,-1);
    \end{pic}
  \]
  The Frobenius 2-structure is special if and only if the Frobenius 3-structure is normal if and only if the Frobenius 3-structure is left or right idempotent.
\end{theorem}

\begin{proof}
  Associativity for $\tinymulttriple$ follows immediately from associativity for $\tinymult$. 
  Unitality of $\tinymulttriple$ follows from the fact that any symmetric Frobenius 2-structure $\tinymult$ satisfies $\tinyinvolutionswap \circ \tinyunitdual = \tinyunit$. Further, $\tinymulttriple$ indeed respects daggers:
  
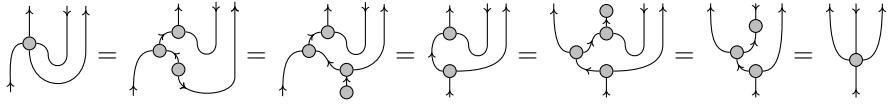
\[
    \begin{pic}[scale=.5]
      \node[dot] (d) at (0,0) {};
      \draw[arrow=.9] (d) to +(0,1);
      \draw[reverse arrow=.9] (d) to[out=180,in=90,looseness=.8] +(-.5,-1.3);
      \draw[reverse arrow=.95] (d) to[out=0,in=90] +(.5,-.3) to[out=-90,in=-90,looseness=2] +(.5,0) to +(0,1.3);
      \draw[arrow=.97] (d) to[out=-90,in=-90,looseness=2] (1.5,-.2) to (1.5,1);
    \end{pic}
    =
    \begin{pic}[scale=.5]
      \node[dot] (l) at (0,0) {};
      \node[dot] (r) at (.5,.5) {};
      \draw[arrow=.6] (l) to[out=90,in=180] (r);
      \draw[arrow=.9] (r) to +(0,.75);
      \draw[reverse arrow=.95] (r) to[out=0,in=90] +(.5,-.3) to[out=-90,in=-90,looseness=2] +(.5,0) to +(0,1.1);    
      \node[dot] (i) at (.5,-.5) {};
      \draw[arrow=.6] (i) to[out=90,in=0] (l);
      \draw[arrow=.1, arrow=.97] (i) to[out=-90,in=-90] +(1.5,-.2) to +(0,2);
      \draw[reverse arrow=.9] (l) to[out=180,in=90] +(-.7,-1.2);
    \end{pic}
    =
    \begin{pic}[scale=.5]
      \node[dot] (l) at (0,0) {};
      \node[dot] (r) at (.5,.5) {};
      \draw[arrow=.6] (l) to[out=90,in=180] (r);
      \draw[arrow=.9] (r) to +(0,.75);
      \draw[reverse arrow=.95] (r) to[out=0,in=90] +(.5,-.3) to[out=-90,in=-90,looseness=2] +(.5,0) to +(0,1.1);    
      \draw[reverse arrow=.9] (l) to[out=180,in=90] +(-.7,-1.2);
      \node[dot] (b) at (1,-.5) {};
      \draw[arrow=.5] (b) to[out=180,in=0] (l);
      \draw[reverse arrow=.35] (b) to +(0,-.6) node[dot]{};
      \draw[arrow=.97] (b) to[out=0,in=-90] +(1,.75) to +(0,1);
    \end{pic}
    =
    \begin{pic}[scale=.5]
      \node[dot] (r) at (.5,.5) {};
      \draw[arrow=.9] (r) to +(0,.75);
      \draw[reverse arrow=.95] (r) to[out=0,in=90] +(.5,-.3) to[out=-90,in=-90,looseness=2] +(.5,0) to +(0,1.1);        
      \node[dot] (b) at (.5,-.5) {};
      \draw[arrow=.5] (b) to[out=180,in=180] (r);
      \draw[reverse arrow=.9] (b) to +(0,-.7);
      \draw[arrow=.97] (b) to[out=0,in=-90] +(1.5,.7) to +(0,1.1);
    \end{pic}
    =
    \begin{pic}[scale=.5]
      \node[dot] (b) at (.5,0) {};
      \node[dot] (m) at (-.3,.5) {};
      \node[dot] (t) at (.5,1) {};
      \draw[reverse arrow=.9] (b) to +(0,-.7);
      \draw[arrow=.5] (b) to[out=180,in=-90] (m);
      \draw[arrow=.6] (m) to[out=0,in=180] (t);
      \draw[arrow=.9] (m) to[out=180,in=-90,looseness=.8] +(-.6,1.3);
      \draw[arrow=.5] (t) to +(0,.6) node[dot]{};
      \draw[reverse arrow=.95] (t) to[out=0,in=90] +(.5,-.3) to[out=-90,in=-90,looseness=2] +(.5,0) to +(0,1.1);            
      \draw[arrow=.97] (b) to[out=0,in=-90] +(1.5,.7) to +(0,1.1);
    \end{pic}
    =
    \begin{pic}[scale=.5]
      \node[dot] (b) at (0,0){};
      \node[dot] (t) at (-.5,.5){};
      \node[dot] (i) at (0,1.2) {};
      \draw[reverse arrow=.9] (b) to +(0,-.7);
      \draw[arrow=.7] (b) to[out=180,in=-90] (t);
      \draw[arrow=.7] (t) to[out=0,in=-90] (i);
      \draw[reverse arrow=.9] (i) to +(0,.6);
      \draw[arrow=.95] (t) to[out=180,in=-90] +(-.7,1.3);
      \draw[arrow=.95] (b) to[out=0,in=-90,looseness=.8] +(.7,1.8);
    \end{pic}
    =
    \begin{pic}[scale=.5]
      \node[dot] (d) at (0,0) {};
      \draw[reverse arrow=.9] (d) to +(0,-1);
      \draw[reverse arrow=.9] (d) to +(0,1.5);
      \draw[arrow=.95] (d) to[out=180,in=-90,looseness=.8] +(-.7,1.5);
      \draw[arrow=.95] (d) to[out=0,in=-90,looseness=.8] +(.7,1.5);
    \end{pic}
  \]
  This also shows that $\tinymulttriple$ is symmetric.

  In the other direction, associativity, unitality and the Frobenius law for $\tinymult$ all follow immediately from those for $\tinymulttriple$, and this Frobenius 2-structure is a dagger one by construction. For symmetry:
  
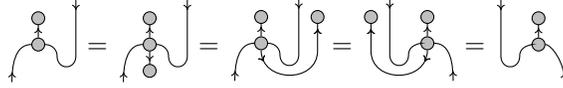
\[
    \begin{pic}[scale=.5]
      \node[dot] (d) at (0,0) {};
      \draw[reverse arrow=.95] (d) to[out=0,in=90] +(.5,-.3) to[out=-90,in=-90,looseness=2] +(.5,0) to +(0,1.5);
      \draw[reverse arrow=.9] (d) to[out=180,in=90] +(-.7,-1);
      \draw[arrow=.5] (d) to +(0,.7) node[dot]{};
    \end{pic}
    =
    \begin{pic}[scale=.5]
      \node[dot] (d) at (0,0) {};
      \draw[reverse arrow=.95] (d) to[out=0,in=90] +(.5,-.3) to[out=-90,in=-90,looseness=2] +(.5,0) to +(0,1.5);
      \draw[reverse arrow=.9] (d) to[out=180,in=90] +(-.7,-1);
      \draw[arrow=.5] (d) to +(0,.7) node[dot]{};
      \draw[arrow=.5] (d) to +(0,-.7) node[dot]{};
    \end{pic}
    =
    \begin{pic}[scale=.5]
      \node[dot] (d) at (0,0) {};
      \draw[reverse arrow=.95] (d) to[out=0,in=90] +(.5,-.3) to[out=-90,in=-90,looseness=2] +(.5,0) to +(0,1.5);
      \draw[reverse arrow=.9] (d) to[out=180,in=90] +(-.7,-1);
      \draw[arrow=.5] (d) to +(0,.7) node[dot]{};
      \draw[arrow=.1,arrow=.9] (d) to[out=-90,in=-90,looseness=1.5] +(1.5,-.2) to +(0,.9) node[dot]{};
    \end{pic}
    =
    \begin{pic}[xscale=-.5,yscale=.5]
      \node[dot] (d) at (0,0) {};
      \draw[reverse arrow=.95] (d) to[out=0,in=90] +(.5,-.3) to[out=-90,in=-90,looseness=2] +(.5,0) to +(0,1.5);
      \draw[reverse arrow=.9] (d.east) to[out=180,in=90] +(-.5,-1);
      \draw[arrow=.5] (d) to +(0,.7) node[dot]{};
      \draw[arrow=.1,arrow=.9] (d) to[out=-90,in=-90,looseness=1.5] +(1.5,-.2) to +(0,.9) node[dot]{};
    \end{pic}
    =
    \begin{pic}[xscale=-.5,yscale=.5]
      \node[dot] (d) at (0,0) {};
      \draw[reverse arrow=.95] (d) to[out=0,in=90] +(.5,-.3) to[out=-90,in=-90,looseness=2] +(.5,0) to +(0,1.5);
      \draw[reverse arrow=.9] (d) to[out=180,in=90] +(-.7,-1);
      \draw[arrow=.5] (d) to +(0,.7) node[dot]{};
    \end{pic}  
  \]
  We now show these constructions are inverse. Starting from a 2-structure $\tinymult$ returns to the same 2-structure using $\tinyinvolutionswap \circ \tinyunitdual = \tinyunit$. Conversely, starting from any unital 3-structure $\tinymulttriple$, Corollary~\ref{cor:normalform_consequences} gives:
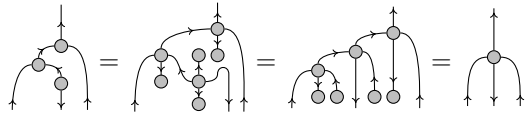
  \[
    \begin{pic}[scale=.5]
      \node[dot] (t) at (0,1) {};
      \node[dot] (m) at (-.6,.5) {};
      \node[dot] (b) at (0,0) {};
      \draw[arrow=.5] (t) to +(0,1.1);
      \draw[arrow=.5] (m) to[out=90,in=180] (t);
      \draw[arrow=.5] (b) to[out=90,in=0] (m);
      \draw[arrow=.9] (b) to +(0,-.7);
      \draw[reverse arrow=.9] (m) to[out=180,in=90,looseness=.8] +(-.7,-1.2);
      \draw[reverse arrow=.9] (t) to[out=0,in=90,looseness=.8] +(.7,-1.7);
    \end{pic}
    =
    \begin{pic}[scale=.5]
      \node[dot] (t) at (1,1.2) {};
      \node[dot] (m) at (-.5,.5) {};
      \draw[arrow=.5] (t) to +(0,.7);
      \draw[arrow=.7] (m) to[out=90,in=180,looseness=.7] (t);
      \draw[reverse arrow=.9] (m) to[out=180,in=90,looseness=.8] +(-.7,-1.5);
      \draw[reverse arrow=.9] (t) to[out=0,in=90,looseness=.7] +(.7,-2.2);
      \draw[arrow=.5] (t) to +(0,-.7) node[dot] {};
      \draw[arrow=.5] (m) to +(0,-.7) node[dot] {};
      \node[dot] (b) at (.5,-.2) {};
      \draw[arrow=.5] (b) to[out=180,in=0] (m);
      \draw[arrow=.5] (b) to +(0,.7) node[dot]{};
      \draw[arrow=.5] (b) to +(0,-.6) node[dot]{};
      \draw[arrow=.9] (b) to[out=0,in=-90] +(.5,.2) to[out=90,in=90,looseness=2] +(.3,0) to +(0,-1);
    \end{pic}
    =
    \begin{pic}[scale=.5]
      \node[dot] (l) at (0,0) {};
      \node[dot] (m) at (1,.5) {};
      \node[dot] (r) at (2,1) {};
      \draw[arrow=.7] (l) to[out=90,in=180,looseness=.7] (m);
      \draw[arrow=.7] (m) to[out=90,in=180,looseness=.7] (r);
      \draw[arrow=.9] (r) to +(0,.7);
      \draw[arrow=.5] (l) to +(0,-.7) node[dot]{};
      \draw[reverse arrow=.9] (l) to[out=180,in=90] +(-.7,-1);
      \draw[reverse arrow=.6] (l) to[out=0,in=90] +(.5,-.7) node[dot]{};
      \draw[arrow=.9] (m) to +(0,-1.5);
      \draw[reverse arrow=.7] (m) to[out=0,in=90,looseness=.7] +(.5,-1.2) node[dot]{};
      \draw[arrow=.75] (r) to +(0,-1.7) node[dot] {};
      \draw[reverse arrow=.9] (r) to[out=0,in=90,looseness=.7] +(.7,-2);
    \end{pic}
    =
    \begin{pic}[scale=.5]
      \node[dot] (d) at (0,0) {};
      \draw[arrow=.9] (d) to +(0,1.3);
      \draw[arrow=.9] (d) to +(0,-1.3);
      \draw[reverse arrow=.9] (d) to[out=180,in=90] +(-.7,-1.3);
      \draw[reverse arrow=.9] (d) to[out=0,in=90] +(.7,-1.3);
    \end{pic}
  \]

  For the final statement, note that the left loop for $\tinymulttriple$ may be described in terms of $\tinymult$ as:
  
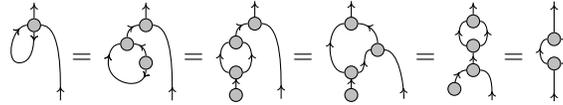
\[ 
    \begin{pic}[scale=.5]
      \node[dot] (d) at (0,0) {};
      \draw[arrow=.9] (d) to +(0,.6);
      \draw[reverse arrow=.9] (d) to[out=0,in=90,looseness=.7] +(.7,-2);
      \draw[arrow=.1, arrow=.9] (d) to[out=-90,in=0] ([xshift=-4mm,yshift=-10mm]d) to[out=180,in=180] (d);
    \end{pic}
    =
    \begin{pic}[scale=.5]
      \node[dot] (d) at (0,0) {};
      \draw[arrow=.9] (d) to +(0,.6);
      \draw[reverse arrow=.9] (d) to[out=0,in=90,looseness=.7] +(.7,-2);
      \node[dot] (l) at (-.5,-.5) {};
      \node[dot] (i) at (0,-1) {};
      \draw[arrow=.7] (l) to[out=90,in=180] (d);
      \draw[arrow=.7] (i) to[out=90,in=0] (l);
      \draw[arrow=.1, arrow=.8] (i) to[out=-90,in=-90,looseness=1.5] +(-1,0) to[out=90,in=180] (l);
    \end{pic}
    =
    \begin{pic}[scale=.5]
      \node[dot] (d) at (0,0) {};
      \draw[arrow=.9] (d) to +(0,.6);
      \draw[reverse arrow=.9] (d) to[out=0,in=90,looseness=.7] +(.7,-2);
      \node[dot] (l) at (-.5,-.5) {};
      \node[dot] (b) at (-.5,-1.3) {};
      \draw[arrow=.7] (l) to[out=90,in=180] (d);
      \draw[reverse arrow=.3] (b) to +(0,-.6) node[dot]{};
      \draw[arrow=.6] (b) to[out=0,in=0] (l);
      \draw[arrow=.6] (b) to[out=180,in=180] (l);
    \end{pic}
    =
    \begin{pic}[scale=.5]
      \node[dot] (d) at (0,0) {};
      \draw[arrow=.9] (d) to +(0,.6);
      \node[dot] (l) at (.7,-.7) {};
      \node[dot] (b) at (0,-1.3) {};
      \draw[reverse arrow=.9] (l) to[out=0,in=90,looseness=.7] +(.7,-1.3);
      \draw[arrow=.7] (l) to[out=90,in=0] (d);
      \draw[reverse arrow=.3] (b) to +(0,-.6) node[dot]{};
      \draw[arrow=.5] (b) to[out=0,in=180] (l);
      \draw[arrow=.6] (b) to[out=180,in=180] (d);
    \end{pic}
    =
    \begin{pic}[scale=.5]
      \node[dot] (t) at (0,.9) {};
      \node[dot] (m) at (0,.25) {};
      \node[dot] (b) at (0,-.4) {};
      \draw[arrow=.9] (t) to +(0,.5);
      \draw[arrow=.7] (b) to (m);
      \draw[arrow=.6] (m) to[out=0,in=0] (t);
      \draw[arrow=.6] (m) to[out=180,in=180] (t);
      \draw[reverse arrow=.9] (b) to[out=0,in=90] +(.5,-.8);
      \draw[reverse arrow=.5] (b) to[out=180,in=90] +(-.5,-.5) node[dot]{};
    \end{pic}
    =    
    \begin{pic}[scale=.5]
      \node[dot] (t) at (0,.9) {};
      \node[dot] (m) at (0,.25) {};
      \draw[arrow=.9] (t) to +(0,1);
      \draw[arrow=.6] (m) to[out=0,in=0] (t);
      \draw[arrow=.6] (m) to[out=180,in=180] (t);
      \draw[reverse arrow=.9] (m) to +(0,-1);
    \end{pic}
  \]
  The same holds for the right loop, similarly. Hence $\tinymult$ is special if and only if $\tinymulttriple$ is normal. 
  In the presence of a unit, it is easy to see that this is equivalent to left or right idempotence of $\tinymulttriple$.
\end{proof}

\begin{remark}
 This result generalises the relationship between groupoids and connectors to arbitrary 2- and 3-structures, via the correspondences of Theorems~\ref{thm:frobeniusstructuresaregroupoids} and~\ref{thm:frobenius3structuresareconnectors}. Firstly, note that a connector $(A, R, S, p)$ defines a groupoid (uniquely) if and only if $A_1 := A$ may be given a reflexive graph structure $d, c \colon A_1 \to A_0$ (with common splitting $e \colon A_0 \to A_1$), compatibly with the equivalence relations $R, S$: this means that $R$ is the kernel pair of $d$ and $S$ the kernel pair of $c$.
This situation ensures that the connector operation $p(a,b,c)$ induces a composition $a \circ b$ of a groupoid on this reflexive graph, by defining the composite $a \circ b = p(a, 1_A, b)$ for any ``composable pair of morphisms'' $a \colon A \to C$ and $b \colon B \to A$ (see Theorem $3.6$  in\cite{CPP} for more details).

By Theorem~\ref{thm:2struct_equal_unit3struc} this holds if and only if the 3-structure $\tinymulttriple$ corresponding to $(A, R, S, p)$ has a unit. In Theorem~\ref{thm:frobeniusstructuresaregroupoids} we also saw that this unit corresponds to the object $A_0$ of the groupoid. 
\end{remark}

We can make the construction functorial as follows. Let $\catD$ be a monoidal dagger category.
A \emph{morphism} of symmetric dagger Frobenius 2-structures $f \colon (A, \tinymult) \to (B,\tinymult[whitedot])$ is a morphism $f \colon A \to B$ preserving multiplication and involution, in that $f \circ \tinymult = \tinymult[whitedot] \circ (f \otimes f)$ and $\tinyinvolution[whitedot] \circ f = f_* \circ \tinyinvolution$. 
A \emph{morphism} of 3-structures $g \colon (A, \tinymulttriple) \to (B,\tinymulttriple[whitedot])$ is a morphism $g \colon A \to B$ satisfying $g \circ \tinymulttriple = \tinymulttriple[whitedot] \circ (g \otimes g_* \otimes g)$.  A morphism $f$ of symmetric dagger 2-structures or unital 3-structures is \emph{unital} when it satisfies ${f \circ \ \tinyunit} = \tinyunit[whitedot]$. 
We write $\BinaryCat(\cat{D})$ and $\TernCat(\cat{D})$ for the categories of symmetric dagger Frobenius 2-structures and Frobenius 3-structures in $\catD$ and their morphisms. Note that Example~\ref{ex:product3} makes $\TernCat(\catD)$ symmetric monoidal when $\catD$ is symmetric monoidal, and combining Examples~\ref{ex:dual3} and~\ref{ex:opposite3} then provides dual objects for every object in $\TernCat(\catD)$. 

It is now easy to see that the construction $U(\tinymult) = \tinymulttriple$ of Theorem~\ref{thm:2struct_equal_unit3struc} defines a functor $U \colon \BinaryCat(\catD) \to \TernCat(\catD)$, since any morphism of 2-structures preserves multiplication and involution, and hence ternary multiplication also. We will return to this shortly in Remark~\ref{rem:functoriality}.

%Functoriality of $U$ is straightforward from the constructions of Theorem~\ref{thm:2struct_equal_unit3struc}. From the lefthand side, we see that any morphism of 2-structures preserves the multiplication and involution, and hence the ternary multiplication als. Conversely, from the righthand side, any unital morphism of 3-structures is also one of 2-structures.

% \begin{definition} \label{def:binary_cat}
%   Let $\catD$ be a monoidal dagger category.
%   The category $\BinaryCat(\cat{D})$ has as objects dagger Frobenius 3-structures in $\cat{D}$, and as morphisms $(A, \tinymult) \to (B, \tinymult[whitedot])$ those morphisms $f \colon A \to B$ in $\cat{D}$ satisfying $f \circ \tinymult = \tinymult[whitedot] \circ (f \otimes f)$ and $\tinyinvolution[whitedot] \circ f = f_* \circ \tinyinvolution$.  We write $\BinaryCatSpec(\catD)$ for the full subcategory of special 2-structures.
% \end{definition}

% \begin{definition}
%   Let $\catD$ be a monoidal dagger category.
%   The category $\TernCat(\catD)$ has dagger Frobenius 3-structures in $\catD$ as objects, and morphisms $f \colon A \to B$ in $\catD$ satisfying $f \circ \tinymulttriple = \tinymulttriple[whitedot] \circ (f \otimes f_* \otimes f)$ as morphisms $(A, \tinymulttriple) \to (B, \tinymulttriple[whitedot])$. 
% \end{definition}

\subsection{Splitting}

For any left idempotent dagger Frobenius 3-structure $(A, \tinymulttriple)$ the morphism $l_A = \tinylidem \colon A^* \otimes A \to A^* \otimes A$ is \emph{dagger idempotent}, \ie~satisfies $l_A = {l_A}^{\dag} \circ {l_A}$. We say such an idempotent $p \colon X \to X$ in a dagger category has a \emph{dagger splitting} when there exists some $i \colon Y \to X$ with $p = i \circ i^{\dag}$ and $i^{\dag} \circ i = \id[Y]$. Such a morphism $i$ is called an \emph{isometry}.

\begin{example} \label{ex:dagger_idem_RelC}
%Dagger idempotents $p \colon A \to A$ in $\Rel(\catC)$ are precisely symmetric, transitive relations on $A$, and these dagger split  iff $\catC$ is \emph{(Barr) exact}. Indeed, any such $p$ restricts isometrically to an equivalence relation on $\{a \mid p(a,a)\}$, and so it suffices to show that equivalence relations dagger split. But exactness is equivalent to requiring that equivalence relations $p$ split in $\Rel(\catC)$ \cite{succi1975teoria}, and in this case there is a dagger splitting $p = e^{\dag} \circ e$ where $e$ is the coequalizer of $p_1$, $p_2 \colon P \to a$. 
Dagger idempotents $P \colon A \relto A$ dagger split in $\Rel(\catC)$ iff $\catC$ is \emph{(Barr) exact}. Explicitly, dagger idempotence says that $P$ is a symmetric, transitive relation on $A$. Any such $P$ restricts isometrically to an equivalence relation on $\{a \mid P(a,a)\}$, and so it suffices to show that equivalence relations dagger split. But exactness is equivalent to requiring that equivalence relations $P$ split in $\Rel(\catC)$ \cite{succi1975teoria}, and in this case there is a dagger splitting $P = e^{\oprel} \circ e$ where $e$ is the coequaliser of $p_1$, $p_2 \colon P \to A$. 
\end{example}

% \begin{example} \label{ex:nine_legged_embedding}
% Extending Example~\ref{example:special_2struc_rep}, for any normal dagger Frobenius 3-structure $(A, \tinymulttriple)$, the isometry $\tinymulttripleflip \colon A \to A \otimes A^* \otimes A$ is a morphism from $\tinymulttriple$ into the following Frobenius 3-structure:
%   \begin{equation} \label{ex:llr} 
%     \begin{pic}[scale=.5]
%       \draw[arrow=.5] (0,0) to (0,1);
%       \draw[reverse arrow=.5] (.5,0) to (.5,1);
%       \draw[arrow=.5] (1,0) to[out=90,in=90,looseness=2] (1.5,0);
%       \draw[arrow=.5] (2,0) to[out=90,in=90,looseness=1.5] (3.5,0);
%       \draw[reverse arrow=.5] (2.5,0) to[out=90,in=90,looseness=2] (3,0);
%       \draw[arrow=.5] (4,0) to (4,1);
%     \end{pic}
%   \end{equation}
% The same result holds for the 3-structure given by horizontally reflecting~\eqref{ex:llr}. 
% \end{example}

Splittings give another way to turn Frobenius 3-structures into 2-structures.

\begin{theorem}[Splitting Construction] \label{thm:splitting}
  Let $(A, \tinymulttriple)$ be a left idempotent dagger Frobenius 3-structure for which $l_A$ dagger splits over an isometry $i \colon L \to A^* \otimes A$. Then $L$ is a symmetric dagger Frobenius 2-structure with:
  \begin{equation} \label{eq:split_construction}
    \begin{pic}[scale=.5]
      \node[whitedot] (d) at (0,0) {};
      \draw[arrow=.9] (d) to +(0,2);
      \draw[reverse arrow=.9] (d) to[out=180,in=90] +(-.9,-2);
      \draw[reverse arrow=.9] (d) to[out=0,in=90] +(.9,-2);
    \end{pic}
    =
    \begin{pic}[scale=.5]
      \node[morphism] (l) at (0,0) {$i$};
      \node[morphism] (r) at (2,0) {$i$};
      \node[morphism,hflip] (t) at (1,2) {$i$};
      \draw[reverse arrow=.9] (l.south) to +(0,-.7);
      \draw[reverse arrow=.9] (r.south) to +(0,-.7);
      \draw[arrow=.9] (t.north) to +(0,.7);
      \draw[reverse arrow=.5] ([xshift=-2mm]l.north west) to[out=90,in=-90] ([xshift=-2mm]t.south west);
      \draw[arrow=.5] ([xshift=2mm]r.north east) to[out=90,in=-90] ([xshift=2mm]t.south east);
      \draw[arrow=.5] ([xshift=2mm]l.north east) to[out=90,in=90] ([xshift=-2mm]r.north west);
    \end{pic}
    \qquad\qquad
    \begin{pic}[scale=.5]
      \draw[reverse arrow=.1] (0,0) to +(0,-1) node[whitedot]{};
    \end{pic}
    =
    \begin{pic}[scale=.5]
      \node[morphism, hflip] (i) at (0,0) {$i$};
      \draw[arrow=.5] ([xshift=-2mm]i.south west) to[out=-90,in=-90,looseness=3] ([xshift=2mm]i.south east);
      \draw[arrow=.9] (i.north) to +(0,1);
    \end{pic}
  \end{equation}
  It is special precisely when $\tinymulttriple$ is additionally right idempotent.
% Conversely, a representable dagger 2-Frobenius structure $(X, \tinymult, i, A)$ makes $A$ a left idempotent dagger Frobenius 3-structure by $l_A = i \circ i^{\dag}$.
\end{theorem}
\begin{proof}
  Note that $l = l_A$ is a dagger idempotent and also satisfies: 
  \begin{equation} \label{eq:frob_idem_useful}
    \begin{pic}[scale=.5]
      \node[morphism] (t) at (0,2) {$l$};
      \node[morphism] (b) at (1,0) {$l$};
      \draw[reverse arrow=.9] ([xshift=-2mm]t.north west) to +(0,.8);
      \draw[arrow=.9] ([xshift=2mm]t.north east) to +(0,.8);
      \draw[reverse arrow=.9] ([xshift=2mm]b.south east) to +(0,-.8);
      \draw[arrow=.9] ([xshift=-2mm]b.south west) to +(0,-.8);
      \draw[arrow=.9] ([xshift=-2mm]t.south west) to +(0,-2.8);
      \draw[arrow=.5] ([xshift=2mm]b.north east) to[out=90,in=-90] ([xshift=2mm]t.south east);
      \draw[reverse arrow=.9] ([xshift=-2mm]b.north west) to[out=90,in=90,looseness=2] +(-.5,0) to +(0,-1.6);
    \end{pic}
    =
    \begin{pic}[scale=.5]
      \node[dot] (r) at (1.5,0) {};
      \node[dot] (l) at (0,1) {};
      \draw[arrow=.9] (r) to +(0,-1);
      \draw[reverse arrow=.9] (r) to[out=0,in=90] +(.8,-1);
      \draw[arrow=.5] (r) to[out=90,in=0] (l);
      \draw[arrow=.9] (l) to +(0,-2);
      \draw[arrow=.9] (l) to +(0,2.2);
      \draw[reverse arrow=.9] (l) to[out=180,in=90] +(-.5,-.3) to[out=-90,in=-90,looseness=2] +(-.5,0) to +(0,2.5);
      \draw[reverse arrow=.9] (r) to[out=180,in=90] +(-.8,-1);
    \end{pic}
    =
    \begin{pic}[scale=.5]
      \node[dot] (r) at (1.5,1) {};
      \node[dot] (l) at (0,0) {};
      \draw[arrow=.9] (r) to +(0,2.2);
      \draw[arrow=.9] (r) to +(0,-2);
      \draw[reverse arrow=.9] (r) to[out=0,in=90] +(.8,-2);
      \draw[arrow=.5] (l) to[out=90,in=180] (r);
      \draw[reverse arrow=.9] (l) to[out=0,in=90] +(.8,-1);
      \draw[arrow=.9] (l) to +(0,-1);
      \draw[reverse arrow=.9] (l) to[out=180,in=90] +(-.5,-.3) to[out=-90,in=-90,looseness=2] +(-.5,0) to +(0,3.5);
    \end{pic}
    =
    \begin{pic}[scale=.5]
      \node[morphism] (l) at (0,0) {$l$};
      \node[morphism] (r) at (2,0) {$l$};
      \draw[arrow=.9] ([xshift=-2mm]l.south west) to +(0,-.7);
      \draw[arrow=.9] ([xshift=-2mm]r.south west) to +(0,-.7);
      \draw[reverse arrow=.9] ([xshift=2mm]l.south east) to +(0,-.7);
      \draw[reverse arrow=.9] ([xshift=2mm]r.south east) to +(0,-.7);
      \draw[reverse arrow=.9] ([xshift=-2mm]l.north west) to[out=90,in=-90] +(1,1.5) to +(0,1.2);
      \draw[arrow=.9] ([xshift=2mm]r.north east) to[out=90,in=-90] +(-1,1.5) to +(0,1.2);
      \draw[arrow=.5] ([xshift=2mm]l.north east) to[out=90,in=90] ([xshift=-2mm]r.north west);
    \end{pic}    
    =
    \begin{pic}[scale=.5]
      \node[morphism] (l) at (0,0) {$l$};
      \node[morphism] (r) at (2,0) {$l$};
      \node[morphism] (t) at (1,2) {$l$};
      \draw[arrow=.9] ([xshift=-2mm]l.south west) to +(0,-.7);
      \draw[arrow=.9] ([xshift=-2mm]r.south west) to +(0,-.7);
      \draw[reverse arrow=.9] ([xshift=2mm]l.south east) to +(0,-.7);
      \draw[reverse arrow=.9] ([xshift=2mm]r.south east) to +(0,-.7);
      \draw[reverse arrow=.9] ([xshift=-2mm]t.north west) to +(0,.7);
      \draw[arrow=.9] ([xshift=2mm]t.north east) to +(0,.7);
      \draw[reverse arrow=.5] ([xshift=-2mm]l.north west) to[out=90,in=-90] ([xshift=-2mm]t.south west);
      \draw[arrow=.5] ([xshift=2mm]r.north east) to[out=90,in=-90] ([xshift=2mm]t.south east);
      \draw[arrow=.5] ([xshift=2mm]l.north east) to[out=90,in=90] ([xshift=-2mm]r.north west);
    \end{pic}
  \end{equation}
  The similar equation with $l$ in the lower left holds dually. 
  The Frobenius 2-structure laws then follow, using that $i$ is an isometry. For example, to show unitality:
  
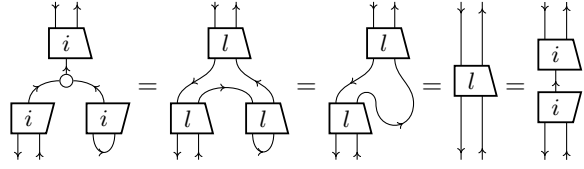
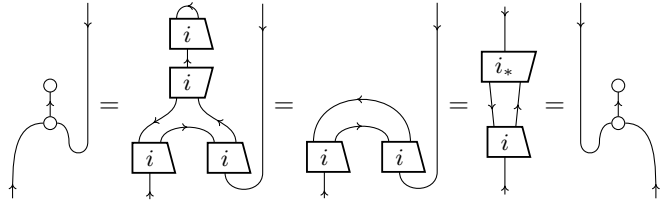
\[
    \begin{pic}[scale=.5]
      \node[morphism,hflip] (l) at (0,0) {$i$};
      \node[morphism,hflip] (r) at (2,0) {$i$};
      \node[morphism] (t) at (1,2) {$i$};
      \node[whitedot] (d) at (1,1) {};
      \draw[arrow=.5] (l.north) to[out=90,in=180] (d);
      \draw[arrow=.5] (r.north) to[out=90,in=0] (d);
      \draw[arrow=.5] (d) to (t.south);
      \draw[arrow=.9] ([xshift=-2mm]l.south west) to +(0,-.7);
      \draw[reverse arrow=.9] ([xshift=2mm]l.south east) to +(0,-.7);
      \draw[reverse arrow=.9] ([xshift=-2mm]t.north west) to +(0,.7);
      \draw[arrow=.9] ([xshift=2mm]t.north east) to +(0,.7);
      \draw[arrow=.5] ([xshift=-2mm]r.south west) to[out=-90,in=-90,looseness=3] ([xshift=2mm]r.south east);
    \end{pic}
    =
    \begin{pic}[scale=.5]
      \node[morphism] (l) at (0,0) {$l$};
      \node[morphism] (r) at (2,0) {$l$};
      \node[morphism] (t) at (1,2) {$l$};
      \draw[arrow=.9] ([xshift=-2mm]l.south west) to +(0,-.7);
      \draw[reverse arrow=.9] ([xshift=2mm]l.south east) to +(0,-.7);
      \draw[reverse arrow=.9] ([xshift=-2mm]t.north west) to +(0,.7);
      \draw[arrow=.9] ([xshift=2mm]t.north east) to +(0,.7);
      \draw[reverse arrow=.5] ([xshift=-2mm]l.north west) to[out=90,in=-90] ([xshift=-2mm]t.south west);
      \draw[arrow=.5] ([xshift=2mm]r.north east) to[out=90,in=-90] ([xshift=2mm]t.south east);
      \draw[arrow=.5] ([xshift=2mm]l.north east) to[out=90,in=90] ([xshift=-2mm]r.north west);
      \draw[arrow=.5] ([xshift=-2mm]r.south west) to[out=-90,in=-90,looseness=3] ([xshift=2mm]r.south east);
    \end{pic}
    =
    \begin{pic}[scale=.5]
      \node[morphism] (l) at (0,0) {$l$};
      \node[morphism] (t) at (1,2) {$l$};
      \draw[arrow=.9] ([xshift=-2mm]l.south west) to +(0,-.7);
      \draw[reverse arrow=.9] ([xshift=2mm]l.south east) to +(0,-.7);
      \draw[reverse arrow=.9] ([xshift=-2mm]t.north west) to +(0,.7);
      \draw[arrow=.9] ([xshift=2mm]t.north east) to +(0,.7);
      \draw[reverse arrow=.5] ([xshift=-2mm]l.north west) to[out=90,in=-90] ([xshift=-2mm]t.south west);
      \draw[arrow=.5] ([xshift=2mm]l.north east) to[out=90,in=90,looseness=2] +(.5,0) to[out=-90,in=-90,looseness=2] +(1,0) to[out=90,in=-90] ([xshift=2mm]t.south east);
    \end{pic}
    =
    \begin{pic}[scale=.5]
      \node[morphism] (l) at (0,0) {$l$};
      \draw[reverse arrow=.9] ([xshift=2mm]l.south east) to +(0,-1.7);
      \draw[arrow=.9] ([xshift=-2mm]l.south west) to +(0,-1.7);
      \draw[arrow=.9] ([xshift=2mm]l.north east) to +(0,1.7);
      \draw[reverse arrow=.9] ([xshift=-2mm]l.north west) to +(0,1.7);
    \end{pic}
    =
    \begin{pic}[scale=.5]
      \node[morphism,hflip] (l) at (0,0) {$i$};
      \node[morphism] (t) at (0,1.4) {$i$};
      \draw[reverse arrow=.9] ([xshift=2mm]l.south east) to +(0,-1);
      \draw[arrow=.9] ([xshift=-2mm]l.south west) to +(0,-1);
      \draw[arrow=.9] ([xshift=2mm]t.north east) to +(0,1);
      \draw[reverse arrow=.9] ([xshift=-2mm]t.north west) to +(0,1);
      \draw[arrow=.5] (l.north) to (t.south);
      %\draw[arrow=.5] ([xshift=2mm]l.north east) to ([xshift=2mm]t.south east);    
    \end{pic}
  \]
  For symmetry:
  \[
    \begin{pic}[scale=.5]
      \node[whitedot] (d) at (0,0) {};
      \draw[reverse arrow=.9] (d) to[out=180,in=90] (-1,-2);
      \draw[arrow=.5] (d) to +(0,1) node[whitedot]{};
      \draw[reverse arrow=.9] (d) to[out=0,in=90] (.5,-.5) to[out=-90,in=-90,looseness=2] +(.5,0) to +(0,3.7);
    \end{pic}
    =
    \begin{pic}[scale=.5]
      \node[morphism] (l) at (0,0) {$i$};
      \node[morphism] (r) at (2,0) {$i$};
      \node[morphism,hflip] (t) at (1,2) {$i$};
      \node[morphism] (T) at (1,3.3) {$i$};
      \draw[reverse arrow=.9] (l.south) to +(0,-.7);
      \draw[arrow=.5] (t.north) to (T.south);
      \draw[arrow=.5] ([xshift=2mm]T.north east) to[out=90,in=90,looseness=2] ([xshift=-2mm]T.north west);
      \draw[reverse arrow=.5] ([xshift=-2mm]l.north west) to[out=90,in=-90] ([xshift=-2mm]t.south west);
      \draw[arrow=.5] ([xshift=2mm]r.north east) to[out=90,in=-90] ([xshift=2mm]t.south east);
      \draw[arrow=.5] ([xshift=2mm]l.north east) to[out=90,in=90] ([xshift=-2mm]r.north west);
      \draw[reverse arrow=.9] (r.south) to[out=-90,in=-90,looseness=1.5] +(1,0) to +(0,4.5);
    \end{pic}
    =
    \begin{pic}[scale=.5]
      \node[morphism] (l) at (0,0) {$i$};
      \node[morphism] (r) at (2,0) {$i$};
      \draw[reverse arrow=.9] (l.south) to +(0,-.7);
      \draw[reverse arrow=.5] ([xshift=-2mm]l.north west) to[out=90,in=90,looseness=1.5] ([xshift=2mm]r.north east);
      \draw[arrow=.5] ([xshift=2mm]l.north east) to[out=90,in=90] ([xshift=-2mm]r.north west);
      \draw[reverse arrow=.9] (r.south) to[out=-90,in=-90,looseness=1.5] +(1,0) to +(0,4.5);
    \end{pic}
    =
    \begin{pic}[scale=.5]
      \node[morphism] (l) at (0,0) {$i$};
      \node[morphism,hflip] (t) at (0,2) {$i_*$};
      \draw[reverse arrow=.9] (l.south) to +(0,-1);
      \draw[reverse arrow=.9] (t.north) to +(0,1.2);
      \draw[reverse arrow=.5] ([xshift=-2mm]l.north west) to ([xshift=-2mm]t.south west);
      \draw[arrow=.5] ([xshift=2mm]l.north east) to ([xshift=2mm]t.south east);
    \end{pic}
    =
    \begin{pic}[xscale=-.5,yscale=.5]
      \node[whitedot] (d) at (0,0) {};
      \draw[reverse arrow=.9] (d.east) to[out=180,in=90] (-1,-2);
      \draw[arrow=.5] (d) to +(0,1) node[whitedot]{};
      \draw[reverse arrow=.9] (d.west) to[out=0,in=90] (.5,-.5) to[out=-90,in=-90,looseness=2] +(.5,0) to +(0,3.7);
    \end{pic}  
  \]
  After composing with $i$ and $i^{\dag}$, speciality is equivalent to:
  
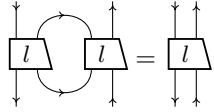
\[
    \begin{pic}[scale=.5]
      \node[morphism] (l) at (0,0) {$l$};
      \node[morphism] (r) at (2,0) {$l$};
      \draw[arrow=.5] ([xshift=2mm]l.north east) to[out=90,in=90,looseness=1.5] ([xshift=-2mm]r.north west);
      \draw[arrow=.5] ([xshift=2mm]l.south east) to[out=-90,in=-90,looseness=1.5] ([xshift=-2mm]r.south west);    
      \draw[arrow=.9] ([xshift=-2mm]l.south west) to +(0,-1);
      \draw[reverse arrow=.9] ([xshift=-2mm]l.north west) to +(0,1);
      \draw[reverse arrow=.9] ([xshift=2mm]r.south east) to +(0,-1);
      \draw[arrow=.9] ([xshift=2mm]r.north east) to +(0,1);
    \end{pic}
    =
    \begin{pic}[scale=.5]
      \node[morphism] (l) at (0,0) {$l$};
      \draw[arrow=.9] ([xshift=-2mm]l.south west) to +(0,-1);
      \draw[reverse arrow=.9] ([xshift=-2mm]l.north west) to +(0,1);
      \draw[reverse arrow=.9] ([xshift=2mm]l.south east) to +(0,-1);
      \draw[arrow=.9] ([xshift=2mm]l.north east) to +(0,1);
    \end{pic}
  \]
  That is, speciality is equivalent to $\tinymult[whitedot]$ being right idempotent. 
\end{proof}

A dual construction holds for right idempotent 3-structures with suitable splittings.

\begin{example} \label{ex:splitting_doubleconstruction} 
Suppose $\catC$ is exact, and consider a connector $(A, R, S, p)$ viewed as a 3-structure in $\Rel(\catC)$, recalling the equivalence relations $l$, $r$ from Definition~\ref{def:double_eq}. The 2-structure~\eqref{eq:split_construction} is the groupoid in $\catC$ with object of morphisms being the object $A_l$ of $l$-equivalence classes, \ie~$[y,z]_l = [x,w]_l$ whenever $p(x,y,z) = w$. Composition is given by \[
[x,y]_l \circ [y,z]_l = [x,z]_l 
\] 
with identities $[x,x]_l$ for all $x \in A$. There is a groupoid defined in terms of $r$ dually. By Proposition~\ref{prop:double_eq_rel} the same construction holds for any double equivalence relation satisfying~\eqref{eq:associativity:ternary}. 
\end{example}

\begin{example} \label{example:special_2struc_rep}
  For any symmetric dagger special Frobenius 2-structure $(A, \tinymult)$, the unital 3-structure $\tinymulttriple$ constructed in Theorem~\ref{thm:2struct_equal_unit3struc} is left idempotent, with $l_A$ always having a splitting:
  \begin{equation} \label{eq:spec_2struc_splitting}
    \begin{pic}[scale=.5]
      \node[dot] (d) at (0,0) {};
      \draw[arrow=.9] (d) to +(0,1);
      \draw[reverse arrow=.9] (d) to[out=0,in=90] +(.8,-1);
      \draw[reverse arrow=.9] (d) to[out=180,in=90] +(-.5,-.3) to[out=-90,in=-90,looseness=2] +(-.5,0) to +(0,1.3);
    \end{pic}
  \end{equation}
  The 2-structure constructed via Theorem~\ref{thm:splitting} is then precisely $\tinymult$.
%   Then the associated Frobenius 3-structure $\tinymulttriple$ on $A$ is that of Theorem~\ref{thm:2struct_equal_unit3struc}, and the representation is unital. 
%   is special if and only if the following is a representation of $\tinymult$:
% %  $(\tinyinvolution \otimes \id[A]) \circ \tinymultflip \colon A \to A^* \otimes A$ is a representation of $\tinymult$.
%   \[
%     \begin{pic}[scale=.5]
%       \node[dot] (d) at (0,0) {};
%       \draw[arrow=.9] (d) to +(0,1);
%       \draw[reverse arrow=.9] (d) to[out=0,in=90] +(.8,-1);
%       \draw[reverse arrow=.9] (d) to[out=180,in=90] +(-.5,-.3) to[out=-90,in=-90,looseness=2] +(-.5,0) to +(0,1.3);
%     \end{pic}
%   \]
%   Then the associated Frobenius 3-structure $\tinymulttriple$ on $A$ is that of Theorem~\ref{thm:2struct_equal_unit3struc}, and the representation is unital.
\end{example}

\begin{remark}[Representable Structures]
We call a dagger Frobenius 2-structure $(L, \tinymult[whitedot])$ \emph{representable} when it arises from the construction of Theorem~\ref{thm:splitting}. Equivalently, a 2-structure is representable when it comes with an isometry $i \colon L \to A^* \otimes A$, for some $A$, which is a morphism of 2-structures from $\tinymult[whitedot]$ into the canonical 2-structure on $A^* \otimes A$ from Example~\ref{ex:pants:binary}.

Indeed, in either case $\tinymulttriple$ is defined by bending one leg of $l = l_A = i \circ i^{\dag}$. Then preservation of the involution and multiplication by $i$ are equivalent to~\eqref{eq:frob_idem_useful} and that $l=l_*$, which are in turn equivalent to $\tinymulttriple$ being dagger symmetric and associative, while it is left idempotent by construction.

%Indeed, if $\tinymult$ arises from Theorem~\ref{thm:splitting} then~\eqref{eq:frob_idem_useful} ensures that $i$ preserves composition, and $l$ satisfies $l = l_*$, ensuring that $i$ preserves the involution. In the other direction, define $\tinymulttriple$ by bending one leg of $l = i \circ i^{\dag}$, giving left idempotence by construction. Preservation of the involution and multiplication again ensure that $l=l_*$ and~\eqref{eq:frob_idem_useful} hold, yielding dagger symmetry and associativity. 

Example~\ref{example:special_2struc_rep} says that any symmetric dagger special 2-structure $\tinymult$ has a representation given by~\eqref{eq:spec_2struc_splitting}. Any symmetric, transitive relation $r \colon R \rightarrowtail A \times A$ in a regular category $\catC$ defines a 2-structure with representation $r$. 
\end{remark}

\begin{remark}\label{rem:functoriality}
  In some sense Theorems~\ref{thm:2struct_equal_unit3struc} and~\ref{thm:splitting} are converses to each other.
  Let the construction $\tinymulttriple \mapsto (L, \tinymult[whitedot])$ of Theorem~\ref{thm:splitting} act on morphisms $f$ as $i'^\dag \circ (f_* \otimes f) \circ i \colon L \to L'$.
  This defines a functor, to $\B$, from the full subcategory of $\T$ of left idempotent dagger Frobenius 3-structures whose left idempotent dagger splits. 
  %Call a morphism $f$ of 2-structures or unital 3-structures \emph{unital} when it satisfies $f \circ \tinyunit = \tinyunit[whitedot]$. 
  Then it restricts to two subcategories (omitting the word `dagger' throughout):
  \begin{center}\begin{tabular}{lcl}
    & $L$ \\
    left and right idempotent 3-structures & $\to$ & special 2-structures \\
  %  unital normal 3-structures & $\to$ special 2-structures \\
    unital normal 3-structures & $\to$ & special 2-structures, unital 
  \end{tabular}\end{center}
  In the second case, we mean that we take unital morphisms on both sides.
%  (Where we take maps of unital 3-structures to be unital).
%(With unital morphisms on both sides in the second case.)
 % (With unital morphisms in the last two cases.)
  Similarly, the functor $U \colon \B \to \T$ from Theorem~\ref{thm:2struct_equal_unit3struc} restricts to three subcategories:
  \begin{center}\begin{tabular}{lcl}
    & $U$ \\
    special 2-structures & $\to$ & normal 3-structures \\
    2-structures, unital  & $\to$ & unital 3-structures \\
    special 2-structures, unital & $\to$ & unital normal 3-structures
  \end{tabular}\end{center}
  In the latter two cases $U$ is an isomorphism of categories; to see this note that a unital morphism of 3-structures is also one of 2-structures.
  In the final case $U$ and $L$ also form an equivalence of categories; to see this use the isometric splitting of Example~\ref{example:special_2struc_rep}, and note that any other choice of splitting gives an isomorphic 2-structure.
\end{remark}

\subsection{Enveloping Structure}

In a dagger category, a \emph{dagger biproduct} of objects $A$, $B$ is a biproduct $(A \oplus B, \coproj_A, \coproj_B, \pi_A, \pi_B)$ whose structure maps satisfy $\pi_A = \coproj_A^{\dag}$, $\pi_B = \coproj_B^{\dag}$. In this section we assume that $\catD$ is a dagger monoidal, with dagger biproducts which are distributive, meaning that the canonical map $u \colon (A \otimes B) \oplus (A \otimes C) \to A \otimes (B \oplus C)$ is an isomorphism with $u^{-1} = u^{\dag}$. Such biproducts make $\catD$ enriched in commutative monoids, with the addition of morphisms satisfying $f^{\dag} + g^{\dag} = (f + g)^{\dag}$. 

%We will also assume that all dagger idempotents in $\catD$ have dagger splittings.

\begin{example} \label{ex:coherent}
When $\catC$ is a regular category which is also \emph{coherent}, each lattice of subobjects $\text{Sub}(A)$ comes with unions. If $\catC$ is moreover \emph{positive}, it has coproducts which form distributive dagger biproducts in $\Rel(\catC)$. In this case, the addition in $\Rel(\catC)$ is given by the union of relations, and so is idempotent, \ie~satisfies $R + R = R$ for all $R$. 
\end{example}

A \emph{sub-3-structure} of a dagger Frobenius 2-structure $(B, \tinymult[whitedot])$ in $\catD$ consists of a dagger 3-structure $(A, \tinymulttriple)$ along with an isometry $i \colon A \to B$ which is a morphism of 3-structures $\tinymulttriple \to \tinymulttriple[whitedot]$. 

%Let $\BinSub(\catD)$ be the category whose objects are dagger symmetric Frobenius 2-structures $(B, \tinymult[whitedot])$ together with a chosen normal sub-3-structure $i \colon A \to B$, for which we additionally have $\tinymult[whitedot] \circ (i \otimes i) = 0$. Morphisms $f$ are those of the 2-structures which further $f \circ i = i' \circ i'^{\dag} \circ f \circ i$. There is an evident functor $S \colon \BinSub(\catD) \to \TernCatNorm(\catD)$ which picks out the sub-3-structure, acting on morphisms by $f \mapsto i'^\dag \circ f \circ i$.

\begin{example}\label{ex:coset}
  Sub-3-structures $A$ of a group $(G,\tinymult[whitedot])$ in $\Rel(\cat{Set})$ correspond to cosets of subgroups $H \subseteq G$. On the one hand, if $H \subseteq G$ is a subgroup and $g \in G$, then the coset $A=gH=\{gh \mid h \in H\}$ is a sub-3-structure with operation $(a,b,c) \mapsto ab^{-1}c$.
  On the other hand, for any sub-3-structure $i$ of $G$, the range $A \subseteq G$ of $i$ is closed under this operation, making $H=A^{-1}A$ a subgroup of $G$ with $A=aH$ for any $a \in A$.
  % $H=A^{-1}A$ is a subgroup of $G$ thanks to closure of $A$ under $ab^{-1}c$.
  % %because $(a^{-1}b)(c^{-1})=a^{-1}(bc^{-1}d) \in A^{-1}A$ for $a,b,c,d \in A$ and $\id[G]=a^{-1}a$ for any $a \in A$.
  % Also $A=aH$ for any $a \in A$: if $b \in A$ then $b=ab^{-1}b \in aH$ so $A \subseteq aH$, and if $a(b^{-1}c) \in aH$ then also $ab^{-1}c \in A$, so that $aH \subseteq A$.
  % Thus $A=aH$ is a coset.
\end{example}

Write $\BinSub(\catD)$ for the category with objects sub-3-structures $(\tinymulttriple, i \colon A \to B,\tinymult[whitedot])$ that moreover are normal and satisfy $\tinymult[whitedot] \circ (i \otimes i) = 0$. Morphisms $f$ are those of 2-structures which further satisfy $f \circ i = i' \circ i'^{\dag} \circ f \circ i$. There is an evident functor $S \colon \BinSub(\catD) \to \TernCatNorm(\catD)$ which picks out the sub-3-structure, acting on morphisms by $f \mapsto i'^\dag \circ f \circ i$.

\begin{theorem}[Enveloping Structure] \label{thm:Envelope_struc_and_UP} 
  Let a monoidal dagger category $\catD$ have distributive dagger biproducts and dagger splittings for all dagger idempotents. Then the functor $S \colon \BinSub(\catD) \to \TernCatNorm(\catD)$ has a left adjoint $\tinymulttriple \mapsto E(\tinymulttriple)$, with the unit of the adjunction being the identity. Moreover, when addition in $\catD$ is idempotent, $E(\tinymulttriple)$ is always special.
\end{theorem}

\begin{proof}
Let $i_L \colon L \to A^* \otimes A$ and $i_R \colon R \to A \otimes A^*$ be dagger splittings for the left and right idempotents of $\tinymulttriple$. Write $B = (A^* \otimes A) \oplus (A \otimes A^*) \oplus A^* \oplus A$, and $E = L \oplus R \oplus A^* \oplus A$, with $e = i_L \oplus i_R \oplus \id \oplus \id \colon E \to B$ the obvious isometry. We define $E(\tinymulttriple) = (E, \tinymult[blackdot])$ as follows. Firstly, the unit $\tinyunit[blackdot]$ is given by $e \circ \tinyunit[blackdot] = \tinycup \oplus \tinycupswap$; since $e$ is an isometry it is split monic and so this determines it uniquely. 

The multiplication $\tinymult[blackdot]$ on $E$ may equivalently be defined in terms of the morphism $e \circ \tinymult[blackdot] \circ (e^{\dag} \otimes e^{\dag}) \colon B \otimes B \to B$, again since $e$ is split monic. Using distributivity of the biproducts, we define the latter to be the $\oplus$ sum of all the canonical morphisms definable in terms of $\tinymulttriple$, depicted in the following table:
\begin{center} \label{eq:table}
\begin{tabular}{crcccc} %NOTE: Can add labels 'Left Input' and 'Right Input' by uncommenting below if desired
%\multicolumn{1}{l}{}      & & %\multicolumn{4}{c}{Right input}                                                                                                                   \\
  & \multicolumn{1}{r|}{$\otimes$}       & $A^* \otimes A$ & $A \otimes A^*$ & $A^*$ & $A$ \\ \cline{2-6}
  \multirow{4}{*}{\pbox{20cm}{}}%{Left \\ input}}
                            & \multicolumn{1}{r|}{$A^* \otimes A$} & $\tablemultdownup$                     & $0$                                 & $\tablemultdual$       & $0$                     \\ \
                            & \multicolumn{1}{r|}{$A \otimes A^*$} & $0$                                 & $\tablemultupdown$                     & $0$                       & 
                            $\tablemult$       \\
                            & \multicolumn{1}{r|}{$A^*$}           & $0$                                   & $\tablemultdual$                 & 0                         & $\tablelidem$                     \\
                            & \multicolumn{1}{r|}{$A$}             & $\tablemult$                   & $0$                                   & $\tableridem$                       & $0$                      
\end{tabular}
\end{center}

Verifying that $E(\tinymulttriple)$ is well-defined is tedious but straightforward using Corollary~\ref{cor:normalform_consequences}. Unitality follows from the fact that we restrict from $B$ to $E$. The object $E$ is self-dual, with each side of~\eqref{eq:symmetric:binary} seen to be the identity, giving symmetry. Associativity and the Frobenius law~\eqref{eq:frobenius:binary} each follow from those for $\tinymulttriple$ and symmetry of the definition. By `counting paths' one may verify that the morphism $p = \tinymultflip[blackdot] \circ \tinymult[blackdot]$ satisfies $p \circ p = p + p$. Hence $p$ is idempotent whenever addition is, making $E(\tinymulttriple)$ special. 
% % For the final statement, one may verify that the morphism $f = \tinymult[whitedot] \circ \tinymult[whitedot]$ has $f = g + g^*$ where $g$ is the restriction along $e$ of the morphism:
% % \[
% % \includegraphics[scale=0.1]{figures/Final_22.png}
% % \]
% % and using this that $f \circ f = f + f$. Hence if addition is idempotent so is $f$, making $\tinymult[whitedot]$ special.
% For a normal 3-structure $(A, \tinymulttriple)$, let $E(\tinymulttriple)$ be the enveloping 2-structure on $E$ with $\tinymulttriple$ as a substructure via $\coproj_A \colon A \to E$, and recall the isometry $e \colon E \to B$, all as in Theorem~\ref{thm:Envelope_struc}. Then indeed $S(E(\tinymulttriple)) = \tinymulttriple$; we now verify the universal property.

In general, the coprojection $\coproj_A \colon A \to E$ indeed makes $\tinymulttriple$ a sub-3-structure of $E(\tinymulttriple)$, with $S(E(\tinymulttriple)) = \tinymulttriple$; we now verify the universal property.
Let $(C, i, \tinymult[whitedot])$ be an object of $\BinSub(\catD)$ and let $h \colon \tinymulttriple \to S(\tinymult[whitedot])$ be a morphism of 3-structures, so that $g = i \circ h$ is a morphism of 3-structures $\tinymulttriple \to \tinymulttriple[whitedot]$. We need to show that $h = S(f) = i^{\dag} \circ f \circ \coproj_A$ for a unique $f \colon E(\tinymulttriple) \to \tinymult[whitedot]$ in $\BinSub(\catD)$. Now, since any such $f$ preserves the multiplication and involution, one may check that:% it has $f \circ e^{\dag}$ equal to:   
 \[
 f \circ e^{\dag} = 
 \Big[
 \quad
  \begin{pic}[scale=.5]
    \node[whitedot] (d) at (0,3) {};
    \node[whitedot] (i) at (-1,2){};
    \node[morphism] (l) at (-1,1) {$g_*$};
    \node[morphism] (r) at (1,1) {$g$};
    \draw[arrow=.9] (l.south) to +(0,-1) node[below]{$A$};
    \draw[arrow=.7] (i.south) to (l.north);
    \draw[arrow=.5] (i.north) to[out=90,in=180] (d.west);
    \draw[arrow=.9] (d.north) to +(0,1) node[above]{$C$};
    \draw[reverse arrow=.9] (r.south) to +(0,-1) node[below]{$A$};
    \draw[arrow=.5] (r.north) to[out=90,in=0] (d.east);
  \end{pic}
  ,
  \qquad
    \begin{pic}[xscale=-.5,yscale=.5]
    \node[whitedot] (d) at (0,3) {};
    \node[whitedot] (i) at (-1,2){};
    \node[morphism] (l) at (-1,1) {$g_*$};
    \node[morphism] (r) at (1,1) {$g$};
    \draw[arrow=.9] (l.south) to +(0,-1) node[below]{$A$};
    \draw[arrow=.7] (i.south) to (l.north);
    \draw[arrow=.5] (i.north) to[out=90,in=180] (d.east);
    \draw[arrow=.9] (d.north) to +(0,1) node[above]{$C$};
    \draw[reverse arrow=.9] (r.south) to +(0,-1) node[below]{$A$};
    \draw[arrow=.5] (r.north) to[out=90,in=0] (d.west);
  \end{pic}
  ,
  \qquad
  \begin{pic}[scale=.5]
    \node[whitedot] (i) at (-1,3){};
    \node[morphism] (l) at (-1,1) {$g_*$};
    \draw[arrow=.9] (l.south) to +(0,-1) node[below]{$A$};
    \draw[arrow=.7] (i.south) to (l.north);
    \draw[arrow=.9] (i.north) to +(0,1) node[above]{$C$};
  \end{pic}
  ,
  \qquad
  \begin{pic}[scale=.5]
    \node[morphism] (l) at (-1,1) {$g$};
    \draw[reverse arrow=.9] (l.south) to +(0,-1) node[below]{$A$};
    \draw[arrow=.9] (l.north) to (-1,4) node[above]{$C$};
  \end{pic}
  \quad
  \Big]
  \]
Conversely, let us define $f$ in this way. By construction $f \circ \coproj_A = i \circ h = i \circ i^{\dag} \circ f \circ \coproj_A$ as required, and $f$ preserves the involution. To see that $f$ preserves multiplication is tedious but straightforward, after noting that $\tinymult[whitedot] \circ (g \otimes g) = \tinymult[whitedot] \circ ((i \circ h) \otimes (i \circ h)) = 0$
and:

  \begin{align*}
    \begin{pic}[xscale=-.5,yscale=.5]
      \node[whitedot] (d) at (0,3) {};
      \node[whitedot] (i) at (-1,2){};
      \node[morphism] (l) at (-1,1) {$g_*$};
      \node[morphism] (r) at (1,1) {$g$};
      \draw[arrow=.9] (l.south) to +(0,-1.3);
      \draw[arrow=.7] (i.south) to (l.north);
      \draw[arrow=.5] (i.north) to[out=90,in=180] (d.east);
      \draw[arrow=.9] (d.north) to +(0,1);
      \draw[reverse arrow=.9] (r.south) to +(0,-1.3);
      \draw[arrow=.5] (r.north) to[out=90,in=0] (d.west);
    \end{pic}
    & = 
    \begin{pic}[xscale=-.5,yscale=.5]
      \node[whitedot] (d) at (0,3) {};
      \node[whitedot] (i) at (-1,2){};
      \node[dot] (b) at (1,0) {};
      \node[morphism] (l) at (-1,1) {$g_*$};
      \node[morphism] (r) at (1,1) {$g$};
      \draw[arrow=.9] (l.south) to +(0,-1.3);
      \draw[arrow=.7] (i.south) to (l.north);
      \draw[arrow=.5] (i.north) to[out=90,in=180] (d.east);
      \draw[arrow=.9] (d.north) to +(0,1);
      \draw[arrow=.5] (r.north) to[out=90,in=0] (d.west);
      \draw[reverse arrow=.5] (r.south) to (b.north);
      \draw[reverse arrow=.9] (b.south) to +(0,-.5);
      \draw (b.140) to[out=140,in=90,looseness=2] ([xshift=3    mm,yshift=.5mm]b.west) to ([xshift=3mm,yshift=-.5mm]b.west) to[out=-90,in=-140,looseness=2] (b.-140);
    \end{pic}
    = 
    \begin{pic}[xscale=-.5,yscale=.5]
      \node[whitedot] (t) at (0,3.5) {};
      \node[whitedot] (r) at (-1,2.5){};
      \node[whitedot] (l) at (2.6,2.5){};
      \node[morphism] (a) at (4.25,1) {$g$};
      \node[morphism] (b) at (2.6,1) {$g_*$};
      \node[morphism] (c) at (-1,1) {$g_*$};
      \node[morphism] (d) at (1,1) {$g$};
      \draw[arrow=.9] (c.south) to +(0,-1.3);
      \draw[arrow=.7] (r.south) to (c.north);
      \draw[arrow=.5] (r.north) to[out=90,in=180] (t.east);
      \draw[arrow=.9] (t.north) to +(0,.5);
      \draw[reverse arrow=.5] (t.west) to[out=0,in=90,looseness=.75] (l.north);
      \draw[arrow=.5] (l.south) to (b.north);
      \draw[reverse arrow=.5] (l.west) to[out=0,in=90] (a.north);
      \draw[reverse arrow=.5] (l.east) to[out=180,in=90] (d.north);
      \draw[reverse arrow=.5] (a.south) to[out=-90,in=-90,looseness=1.5] (b.south);
      \draw[reverse arrow=.9] (d.south) to +(0,-1.3);
    \end{pic}
    \\ & =
    \begin{pic}[scale=.5]
      \node[morphism] (a) at (0,0) {$g$};
      \node[morphism] (b) at (1.75,0) {$g_*$};
      \node[morphism] (c) at (3.5,0) {$g$};
      \node[morphism] (d) at (5.25,0) {$g_*$};
      \node[whitedot] (l) at (0,1) {};
      \node[whitedot] (r) at (1,2) {};
      \draw[reverse arrow=.5] (a.south) to[out=-90,in=-90,looseness=1.5] (b.south);
      \draw[reverse arrow=.9] (c.south) to +(0,-1);
      \draw[arrow=.9] (d.south) to +(0,-1);
      \draw[arrow=.7] (a.north) to (l.south);
      \draw[arrow=.5] (l.east) to[out=0,in=-90] (r.south);
      \draw[arrow=.8] (r.east) to[out=0,in=-90] +(.2,.2) to[out=90,in=90] ([yshift=18mm]b.north) to (b.north);
      \draw[reverse arrow=.9] (r.north) to[out=90,in=90] ([yshift=16mm]c.north) to (c.north);
      \draw[arrow=.95] (r.west) to[out=180,in=-90] +(-.2,.2) to[out=90,in=90] ([yshift=16mm]d.north) to (d.north);
      \draw[arrow=.9] (l.west) to[out=180,in=-90,looseness=.7] +(-.5,2.5);
    \end{pic}
    =
    \begin{pic}[scale=.5]
      \node[morphism] (a) at (0,0) {$g$};
      \node[morphism,hflip] (b) at (1.5,2) {$g$};
      \node[whitedot] (l) at (0,1) {};
      \node[dot] (r) at (1.5,3) {};
      \draw[arrow=.7] (a.north) to (l.south);
      \draw[arrow=.5] (l.east) to[out=0,in=-90,looseness=.8] (b.south);
      \draw[arrow=.7] (b.north) to (r.south);
      \draw[arrow=.5] (r.east) to[out=0,in=-90] +(.2,.2) to[out=90,in=90] +(.5,0) to +(0,-3.5) to[out=-90,in=-90] (a.south);
      \draw[reverse arrow=.95] (r.north) to[out=90,in=90] +(1.5,0) to +(0,-4.5);
      \draw[arrow=.97] (r.west) to[out=180,in=-90] +(-.2,.2) to[out=90,in=90] +(2.5,0) to +(0,-4.5); 
      \draw[arrow=.9] (l.west) to[out=180,in=-90,looseness=.5] +(-.5,3);
    \end{pic}
    =
    \begin{pic}[xscale=-.5,yscale=.5]
      \node[whitedot] (d) at (0,3) {};
      \node[whitedot] (i) at (-1,2){};
      \node[morphism] (l) at (-1,1) {$g_*$};
      \node[morphism] (r) at (1,1) {$g$};
      \draw[arrow=.7] (i.south) to (l.north);
      \draw[arrow=.5] (i.north) to[out=90,in=180] (d.east);
      \draw[arrow=.9] (d.north) to +(0,1);
      \draw[arrow=.5] (r.north) to[out=90,in=0] (d.west);
      \node[dot] (b) at (1,0) {};
      \draw[arrow=.7] (b.north) to (r.south);
      \draw[arrow=.9] (b.south) to +(0,-.8);
      \draw[reverse arrow=.9] (b.west) to[out=0,in=90] +(.5,-1);
      \draw[reverse arrow=.8] (b.east) to[out=180,in=90] +(-.3,-.3) to[out=-90,in=-90] ([yshift=-9mm]l.south) to (l.south);
    \end{pic}
  \end{align*}
  This finishes the proof.
\end{proof}

\begin{example} \label{ex:envel_groupoid}
For a normal Frobenius 3-structure $\tinymulttriple$ in $\Rel(\cat{Set})$, corresponding to a connector in $\cat{Set}$, the construction of $E(\tinymulttriple)$ is studied in detail by Kock in~\cite{kock2007principal} under the name of the \emph{enveloping groupoid}. In particular~\eqref{eq:table} appears as table (7) in~\cite{kock2007principal}. The same construction makes sense in $\Rel(\catC)$ whenever $\catC$ is coherent, and Kock shows that the above construction defines a left adjoint to the forgetful functor $\Gpd(\catC) \to \Conn(\catC)$ (at the level of $\catC$, rather than $\Rel(\catC)$).
\end{example}

%Linking Algebra Notes: 
% Kaur's Thesis studies relation between a TRO and its linking C*-algebra.

% Kaur's main thesis results on the TROs and their linking C*-algebra is in \cite{kaur2002local}. Doesn't seem like the first place this was discovered though! But possibly the neatest and most focused on the topic. In fact \cite{effros2001injectivity} describes the linking algebra very neatly at the start, and is earlier. Again not sure if its the first place however.

%Another classification result for ternary rings: http://msp.org/pjm/2003/209-2/pjm-v209-n2-p10-s.pdf

\begin{example} \label{ex:linkingC*algebra}
Any TRO $V \subseteq \boundedops(H, K)$ embeds as a substructure of its \emph{linking C*-algebra}~\cite{kaur:ternary}, the closure in $\boundedops(H \oplus K)$ of the *-algebra
\[
\begin{bmatrix}
    V^*V       & V^* \\
    V       & VV^*
\end{bmatrix}
\]
where $V^*V = \{a^*b \mid a, b \in V\}$ and $VV^* = \{ab^* \mid a, b \in V\}$. In the finite-dimensional setting, $V$ is of the form $\bigoplus^n_{i=1} \boundedops(H_i, K_i)$, forming a 3-structure $\tinymulttriple$ in $\FHilb$, and this construction is formally identical to that of Theorem~\ref{thm:Envelope_struc_and_UP}. Note, however, there is some subtlety; $\tinymulttriple$ is not strictly even left and right idempotent, and instead the idempotents corresponding to $V^*V = \boundedops(\bigoplus^n_{i=1} H_i)$ and $VV^* = \boundedops(\bigoplus^n_{i=1} K_i)$ should be used in the construction.
\end{example}

\bibliographystyle{plain}
\bibliography{monoidalalgebra}

\end{document}